\documentclass[11pt,a4paper]{amsart}
\pdfoutput=1
\usepackage[centering, margin=1in]{geometry}
\usepackage{amsmath,amsfonts,amsthm,amssymb,enumitem,mathrsfs,mathtools}
\usepackage[stretch=10]{microtype}
\usepackage[unicode]{hyperref}
\usepackage{xcolor}
\definecolor{dark-red}{rgb}{0.4,0.15,0.15}
\definecolor{dark-blue}{rgb}{0.15,0.15,0.4}
\definecolor{medium-blue}{rgb}{0,0,0.5}
\hypersetup{
	pdftitle={Strong Hybrid Subconvexity for Twisted Selfdual \texorpdfstring{$\GL_3$}{GL\textthreeinferior} \texorpdfstring{$L$}{L}-Functions},
	pdfauthor={Soumendra Ganguly, Peter Humphries, Yongxiao Lin, and Ramon Nunes},
	pdfnewwindow=true,
	colorlinks, linkcolor={dark-red},
	citecolor={dark-blue}, urlcolor={medium-blue}
}
\allowdisplaybreaks

\newcommand{\BB}{\mathcal{B}}
\newcommand{\C}{\mathbb{C}}
\newcommand{\CC}{\mathcal{C}}
\newcommand{\dee}{\partial}
\newcommand{\Dscr}{\mathscr{D}}
\newcommand{\e}{\varepsilon}
\newcommand{\Fscr}{\mathscr{F}}
\newcommand{\Gscr}{\mathscr{G}}
\newcommand{\Hb}{\mathbb{H}}
\newcommand{\HH}{\mathcal{H}}
\newcommand{\hol}{\mathrm{hol}}
\newcommand{\II}{\mathcal{I}}
\newcommand{\JJ}{\mathcal{J}}

\newcommand{\Kscr}{\mathscr{K}}
\newcommand{\N}{\mathbb{N}}

\newcommand{\R}{\mathbb{R}}

\newcommand{\VV}{\mathcal{V}}
\newcommand{\Z}{\mathbb{Z}}
\newcommand{\ZZ}{\mathcal{Z}}
\DeclareMathOperator{\ad}{ad}
\DeclareMathOperator{\arsinh}{arsinh}
\DeclareMathOperator{\artanh}{artanh}
\DeclareMathOperator{\cosech}{cosech}
\DeclareMathOperator{\GL}{GL}
\DeclareMathOperator{\PGL}{PGL}
\DeclareMathOperator*{\Res}{Res}
\DeclareMathOperator{\sech}{sech}
\DeclareMathOperator{\sgn}{sgn}
\DeclareMathOperator{\SL}{SL}
\DeclareMathOperator{\vol}{vol}
\numberwithin{equation}{section}
\newtheorem{theorem}[equation]{Theorem}

\newtheorem{corollary}[equation]{Corollary}

\newtheorem{lemma}[equation]{Lemma}
\newtheorem{proposition}[equation]{Proposition}
\theoremstyle{remark}
\newtheorem{remark}[equation]{Remark}
\theoremstyle{definition}

\begin{document}

\title{Strong Hybrid Subconvexity for Twisted Selfdual $\GL_3$ $L$-Functions}

\author{Soumendra Ganguly}

\address{Department of Mathematics, Aarhus University, Ny Munkegade 118, 8000 Aarhus C, Denmark}

\email{\href{mailto:sg@math.au.dk}{sg@math.au.dk}}

\urladdr{\href{https://soumendraganguly.com/}{https://soumendraganguly.com/}}

\author{Peter Humphries}

\address{Department of Mathematics, University of Virginia, Charlottesville, VA 22904, USA}

\email{\href{mailto:pclhumphries@gmail.com}{pclhumphries@gmail.com}}

\urladdr{\href{https://sites.google.com/view/peterhumphries/}{https://sites.google.com/view/peterhumphries/}}

\author{Yongxiao Lin}

\address{Data Science Institute, Shandong University, Jinan 250100, China; State Key Laboratory of Cryptography and Digital Economy Security, Shandong University, Jinan 250100, China}

\email{\href{mailto:yongxiao.lin@sdu.edu.cn}{yongxiao.lin@sdu.edu.cn}}

\urladdr{\href{https://faculty.sdu.edu.cn/linyongxiao/}{https://faculty.sdu.edu.cn/linyongxiao/}}

\author{Ramon Nunes}

\address{Departamento de Matem\'{a}tica, Universidade Federal do Cear\'{a}, Campus do Pici, Bloco 914, 60440-900 Fortaleza-CE, Brasil}

\email{\href{mailto:ramonmnunes@gmail.com}{ramonmnunes@gmail.com}}

\urladdr{\href{https://www.ramonmn.com/}{https://www.ramonmn.com/}}


\thanks{The second author was supported by the National Science Foundation (grant DMS-2302079) and by the Simons Foundation (award 965056). The third author was partially supported by the National Key R\&D Program of China (No. 2021YFA1000700). The fourth author was supported by Instituto Serrapilheira (grant no. 8277).}

\begin{abstract}
We prove strong hybrid subconvex bounds simultaneously in the $q$ and $t$ aspects for $L$-functions of selfdual $\GL_3$ cusp forms twisted by primitive Dirichlet characters. We additionally prove analogous hybrid subconvex bounds for central values of certain $\GL_3 \times \GL_2$ Rankin--Selberg $L$-functions. The subconvex bounds that we obtain are strong in the sense that, modulo current knowledge on estimates for the second moment of $\GL_3$ $L$-functions, they are the natural limit of the first moment method pioneered by Li and by Blomer.

The method of proof relies on an explicit $\GL_3 \times \GL_2 \leftrightsquigarrow \GL_4 \times \GL_1$ spectral reciprocity formula, which relates a $\GL_2$ moment of $\GL_3 \times \GL_2$ Rankin--Selberg $L$-functions to a $\GL_1$ moment of $\GL_4 \times \GL_1$ Rankin--Selberg $L$-functions. A key additional input is a Lindel\"{o}f-on-average upper bound for the second moment of Dirichlet $L$-functions restricted to a coset, which is of independent interest.
\end{abstract}

\maketitle

\section{Introduction}

\subsection{Hybrid Subconvex Bounds for \texorpdfstring{$L$}{L}-Functions}

The main result of this paper concerns strong subconvex bounds for certain degree $3$ and degree $6$ $L$-functions involving selfdual Hecke--Maa\ss{} cusp forms for $\SL_3(\Z)$, with an emphasis on bounds that are uniform in several aspects simultaneously.

\begin{theorem}
\label{thm:subconvexbounds}
Let $F$ be a fixed selfdual Hecke--Maa\ss{} cusp form for $\SL_3(\Z)$ and let $\chi$ be a primitive Dirichlet character of conductor $q$, where $q$ is an arbitrary positive integer. Let $q_1$ be a divisor of $q$ for which $(q_1,\frac{q}{q_1}) = 1$.
\begin{enumerate}[leftmargin=*,label=\textup{(\arabic*)}]
\item\label{thmno:subconvexbounds1} For $t \in \R$, we have that
\[L\left(\frac{1}{2} + it, F \otimes \chi\right) \ll_{F,\e} (q(|t| + 1))^{\frac{3}{5} + \e} \left(1 + \frac{q^{2/5}}{q_1^{1/2} (|t| + 1)^{1/10}} + \frac{q_1^{1/8}}{(q(|t| + 1))^{1/10}}\right).\]
\item\label{thmno:subconvexbounds2} Let $f$ be a Hecke--Maa\ss{} newform of weight $0$, level $q^2$, principal nebentypus, and Laplacian eigenvalue $\frac{1}{4} + t_f^2$, and suppose that $f \otimes \overline{\chi}$ has level dividing $q$. We have that
\[L\left(\frac{1}{2},F \otimes f\right) \ll_{F,\e} (q(|t_f| + 1))^{\frac{6}{5} + \e} \left(1 + \frac{q^{4/5}}{q_1 (|t_f| + 1)^{1/5}} + \frac{q_1^{1/4}}{(q(|t_f| + 1))^{1/5}}\right).\]
\item\label{thmno:subconvexbounds3} Let $f$ be a holomorphic Hecke newform of even weight $k_f$, level $q^2$, and principal nebentypus, and suppose that $f \otimes \overline{\chi}$ has level dividing $q$. We have that
\[L\left(\frac{1}{2},F \otimes f\right) \ll_{F,\e} (q k_f)^{\frac{6}{5} + \e} \left(1 + \frac{q^{4/5}}{q_1 k_f^{1/5}} + \frac{q_1^{1/4}}{(q k_f)^{1/5}}\right).\]
\end{enumerate}
\end{theorem}

These bounds imply hybrid subconvexity simultaneously in the $q$ and $t$, $t_f$, or $k_f$ aspects. Focusing in particular on \hyperref[thm:subconvexbounds]{Theorem \ref*{thm:subconvexbounds}} \ref{thmno:subconvexbounds1}, we note that the convexity bound for $L(1/2 + it,F \otimes \chi)$ is $O_{F,\e}((q(|t| + 1))^{3/4 + \e})$; upon taking $q_1 = q$ (as we are free to do), \hyperref[thm:subconvexbounds]{Theorem \ref*{thm:subconvexbounds}} \ref{thmno:subconvexbounds1} gives the hybrid subconvex bound $O_{F,\e}(q^{5/8 + \e} (|t| + 1)^{3/5 + \e})$. Notably, this extends and improves upon several previous subconvexity results in the literature, such as the seminal works of Li \cite{Li11} and of Blomer \cite{Blo12}, as we discuss further in \hyperref[sect:previous]{Section \ref*{sect:previous}}. Moreover, we obtain a stronger subconvex bound if $q$ has a divisor $q_1$ with $(q_1,\frac{q}{q_1}) = 1$ and $q^{3/4 + \eta} (|t| + 1)^{-1/5} \ll q_1 \ll q^{1 - \eta} (|t| + 1)^{4/5}$ for some $\eta > 0$, which is guaranteed if $q$ is squarefree and $q^{\delta}$-smooth for some $\delta \in (0,1/4)$. The result is strongest when $q$ has a divisor $q_1$ with $(q_1,\frac{q}{q_1}) = 1$ and $q^{4/5} (|t| + 1)^{-1/5} \ll q_1 \ll q^{4/5} (|t| + 1)^{4/5}$, where we obtain the subconvex bound $L(1/2 + it,F \otimes \chi) \ll_{F,\e} (q(|t| + 1))^{3/5 + \e}$.

These subconvex bounds are similar in nature to a classical result of Heath-Brown \cite[Theorem 2]{H-B78} (now superseded by recent work of Petrow and Young \cite{PY20,PY23}), who proves that if $\chi$ is a primitive Dirichlet character modulo $q$ and if $q_1 \mid q$, then
\[L\left(\frac{1}{2} + it,\chi\right) \ll_{\e} (q(|t| + 1))^{\frac{1}{6} + \e} \left(1 + \frac{q^{1/3}}{q_1^{1/2} (|t| + 1)^{1/6}} + \frac{q_1^{1/4}}{q^{1/6} (|t| + 1)^{1/6}}\right).\]
This bound is strongest when $q$ has a divisor $q_1$ with $q^{2/3} (|t| + 1)^{-1/3} \ll q_1 \ll q^{2/3} (|t| + 1|)^{2/3}$, where it yields the Weyl-strength subconvex bound $L(1/2 + it,\chi) \ll_{\e} (q(|t| + 1))^{1/6 + \e}$. 

The subconvex bounds in \hyperref[thm:subconvexbounds]{Theorem \ref*{thm:subconvexbounds}} follow from bounds for moments of certain $L$-functions together with the nonnegativity of the central $L$-value $L(1/2,F \otimes f)$. To state these moment bounds precisely, we require some notation. We write $q_{\chi}$ for the conductor of a Dirichlet character $\chi$. We let $\BB_{\kappa}^{\ast}(q,\chi)$ denote an orthonormal basis of Hecke--Maa\ss{} newforms of weight $\kappa \in \{0,1\}$, level $q$, and nebentypus $\chi$, where $\chi$ is a primitive Dirichlet character of conductor $q_{\chi} \mid q$; we denote by $t_f$ the spectral parameter of $f \in \BB_{\kappa}^{\ast}(q,\chi)$. Similarly, we let $\BB_{\hol}^{\ast}(q,\chi)$ denote an orthonormal basis of holomorphic Hecke newforms of level $q$ and nebentypus $\chi$; we denote by $k_f$ the weight of $f \in \BB_{\hol}^{\ast}(q,\chi)$.

\begin{theorem}
\label{thm:firstmomentbounds}
Let $F$ be a selfdual Hecke--Maa\ss{} cusp form for $\SL_3(\Z)$. Let $q_1,q_2$ be coprime positive integers. Let $\chi_1$ be a primitive Dirichlet character of conductor $q_1$. Then for $T \geq 1$ and $1 \leq U \leq T$, we have that
\begin{multline}
\label{eqn:firstmomentbounds}
\begin{drcases*}
\sum_{\substack{q' \mid q_1 q_2 \\ q' \equiv 0 \hspace{-.25cm} \pmod{q_{\overline{\chi_1}^2}}}} \sum_{\substack{\psi_1,\psi_2 \hspace{-.25cm} \pmod{q_1 q_2} \\ \psi_1 \psi_2 = \overline{\chi_1}^2 \\ q_{\psi_1} q_{\psi_2} = q'}} \, \int\limits_{T - U \leq |t| \leq T + U} \left|\frac{L\left(\frac{1}{2} + it,F \otimes \psi_1 \chi_1\right)}{L(1 + 2it,\psi_1\overline{\psi_2})}\right|^2 \, dt \\
\sum_{\substack{q' \mid q_1 q_2 \\ q' \equiv 0 \hspace{-.25cm} \pmod{q_{\overline{\chi_1}^2}}}} \sum_{\substack{f \in \BB_0^{\ast}(q',\overline{\chi_1}^2) \\ T - U \leq t_f \leq T + U}} \frac{L\left(\frac{1}{2},F \otimes f \otimes \chi_1\right)}{L(1,\ad f)} \\
\sum_{\substack{q' \mid q_1 q_2 \\ q' \equiv 0 \hspace{-.25cm} \pmod{q_{\overline{\chi_1}^2}}}} \sum_{\substack{f \in \BB_{\hol}^{\ast}(q',\overline{\chi_1}^2) \\ T - U \leq k_f \leq T + U}} \frac{L\left(\frac{1}{2},F \otimes f \otimes \chi_1\right)}{L(1,\ad f)} \\
\end{drcases*}	\\
\ll_{F,\e} q_1 q_2 TU (q_1 q_2 T)^{\e} + \frac{(q_1 T)^{5/4} q_2^{1/2}}{U^{1/4}} (q_1 q_2 T)^{\e}.
\end{multline}
\end{theorem}

The method of proof of \hyperref[thm:firstmomentbounds]{Theorem \ref*{thm:firstmomentbounds}} remains valid, with some alterations, when the selfdual Hecke--Maa\ss{} cusp form $F$ for $\SL_3(\Z)$ is replaced by a minimal parabolic Eisenstein series, and the results are stronger. This has the effect of replacing $L(1/2,F \otimes f \otimes \chi_1)$ with $L(1/2,f \otimes \chi_1)^3$ and of replacing $L(1/2 + it,F \otimes \psi_1\chi_1)$ with $L(1/2 + it,\psi_1\chi_1)^3$ in \eqref{eqn:firstmomentbounds}. We state the analogues of \hyperref[thm:subconvexbounds]{Theorems \ref*{thm:subconvexbounds}} and \ref{thm:firstmomentbounds} in this Eisenstein setting in \hyperref[sect:Eisenstein]{Section \ref*{sect:Eisenstein}} and give a brief explanation of the alterations required in order to prove these analogues.

\subsection{\texorpdfstring{$\mathrm{GL}_3 \times \mathrm{GL}_2 \leftrightsquigarrow \mathrm{GL}_4 \times \mathrm{GL}_1$}{GL\textthreeinferior \textmultiply GL\texttwoinferior \textleftrightarrow GL\textfourinferior \textmultiply GL\textoneinferior} Spectral Reciprocity}

\subsubsection{An Identity of Moments of $L$-Functions}

\hyperref[thm:firstmomentbounds]{Theorem \ref*{thm:firstmomentbounds}} is proven via a spectral reciprocity formula, which is given in \hyperref[thm:central]{Theorem \ref*{thm:central}}. Roughly speaking, we show that given a sufficiently well-behaved tuple of test functions $(h,h^{\hol})$, the $\GL_2$ moment of $\GL_3 \times \GL_2$ Rankin--Selberg $L$-functions
\begin{multline*}
\sum_{\substack{q' \mid q_1 q_2 \\ q' \equiv 0 \hspace{-.25cm} \pmod{q_{\overline{\chi_1}^2}}}} \sum_{f \in \BB_0^{\ast}(q',\overline{\chi_1}^2)} \frac{L\left(\frac{1}{2},F \otimes f \otimes \chi_1\right)}{L(1,\ad f)} h(t_f)	\\
+ \sum_{\substack{q' \mid q_1 q_2 \\ q' \equiv 0 \hspace{-.25cm} \pmod{q_{\overline{\chi_1}^2}}}} \sum_{\substack{\psi_1,\psi_2 \hspace{-.25cm} \pmod{q_1 q_2} \\ \psi_1 \psi_2 = \overline{\chi_1}^2 \\ q_{\psi_1} q_{\psi_2} = q'}} \int_{-\infty}^{\infty} \left|\frac{L\left(\frac{1}{2} + it,F \otimes \psi_1 \chi_1\right)}{L(1 + 2it,\psi_1\overline{\psi_2})}\right|^2 h(t) \, dt	\\
+ \sum_{\substack{q' \mid q_1 q_2 \\ q' \equiv 0 \hspace{-.25cm} \pmod{q_{\overline{\chi_1}^2}}}} \sum_{f \in \BB_{\hol}^{\ast}(q',\overline{\chi_1}^2)} \frac{L\left(\frac{1}{2},F \otimes f \otimes \chi_1\right)}{L(1,\ad f)} h^{\hol}(k_f)
\end{multline*}
is equal to the sum of two main terms and a dual moment. This dual moment is a $\GL_1$ moment (i.e.\ a sum over Dirichlet characters together with an integral over $t \in \R$) of $\GL_4 \times \GL_1$ Rankin--Selberg $L$-functions. These $\GL_4 \times \GL_1$ $L$-functions are imprimitive: they factorise as the product of a $\GL_3 \times \GL_1$ $L$-function and a $\GL_1$ $L$-function. The dual moment roughly takes the form
\[\frac{q_2^{1/2}}{q_1} \sum_{\psi_1 \hspace{-.25cm} \pmod{q_1}} \int_{-\infty}^{\infty} L\left(\frac{1}{2} + it,F \otimes \psi_1\right) L\left(\frac{1}{2} - it,\overline{\psi_1}\right) g(\chi_1,\psi_1) \HH(t) \, dt.\]
Here $g(\chi_1,\psi_1)$ is a certain character sum, studied in \cite{CI00,PY20,PY23,Xi23}, while $\HH(t)$ is a certain transform of the tuple of test functions $(h,h^{\hol})$.

We show that if we choose $h(t)$ to localise to $T - U \leq |t| \leq T + U$ and $h^{\hol}(k)$ to localise to $T - U \leq k \leq T + U$, where they are of size $\approx 1$, then the two main terms are $O_{F,\e} (q_1 q_2 TU (q_1 q_2 T)^{\e})$, while the transform $\HH(t)$ is essentially localised to $|t| \leq T/U$, where it is of size $\approx U$. We also invoke work of Petrow and Young \cite{PY20,PY23} that shows that the character sum $g(\chi_1,\psi_1)$ is of size $O(q_1)$ for \emph{most} characters $\psi_1$ modulo $q_1$. Our remaining manoeuvre is to apply the Cauchy--Schwarz inequality and invoke second moment bounds for the $L$-functions $L(1/2 + it,F \otimes \psi)$ and $L(1/2 - it,\overline{\psi})$ (in the latter case restricted to \emph{cosets} of the group of Dirichlet characters, following \cite{PY23}). In this way, we show that the dual moment is $O_{F,\e}((q_1 T)^{5/4} q_2^{1/2} U^{-1/4} (q_1 q_2 T)^{\e})$, which yields \hyperref[thm:firstmomentbounds]{Theorem \ref*{thm:firstmomentbounds}}.

\subsubsection{A Sketch of the Proof}

The proof of this spectral reciprocity formula follows the approach of the second author and Khan in \cite{HK25}, where a related result is proven in the special case $q_1 = q_2 = 1$ and $U = T$. We replace the central value $1/2$ with a complex parameter $w$ with large real part, which allows us to replace the $L$-function $L(w,F \otimes f \otimes \chi_1)$ with its absolutely convergent Dirichlet series. After interchanging the order of summation, we apply the Kuznetsov and Petersson formul\ae{}. On the right-hand side of these formul\ae{}, the delta terms give us a main term, while for the Kloosterman terms, we open up the Kloosterman sums, interchange the order of summation, and apply the $\GL_3$ Vorono\u{\i} summation formula. After some careful rearrangements, using both additive reciprocity and analytic reciprocity (see \eqref{eqn:additivereciprocity} and \eqref{eqn:analyticreciprocity}), we find that the Kloosterman terms give rise to a dual moment involving two Dirichlet series, one representing $L(1/2 + it, F \otimes \psi)$ and the other representing $L(2w - 1/2 - it,\overline{\psi})$, as well as a distinguished character sum resembling $g(\chi_1,\psi_1)$. This gives us a spectral reciprocity formula for $L(w,F \otimes f \otimes \chi_1)$ with $\Re(w)$ sufficiently large, which we state in \hyperref[prop:absconv]{Proposition \ref*{prop:absconv}}. To obtain the desired spectral reciprocity formula stated in \hyperref[thm:central]{Theorem \ref*{thm:central}}, we holomorphically extend this identity to the central value $w = 1/2$.

\subsubsection{A Comparison to Alternate Approaches}

Instead of using the absolutely convergent Dirichlet series for $L(w,F \otimes f \otimes \chi_1)$ with $\Re(w) > 1$ followed by analytic continuation to $w = 1/2$, a more traditional approach towards proving \hyperref[thm:firstmomentbounds]{Theorem \ref*{thm:firstmomentbounds}} is to use the approximate functional equation for $L(1/2,F \otimes f \otimes \chi_1)$. Unlike our method, this traditional approach does not yield an exact spectral reciprocity identity for the $\GL_2$ moment of $\GL_3 \times \GL_2$ Rankin--Selberg $L$-functions. Nonetheless, it has been successful in prior results that prove upper bounds of a similar strength to those in \hyperref[thm:firstmomentbounds]{Theorem \ref*{thm:firstmomentbounds}} when $q_2 = 1$; see \cite{Blo12,LNQ23}. When $q_2 > 1$, however, major difficulties arise due to the fact that the length of the approximate functional equation now depends delicately on a divisor of $q_2$, as discussed in \cite[Section 1.3]{PY19}. One approach to overcome this issue is to sieve to newforms and use the Kuznetsov and Petersson formul\ae{} for newforms. When $q_2$ is squarefree, this direction is pursued in \cite{PY19}; for nonsquarefree $q_2$, however, this approach becomes extremely intricate.

A major advantage of our analytic continuation approach is that it circumvents this issue altogether; in particular, we do not require that $q$ be squarefree, as in \cite{PY19}, nor cubefree, as in \cite{PY20}. On the other hand, a disadvantage is that the process of analytic continuation is rather delicate: ensuring that all of the expressions involved are absolutely convergent is a nontrivial task, which we address in \hyperref[sect:testfunction]{Sections \ref*{sect:testfunction}} and \ref{sect:firstmomentcentral}. A further difficult challenge is showing that the transform $\HH(t)$ of $(h,h^{\hol})$ is localised to $|t| \leq T/U$, where it is of size $\approx U$; this involves a careful multivariable stationary phase argument given in \hyperref[sect:testfunctionsII]{Section \ref*{sect:testfunctionsII}}.

\subsection{Previous Results}
\label{sect:previous}

The $\GL_3 \times \GL_2 \leftrightsquigarrow \GL_4 \times \GL_1$ spectral reciprocity formula proven in \hyperref[thm:central]{Theorem \ref*{thm:central}} extends earlier work of the second author and Khan \cite[Theorem 3.1]{HK25} and of Kwan \cite[Theorem 1.1]{Kwa24}, who proved results of this form for $q = 1$. The former follows the same strategy as the proof of \hyperref[thm:central]{Theorem \ref*{thm:central}}, while the latter instead proceeds by evaluating in two different ways the integral of the product of a Poincar\'{e} series and of the restriction of $F$ to $\SL_2(\Z) \backslash \SL_2(\R)$. Recently, Wu generalised this latter approach to arbitrary number fields \cite[Theorem 1.1]{Wu23}. A third approach to such a spectral reciprocity formula, namely evaluating in two different ways the integral of the product of two Hecke--Maa\ss{} newforms and two half-integral weight theta series, is explored in \cite{Nel19a} and \cite{Bir25}. When $F$ is replaced by a minimal parabolic Eisenstein series, such a spectral reciprocity formula is known as Motohashi's formula \cite[Theorem 4.2]{Mot97}, and has been generalised in many directions; see \cite{BCF25,BFW21,BHKM20,Fro20,HPY25,Kan22,Kwa23,Nel19b,Wu22,WX23}.

The approach of proving subconvex bounds for $\GL_3 \times \GL_1$ and $\GL_3 \times \GL_2$ $L$-functions involving selfdual $\GL_3$ cusp forms via bounds for the first moment goes back to work of Li \cite{Li11} in the $t$, $t_f$, or $k_f$ aspects and to work of Blomer \cite{Blo12} in the $q$ aspect (subject to the restriction that $q$ is prime and $\chi$ is quadratic); see additionally \cite{Gan23,Hua21,LNQ23,MSY18,Nun17,Qi19,Qi24,SY19} for various subsequent improvements and extensions of these results. Notably, \hyperref[thm:firstmomentbounds]{Theorem \ref*{thm:firstmomentbounds}} recovers \cite[Theorem 1.1]{LNQ23} upon taking $q = 1$ and recovers \cite[Proposition 3]{Blo12} upon taking $t,t_f,k_f$ to be fixed, $q$ prime, and $\chi$ quadratic. This latter result of Blomer was further extended by the first author to allow for $q$ cube-free and $\chi$ nonquadratic \cite[Theorem 2.0.1]{Gan23}, which in turn is similarly superseded by \hyperref[thm:firstmomentbounds]{Theorem \ref*{thm:firstmomentbounds}}. When $F$ is replaced by a minimal parabolic Eisenstein series, so that the first moment of $L(1/2,F \otimes f \otimes \chi)$ is replaced by the third moment of $L(1/2,f \otimes \chi)$, this approach was pioneered by Conrey and Iwaniec \cite{CI00} and has also been improved and extended in various ways; see \cite{BFW21,Fro20,Ivi01,Lu12,Nel19b,Pen01,Pet15,PY19,PY20,PY23,Wu22,WX23,You17}.

We emphasise that the assumption that $F$ is selfdual in \hyperref[thm:subconvexbounds]{Theorem \ref*{thm:subconvexbounds}} is crucial, since we rely on the nonnegativity of the central $L$-value $L(1/2,F \otimes f)$. Munshi \cite{Mun15a,Mun15b,Mun22} has shown that one can nevertheless prove subconvex bounds (albeit with weaker exponents) for $L(1/2 + it,F \otimes f)$ and $L(1/2 + it, F \otimes \chi)$ without this selfduality assumption via a different method, namely the \emph{delta method}. We direct the reader to \cite{Agg21,Hua24,HX23,Lin21} for a smattering of the state-of-the-art results in this regard.

\subsection{Improvements}

The first term on the right-hand side of \eqref{eqn:firstmomentbounds} is related to the size of the family of automorphic forms over which we average. The size of the family is minimised by taking $q_2$ and $U$ to be as small as possible, namely $q_2 = U = 1$. It was noted in \cite{LNQ23} by the third and fourth authors and Qi, however, that minimising the size of the family comes at the cost of enlarging the size of the dual moment. They showed that if one slightly enlarges the size of the family by instead taking $U = T^{1/5}$, this enlarging of the size of the dual moment is mitigated, leading to improved subconvex bounds. \hyperref[thm:firstmomentbounds]{Theorem \ref*{thm:firstmomentbounds}} introduces this trick additionally in the level aspect: by allowing for the possibility that $q_2 > 1$, we may enlarge the size of the family in the level aspect, which leads to improved subconvex bounds.

The size of the second term on the right-hand size of \eqref{eqn:firstmomentbounds} is intimately related to bounds for the second moment
\begin{equation}
\label{eqn:secondmomentofinterest}
\sum_{\psi \hspace{-.25cm} \pmod{q}} \int_{-T}^{T} \left|L\left(\frac{1}{2} + it,F \otimes \psi\right)\right|^2 \, dt
\end{equation}
for a positive integer $q$ and for $T \geq 1$. The generalised Lindel\"{o}f hypothesis implies the conditional bound $O_{F,\e}((qT)^{1 + \e})$ for this hybrid second moment. A standard application of Gallagher's hybrid large sieve yields only the weaker unconditional bound $O_{F,\e}((q T)^{3/2 + \e})$, as we show in \hyperref[prop:secondmomentFlargesieve]{Proposition \ref*{prop:secondmomentFlargesieve}}. Were we able to improve this to $O_{F,\e}(q^{3/2 - \delta_1 + \e} T^{3/2 - \delta_2 + \e})$ for some $\delta_1,\delta_2 \in [0,1/2]$, we would in turn be able to improve the size of the second term on the right-hand size of \eqref{eqn:firstmomentbounds} to
\[O_{F,\e}\left(q_1^{\frac{5}{4} - \frac{\delta_1}{2}} q_2^{\frac{1}{2}} T^{\frac{5}{4} - \frac{\delta_2}{2}} U^{\frac{\delta_2}{2} - \frac{1}{4}} (q_1 q_2 T)^{\e}\right).\]
In turn, this would yield strengthenings of the subconvex bounds in \hyperref[thm:subconvexbounds]{Theorem \ref*{thm:subconvexbounds}}. We note that when $q = 1$, Dasgupta, Leung, and Young \cite[Theorem 1.1]{DLY24} have shown that the improved bound $O_{F,\e}(T^{3/2 - \delta_2 + \e})$ for \eqref{eqn:secondmomentofinterest} holds with $\delta_2 = 1/6$, which improves upon the exponent $\delta_2 = 3/88$ proved earlier by Pal \cite[Theorem 1]{Pal25}.

\section{Summation Formul\ae{}}

\subsection{Kuznetsov and Petersson Formul\ae{}}

We first state the Kuznetsov and Petersson formul\ae{} for automorphic forms of arbitrary level, where the spectral side is explicitly written in terms of Hecke eigenvalues of newforms. Before we state these formul\ae{}, we explain some notation. Given an $L$-function $L(s,\pi)$, we write $L_p(s,\pi)$ for the $p$-component of the Euler product of $L(s,\pi)$ and define $L_q(s,\pi) \coloneqq \prod_{p \mid q} L_p(s,\pi)$ and $L^q(s,\pi) \coloneqq L(s,\pi)/L_q(s,\pi)$. Given a cusp form $f$ either in $\BB_0^{\ast}(q,\chi)$ or in $\BB_{\hol}^{\ast}(q,\chi)$, we write $\lambda_f(n)$ for the $n$-th Hecke eigenvalue of $f$. Moreover, given a pair of Dirichlet characters $\psi_1,\psi_2$, we define $\lambda_{\psi_1,\psi_2}(n,t) \coloneqq \sum_{ab = n} \psi_1(a) a^{-it} \psi_2(b) b^{it}$, which is the $n$-th Hecke eigenvalue of the Eisenstein series associated to this pair.

\begin{lemma}[Kuznetsov formula]
Let $q$ be a positive integer and let $\chi$ be a primitive Dirichlet character of conductor $q_{\chi} \mid q$. Let $h$ be an even function that is holomorphic in the strip $|\Im(t)| < 1/2 + \delta$ in which it satisfies $h(t) \ll (1 + |t|)^{-2 - \delta}$ for some $\delta > 0$. Then for $(mn,q) = 1$, we have that
\begin{multline}
\label{eqn:Kuznetsov}
\sum_{\substack{q' \mid q \\ q' \equiv 0 \hspace{-.25cm} \pmod{q_{\chi}}}} \alpha(q,q',q_{\chi}) \sum_{f \in \BB_0^{\ast}(q',\chi)} \frac{\overline{\lambda_f(m)} \lambda_f(n)}{L^q(1,\ad f)} h(t_f)	\\
+ \sum_{\substack{q' \mid q \\ q' \equiv 0 \hspace{-.25cm} \pmod{q_{\chi}}}} \alpha(q,q',q_{\chi}) \sum_{\substack{\psi_1,\psi_2 \hspace{-.25cm} \pmod{q} \\ \psi_1 \psi_2 = \chi \\ q_{\psi_1} q_{\psi_2} = q'}} \frac{1}{2\pi} \int_{-\infty}^{\infty} \frac{\overline{\lambda_{\psi_1,\psi_2}(m,t)} \lambda_{\psi_1,\psi_2}(n,t)}{L^q(1 + 2it,\psi_1 \overline{\psi_2})L^q(1 - 2it, \overline{\psi_1} \psi_2)} h(t) \, dt	\\
= \delta_{m,n} q \frac{1}{2\pi^2} \int_{-\infty}^{\infty} h(r) r \tanh \pi r \, dr + q \sum_{\substack{c = 1 \\ c \equiv 0 \hspace{-.25cm} \pmod{q}}}^{\infty} \frac{S_{\chi}(m,n;c)}{c} (\Kscr h)\left(\frac{\sqrt{mn}}{c}\right).
\end{multline}
Here
\begin{gather*}
\alpha(q,q',q_{\chi}) \coloneqq \prod_{\substack{p \mid q' \\ p \nmid \frac{q}{q_{\chi}}}} \left(1 - \frac{1}{p}\right) \prod_{\substack{p \parallel q \\ p \nmid q_{\chi}}} \left(1 - \frac{1}{p^2}\right), \qquad S_{\chi}(m,n;c) \coloneqq \sum_{d \in (\Z/c\Z)^{\times}} \chi(d) e\left(\frac{md + n\overline{d}}{c}\right),	\\
(\Kscr h)(x) \coloneqq \frac{1}{2\pi^2} \int_{-\infty}^{\infty} \JJ_r^{+}(x) h(r) r \tanh \pi r \, dr, \qquad \JJ_r^{+}(x) \coloneqq \frac{\pi i}{\sinh \pi r} \left(J_{2ir}(4\pi x) - J_{-2ir}(4\pi x)\right).
\end{gather*}
\end{lemma}

\begin{proof}
This is stated in \cite[Proposition 3.17]{Hum18} except with the left-hand side written in a slightly different form. The first term on the left-hand side is instead written in the form
\[\sum_{\substack{q' \mid q \\ q' \equiv 0 \hspace{-.25cm} \pmod{q_{\chi}}}} \sum_{f \in \BB_0^{\ast}(q',\chi)} \frac{2q \xi_f \left|\rho_f(1)\right|^2}{\cosh \pi t_f} \overline{\lambda_f(m)} \lambda_f(n) h(t_f),\]
where $\rho_f(1)$ denotes the first Fourier--Whittaker coefficient of $f$. Here $\xi_f$ is a certain constant that, by \cite[Lemma 4.1]{Hum18}, is equal to
\[\xi_f = \sum_{n \mid \left(\frac{q}{q'}\right)^{\infty}} \frac{|\lambda_f(n)|^2}{n} \prod_{\substack{p \parallel \frac{q}{q'} \\ p \nmid q'}} \left(1 - \frac{1}{p^2}\right).\]
Moreover, from \cite[Lemma 4.2]{Hum18}, we have that
\[\frac{\cosh \pi t_f}{|\rho_f(1)|^2} = \pi \vol\left(\Gamma_0(q) \backslash \Hb\right) \Res_{s = 1} \sum_{n = 1}^{\infty} \frac{|\lambda_f(n)|^2}{n^s} = 2q L^q(1,\ad f) \sum_{n \mid q^{\infty}} \frac{|\lambda_f(n)|^2}{n},\]
where the second equality follows from the fact that
\[\sum_{\substack{n = 1 \\ (n,q) = 1}}^{\infty} \frac{|\lambda_f(n)|^2}{n^s} = \frac{\zeta^q(s) L^q(s,\ad f)}{\zeta^q(2s)}.\]
Since
\[|\lambda_f(p^{\beta})|^2 = \begin{dcases*}
1 & if $p \mid q'$ and $p \nmid \frac{q'}{q_{\chi}}$,	\\
\frac{1}{p^{\beta}} & if $p \parallel q'$ and $p \nmid q_{\chi}$,	\\
0 & if $p^2 \mid q'$ and $p \mid \frac{q'}{q_{\chi}}$,
\end{dcases*}\]
we have that
\[\frac{\sum_{n \mid q^{\infty}} |\lambda_f(n)|^2 n^{-1}}{\sum_{n \mid \left(\frac{q}{q'}\right)^{\infty}} |\lambda_f(n)|^2 n^{-1}} = \sum_{\substack{n \mid q'^{\infty} \\ \left(n,\frac{q}{q'}\right) = 1}} \frac{|\lambda_f(n)|^2}{n} = \prod_{\substack{p \mid q' \\ p \nmid \frac{q}{q_{\chi}}}} \frac{1}{1 - p^{-1}} \prod_{\substack{p \parallel q' \\ p \nmid \frac{q q_{\chi}}{q'}}} \frac{1}{1 - p^{-2}}.\]
In particular,
\[\frac{2q \xi_f \left|\rho_f(1)\right|^2}{\cosh \pi t_f} = \frac{\alpha(q,q',q_{\chi})}{L^q(1,\ad f)}.\]
Similarly, the second term on the left-hand side is written in a form that, using the theory of Eisenstein newforms developed in \cite{You19}, can be written as
\[\sum_{\substack{q' \mid q \\ q' \equiv 0 \hspace{-.25cm} \pmod{q_{\chi}}}} \sum_{\substack{\psi_1,\psi_2 \hspace{-.25cm} \pmod{q} \\ \psi_1 \psi_2 = \chi \\ q_{\psi_1} q_{\psi_2} = q'}} \frac{1}{2\pi} \int_{-\infty}^{\infty} \frac{2q \xi_{\psi_1,\psi_2;t} \left|\rho_{\psi_1,\psi_2}(1,t)\right|^2}{\cosh \pi t} \overline{\lambda_{\psi_1,\psi_2}(m,t)} \lambda_{\psi_1,\psi_2}(n,t) h(t) \, dt,\]
where $\rho_{\psi_1,\psi_2}(1,t)$ denotes the first Fourier--Whittaker coefficient of the Eisenstein newform associated to the pair of Dirichlet characters $\psi_1,\psi_2$. An analogous argument shows that
\[\frac{2q \xi_{\psi_1,\psi_2;t} \left|\rho_{\psi_1,\psi_2}(1,t)\right|^2}{\cosh \pi t} = \frac{\alpha(q,q',q_{\chi})}{L^q(1 + 2it,\psi_1 \overline{\psi_2}) L^q(1 - 2it,\overline{\psi_1} \psi_2)}.\qedhere\]
\end{proof}

\begin{lemma}[Petersson formula]
Let $q$ be a positive integer and let $\chi$ be a primitive Dirichlet character of conductor $q_{\chi} \mid q$. Let $h^{\hol} : 2\N \to \C$ be a sequence that satisfies $h^{\hol}(k) \ll k^{-2 - \delta}$ for some $\delta > 0$. Then for $(mn,q) = 1$, we have that
\begin{multline}
\label{eqn:Petersson}
\sum_{\substack{q' \mid q \\ q' \equiv 0 \hspace{-.25cm} \pmod{q_{\chi}}}} \alpha(q,q',q_{\chi}) \sum_{f \in \BB_{\hol}^{\ast}(q',\chi)} \frac{\overline{\lambda_f(m)} \lambda_f(n)}{L^q(1,\ad f)} h^{\hol}(k_f)	\\
= \delta_{m,n} q \sum_{\substack{k = 2 \\ k \equiv 0 \hspace{-.25cm} \pmod{2}}}^{\infty} \frac{k - 1}{2\pi^2} h^{\hol}(k) + q \sum_{\substack{c = 1 \\ c \equiv 0 \hspace{-.25cm} \pmod{q}}}^{\infty} \frac{S_{\chi}(m,n;c)}{c} (\Kscr^{\hol} h^{\hol})\left(\frac{\sqrt{mn}}{c}\right),
\end{multline}
where
\[(\Kscr^{\hol} h^{\hol})(x) \coloneqq \sum_{\substack{k = 2 \\ k \equiv 0 \hspace{-.25cm} \pmod{2}}}^{\infty} \frac{k - 1}{2\pi^2} \JJ_k^{\hol}(x) h^{\hol}(k), \qquad \JJ_k^{\hol}(x) \coloneqq 2\pi i^{-k} J_{k - 1}(4\pi x).\]
\end{lemma}

\begin{proof}
This follows in the same way as for the Kuznetsov formula.
\end{proof}

\subsection{The \texorpdfstring{$\GL_3$}{GL\textthreeinferior} Vorono\u{\i} Summation Formula}

We record some standard facts, which follow from Stirling's formula, regarding certain products of gamma functions that appear in the $\GL_3$ Vorono\u{\i} summation formula.

\begin{lemma}[{\cite[Section 2.2 and Lemma 4]{BK19}}]
For $s \in \C$, define
\begin{equation}
\label{eqn:Gpm}
G^{\pm}(s) \coloneqq \frac{1}{2} G_0(s) \mp \frac{1}{2i} G_1(s) = (2\pi)^{-s} \Gamma(s) \exp\left(\pm \frac{\pi i s}{2}\right),
\end{equation}
where for $\Gamma_{\R}(s) \coloneqq \pi^{-s/2} \Gamma(s/2)$,
\[G_j(s) \coloneqq \frac{\Gamma_{\R}(s + j)}{\Gamma_{\R}(1 - s + j)} = 2(2\pi)^{-s} \Gamma(s) \times \begin{dcases*}
\cos \frac{\pi s}{2} & if $j = 0$,	\\
\sin \frac{\pi s}{2} & if $j = 1$.
\end{dcases*}\]
Then $G^{\pm}(s)$ is meromorphic on $\C$ with simple poles at $s = -\ell$ for each $\ell \in \N_0$. Moreover, if $s = \sigma + i\tau$ is a bounded distance away from such a pole, we have that
\begin{equation}
\label{eqn:Gpmbounds}
G^{\pm}(s) \ll_{\sigma} (1 + |\tau|)^{\sigma - \frac{1}{2}} e^{-\pi \Omega^{\pm}(\tau)},
\end{equation}
where
\[\Omega^{\pm}(\tau) \coloneqq \begin{dcases*}
0 & if $\sgn(\tau) = \mp$,	\\
|\tau| & if $\sgn(\tau) = \pm$.
\end{dcases*}\]
For $|\tau| \geq 1$ and any $M > 0$, there exists a smooth function $\widetilde{g}_{\sigma,M}$ satisfying $|\tau|^m \widetilde{g}_{\sigma,M}^{(m)}(\tau) \ll_{m,\sigma,M} 1$ such that
\begin{equation}
\label{eqn:Gpmasymp}
G^{\pm}(s) = |\tau|^{\sigma - \frac{1}{2}} \exp\left(i\tau \log \frac{|\tau|}{2\pi e}\right) \widetilde{g}_{\sigma,M}(\tau) + O_{\sigma,M}(|\tau|^{-M}).
\end{equation}

Similarly, let $\mu = (\mu_1,\mu_2,\mu_3) \in (i\R)^3$, and for $s \in \C$, define
\begin{equation}
\label{eqn:Gscrmupm}
\Gscr_{\mu}^{\pm}(s) \coloneqq \frac{1}{2} \prod_{j = 1}^{3} G_0(s + \mu_j) \pm \frac{1}{2i} \prod_{j = 1}^{3} G_1(s + \mu_j).
\end{equation}
Then $\Gscr_{\mu}^{\pm}(s)$ is meromorphic on $\C$ with simple poles at $s = -\mu_j - \ell$ for each $\ell \in \N_0$. Moreover, if $s = \sigma + i\tau$ is a bounded distance away from such a pole, we have that
\begin{equation}
\label{eqn:Gscrmubounds}
\Gscr_{\mu}^{\pm}(s) \ll_{\sigma,\mu} (1 + |\tau|)^{3\sigma - \frac{3}{2}} e^{-3\pi \Omega^{\pm}(\tau)}.
\end{equation}
For $|\tau| \geq 1$ and any $M > 0$, there exists a smooth function $g_{\sigma,M,\mu}$ satisfying $|\tau|^m g_{\sigma,M,\mu}^{(m)}(\tau) \ll_{m,\sigma,\mu,M} 1$ such that
\begin{equation}
\label{eqn:Gscrpmasymp}
\Gscr_{\mu}^{\pm}(s) = |\tau|^{3\sigma - \frac{3}{2}} \exp\left(3i\tau \log \frac{|\tau|}{2\pi e}\right) g_{\sigma,M,\mu}(\tau) + O_{\sigma,M,\mu}(|\tau|^{-M}).
\end{equation}
\end{lemma}

We make use of the $\GL_3$ Vorono\u{\i} summation formula due to Miller and Schmid \cite{MS06} in a Dirichlet series form recorded by Blomer and Khan \cite{BK19}. This involves the Kloosterman sum $S(m,n;c) \coloneqq S_{\chi_{0(1)}}(m,n;c)$, where we write $\chi_{0(q)}$ to denote the principal character modulo $q$.

\begin{lemma}[{$\GL_3$ Vorono\u{\i} Summation Formula \cite[Section 4]{BK19}}]
Given a Hecke--Maa\ss{} cusp form $F$ for $\SL_3(\Z)$ with Hecke eigenvalues $A_F(\ell,n)$, define the Vorono\u{\i} series
\begin{align}
\label{eqn:PhiFdefeq}
\Phi_F(c,d,\ell;w) & \coloneqq \sum_{n = 1}^{\infty} \frac{A_F(\ell,n)}{n^w} e\left(\frac{n\overline{d}}{c}\right),	\\
\label{eqn:XiFdefeq}
\Xi_F(c,d,\ell;w) & \coloneqq c \sum_{n_1 \mid c\ell} \sum_{n_2 = 1}^{\infty} \frac{A_F(n_2,n_1)}{n_2 n_1} S\left(d\ell,n_2; \frac{c\ell}{n_1}\right) \left(\frac{n_2 n_1^2}{c^3 \ell}\right)^{-w},
\end{align}
where $c,\ell \in \N$ and $d \in (\Z/c\Z)^{\times}$. The former series converges absolutely for $\Re(w) > 1$, while the latter series converges absolutely for $\Re(w) > 0$, and both series extend holomorphically to the entire complex plane, in which they satisfy the functional equation
\begin{equation}
\label{eqn:GL3Voronoi}
\Phi_F(c,d,\ell;w) = \sum_{\pm} \Gscr_{\mu_F}^{\pm}(1 - w) \Xi_F(c,\pm d,\ell;-w)
\end{equation}
for $\Re(w) < 0$, with $\Gscr_{\mu_F}^{\pm}$ as in \eqref{eqn:Gscrmupm} with $\mu = \mu_F$ equal to the spectral parameters of $F$. Moreover, we have the bounds
\begin{align}
\label{eqn:PhiPL}
\Phi_F(c,d,\ell;w) & \ll_{F,\e} \begin{dcases*}
(c^3 \ell (1 + |\Im(w)|^3))^{\e} \max_{a \mid \ell} |A_F(a,1)| & if $\Re(w) > 1$,	\\
(c^3 \ell (1 + |\Im(w)|^3))^{\frac{1}{2}(1 - \Re(w)) + \e} \max_{a \mid \ell} |A_F(a,1)| & if $0 \leq \Re(w) \leq 1$,	\\
(c^3 \ell (1 + |\Im(w)|^3))^{\frac{1}{2}(1 - 2\Re(w)) + \e} \max_{a \mid \ell} |A_F(a,1)| & if $\Re(w) < 0$,
\end{dcases*}	\\
\label{eqn:XiFPL}
\Xi_F(c,d,\ell;w) & \ll_{F,\e} \begin{dcases*}
(c^3 \ell)^{\frac{1}{2}(1 + 2\Re(w)) + \e} (1 + |\Im(w)|^3)^{\e} \max_{a \mid \ell} |A_F(a,1)| & if $\Re(w) > 0$,	\\
(c^3 \ell)^{\frac{1}{2}(1 + \Re(w)) + \e} (1 + |\Im(w)|^3))^{-\frac{1}{2}\Re(w) + \e} \max_{a \mid \ell} |A_F(a,1)| & if $-1 \leq \Re(w) \leq 0$,	\\
(c^3 \ell)^{\e} (1 + |\Im(w)|^3))^{-\Re(w) - \frac{1}{2} + \e} \max_{a \mid \ell} |A_F(a,1)| & if $\Re(w) < -1$.
\end{dcases*}
\end{align}
\end{lemma}

\subsection{An Application of the \texorpdfstring{$\GL_3$}{GL\textthreeinferior} Vorono\u{\i} Summation Formula}

Given Dirichlet characters $\chi,\psi$ modulo $q \in \N$ and nonzero integers $m_1,m_2,m_3,r \in \Z \setminus \{0\}$, we define the character sum
\begin{equation}
\label{eqn:VVchidefeq}
\VV_{\chi}(\psi;m_1,m_2,m_3,r) \coloneqq \frac{1}{q} \sum_{t,u \in \Z/q\Z} \tau(\overline{\chi},t + m_2 u) \overline{\chi}(rt + m_1 m_2) \tau(\chi,u) \chi(ru - m_1) \tau(\overline{\psi},m_3 t),
\end{equation}
where $\tau(\chi,a)$ denotes the generalised Gauss sum
\begin{equation}
\label{eqn:tauchiadefeq}
\tau(\chi,a) \coloneqq \sum_{b \in \Z/q\Z} \chi(b) e\left(\frac{ab}{q}\right).
\end{equation}
Note that the generalised Gauss sum $\tau(\chi,a)$ satisfies
\begin{equation}
\label{eqn:GausssumCOV}
\tau(\chi,an) = \overline{\chi}(n) \tau(\chi,a), \qquad (n,q) = 1.
\end{equation}
We do \emph{not} assume that $\chi$ is primitive; in particular, it need not be the case that $|\tau(\chi,1)| = \sqrt{q}$. We briefly mention that the character sum \eqref{eqn:VVchidefeq} is trivially bounded in absolute value by $q^4$; in practice, much better bounds are valid, as we discuss in \hyperref[sect:charsumsI]{Section \ref*{sect:charsumsI}}.

We record here a useful fact about Kloosterman sums, namely that they satisfy a twisted multiplicativity property. If $c_1,c_2$ are positive integers for which $(c_1,c_2) = 1$ and if $\chi_1,\chi_2$ are Dirichlet characters modulo $c_1,c_2$, then
\begin{equation}
\label{eqn:Kloostermantwistedmult}
\begin{split}
S_{\chi_1 \chi_2}(m,n;c_1 c_2) & = S_{\chi_1}(m\overline{c_2},n\overline{c_2};c_1) S_{\chi_2}(m\overline{c_1},n\overline{c_1};c_2)	\\
& = \chi_1(c_2) \chi_2(c_1)S_{\chi_1}(m,n\overline{c_2^2};c_1) S_{\chi_2}(m,n\overline{c_1^2};c_2).
\end{split}
\end{equation}

We now prove an identity relating sums of Kloosterman sums and Vorono\u{\i} series to integrals of $L$-functions, which further involves a certain finite Euler product that includes the character sum \eqref{eqn:VVchidefeq}. This identity is central to the proof of $\GL_3 \times \GL_2 \leftrightsquigarrow \GL_4 \times \GL_1$ spectral reciprocity.

\begin{lemma}
\label{lem:Voronoiidentity}
Let $F$ be a Hecke--Maa\ss{} cusp form for $\SL_3(\Z)$, let $q$ be a positive integer, and let $\chi$ be a Dirichlet character modulo $q$. Then for $w = u + iv$ and $s = \sigma + i\tau$ with $u > 3/2$ and $5 - 6u < \sigma < -2u - 1$, we have that
\begin{multline}
\label{eqn:Voronoiidentity}
\sum_{\substack{\ell' = 1 \\ (\ell',q) = 1}}^{\infty} \frac{1}{{\ell'}^{2w}} \sum_{c_0 \mid q^{\infty}} c_0^{s - 2} \sum_{c' \mid \ell'} {c'}^{s + 2w - 2} \sum_{\substack{a \in \Z/c' c_0 q\Z \\ (a,q) = 1}} \chi(a) S_{\overline{\chi}^2}(1,a;c' c_0 q)	\\
\times \sum_{c_1 \mid c' c_0 q} \sum_{b \in (\Z/c_1\Z)^{\times}} e\left(\frac{a\overline{b}}{c_1}\right) \Xi_F\left(c_1,\mp b,\frac{\ell'}{c'};-\frac{s}{2} - w\right)	\\
= \frac{q^{1 - s}}{\varphi(q)} \sum_{\psi \hspace{-.25cm} \pmod{q}} \frac{1}{2\pi i} \int_{\CC_1} L\left(\frac{1}{2} + z, \widetilde{F} \otimes \psi\right) L\left(2w - \frac{1}{2} - z,\overline{\psi}\right) \ZZ_{\chi}(\psi;w,z)	\\
\times \psi(\mp 1) G^{\mp}\left(\frac{s}{2} + w - \frac{1}{2} + z\right) \, dz.
\end{multline}
Here $G^{\pm}$ is as in \eqref{eqn:Gpm}, while $\CC_1$ is the contour consisting of the straight lines connecting the points $x_1 - i\infty$, $x_1 - i(\tau/2 + v + 1)$, $1/2 - s/2 - w + \delta$, $x_1 - i(\tau/2 + v - 1)$, $x_1 + i\infty$, with $1/2 < x_1 < -\sigma/2 - u$ and $0 < \delta < \sigma/2 + 3u - 2$, and
\begin{multline}
\label{eqn:ZZchipsiwzdefeq}
\ZZ_{\chi}(\psi;w,z) \coloneqq \sum_{c_0 \mid q^{\infty}} \frac{1}{\varphi(c_0 q)^2 c_0^{2w - \frac{1}{2} - z}} \sum_{\substack{c_{1,0} c_{2,0} d_0 n_{1,0} = q \\ (c_{2,0} d_0 n_{1,0},c_0) = 1}} \frac{\varphi(c_0 c_{1,0}) \varphi(c_0 c_{1,0} d_0 n_{1,0}) \mu(d_0) A_F(1,n_{1,0})}{c_{1,0}^{2w - \frac{1}{2} - z} c_{2,0}^{2w - 1 + 2z} d_0^{2w + 2z} n_{1,0}^{2w - \frac{1}{2} + z}}	\\
\times \sum_{\substack{n_{2,0} \mid q^{\infty} \\ (n_{2,0},c_0) = 1}} \frac{A_F(n_{2,0},1)}{n_{2,0}^{\frac{1}{2} + z}} \VV_{\chi}(\psi; c_{2,0} d_0 n_{1,0}, c_{2,0} d_0^2 n_{1,0} n_{2,0}, c_{2,0}, c_0),
\end{multline}
with $\VV_{\chi}(\psi;m_1,m_2,m_3,r)$ as in \eqref{eqn:VVchidefeq}.
\end{lemma}

Here for $\Re(s) > 1$ and $\psi$ a Dirichlet character modulo $q$, we have that
\[L(s,\widetilde{F} \otimes \psi) \coloneqq \sum_{n = 1}^{\infty} \frac{A_F(n,1) \psi(n)}{n^s}, \qquad L(s,\overline{\psi}) \coloneqq \sum_{n = 1}^{\infty} \frac{\overline{\psi}(n)}{n^s}.\]
If $\psi$ is imprimitive, these must be corrected by certain Euler factors in order to obtain primitive $L$-functions.

\begin{remark}
The expression \eqref{eqn:ZZchipsiwzdefeq} defining $\ZZ_{\chi}(\psi;w,z)$ is valid more generally for $(w,z) \in \C^2$ satisfying $\Re(w) > 5/28$ and $-1/7 < \Re(z) < 2\Re(w) - 1/2$, since this expression is absolutely convergent in this region. Here the lower bound for $\Re(z)$ is required to ensure that the sum over $n_{2,0} \mid q^{\infty}$ converges absolutely, noting that the best known bounds for the generalised Ramanujan conjecture ensures that $A_F(n_{2,0},1) \ll_{\e} n_{2,0}^{5/14 + \e}$ (see \cite{Kim03}). The upper bound for $\Re(z)$ is required to ensure that the sum over $c_0 \mid q^{\infty}$ converges absolutely.
\end{remark}

\begin{proof}[Proof of {\hyperref[lem:Voronoiidentity]{Lemma \ref*{lem:Voronoiidentity}}}]
From \eqref{eqn:XiFPL}, the assumption that $5 - 6\Re(w) < \Re(s) < -2\Re(w) - 1$ ensures the absolute convergence of the sum over $\ell' \in \N$ with $(\ell',q) = 1$ on the left-hand side of \eqref{eqn:Voronoiidentity}. We may replace the Vorono\u{\i} series $\Xi_F$ with the absolutely convergent series \eqref{eqn:XiFdefeq}. We then write $c_1 = c_1' c_{1,0}$, $n_1 = n_1' n_{1,0}$, and $n_2 = n_2' n_{2,0}$, where $(c_1' n_1' n_2',q) = 1$ and $c_{1,0} n_{1,0} n_{2,0} \mid q^{\infty}$. The left-hand side of \eqref{eqn:Voronoiidentity} becomes
\begin{multline}
\label{eqn:openVoronoi1}
\sum_{\substack{\ell' = 1 \\ (\ell',q) = 1}}^{\infty} \frac{1}{{\ell'}^{\frac{s}{2} + 3w}} \sum_{c' \mid \ell'} \frac{1}{c'} \sum_{c_1' c_2' = c'} {c_2'}^{\frac{3s}{2} + 3w - 1} \sum_{n_1' \mid \frac{\ell'}{c_2'}} \sum_{\substack{n_2' = 1 \\ (n_2',q) = 1}}^{\infty} \frac{A_F(n_2',n_1')}{{n_2'}^{1 - \frac{s}{2} - w} {n_1'}^{1 - s - 2w}}	\\
\times \sum_{c_0 \mid q^{\infty}} c_0^{s - 2} \sum_{c_{1,0} c_{2,0} = c_0 q} c_{1,0}^{1 - \frac{3s}{2} - 3w} \sum_{n_{1,0} \mid c_{1,0}} \sum_{n_{2,0} \mid q^{\infty}} \frac{A_F(n_{2,0},n_{1,0})}{n_{2,0}^{1 - \frac{s}{2} - w} n_{1,0}^{1 - s - 2w}}	\\
\times \sum_{\substack{a \in \Z/c' c_0 q\Z \\ (a,q) = 1}} \chi(a) S_{\overline{\chi}^2}(1,a;c' c_0 q) \sum_{b \in (\Z/c_1' c_{1,0} \Z)^{\times}} e\left(\frac{a\overline{b}}{c_1' c_{1,0}}\right) S\left(b\frac{\ell'}{c'}, \mp n_2' n_{2,0}; \frac{\ell'}{c_2' n_1'} \frac{c_{1,0}}{n_{1,0}}\right),
\end{multline}
where we have used the multiplicativity of the Hecke eigenvalues $A_F(n_2,n_1)$, namely the Hecke relations
\[A_F(m_1 n_1,m_2 n_2) = A_F(m_1,m_2) A_F(n_1,n_2)\]
whenever $(m_1,n_1) = (m_2,n_2) = 1$ \cite[Theorem 6.4.11]{Gol06}.

By the Chinese remainder theorem, we may write $a = a_0 c' + a' c_0 q$, where $a_0 \in (\Z/c_0 q\Z)^{\times}$ and $a' \in \Z/c'\Z$, and $b = b_0 c_1' + b' c_{1,0}$, where $b_0 \in (\Z/c_{1,0}\Z)^{\times}$ and $b' \in (\Z/c_1'\Z)^{\times}$, so that $\overline{b} \equiv \overline{b_0 c_1'} \pmod{c_{1,0}}$ and $\overline{b} \equiv \overline{b' c_{1,0}} \pmod{c_1'}$. Via the twisted multiplicativity of Kloosterman sums \eqref{eqn:Kloostermantwistedmult} together with the change of variables $a_0 \mapsto a_0 c'$ and $a' \mapsto a' c_0 q$, the last line of \eqref{eqn:openVoronoi1} is equal to
\begin{multline}
\label{eqn:twistedmult}
\sum_{a' \in \Z/c'\Z} S(1,a';c') \sum_{b' \in (\Z/c_1'\Z)^{\times}} e\left(\frac{a' \overline{b'} c_{2,0}^2}{c_1'}\right) S\left(b' \frac{\ell'}{c'} n_{1,0}, \mp \overline{\frac{c_{1,0}}{n_{1,0}}} n_2' n_{2,0}; \frac{\ell'}{c_2' n_1'}\right)	\\
\times \sum_{a_0 \in (\Z/c_0 q\Z)^{\times}} \chi(a_0) S_{\overline{\chi}^2}(1,a_0;c_0 q) \sum_{b_0 \in (\Z/c_{1,0}\Z)^{\times}} e\left(\frac{a_0 \overline{b_0} c_2^{\prime 2}}{c_{1,0}}\right) S\left(b_0 n_1', \mp \overline{\frac{\ell'}{c_2' n_1'}} n_2' n_{2,0}; \frac{c_{1,0}}{n_{1,0}}\right).
\end{multline}

We first deal with the double sum over $a'$ and $b'$ on the first line of \eqref{eqn:twistedmult}. We inflate the sum over $b' \in (\Z/c_1'\Z)^{\times}$ to run over elements of $(\Z/c'\Z)^{\times}$, at a cost of multiplying through by $\varphi(c_1')/\varphi(c')$. We then open up the Kloosterman sum $S(1,a';c')$ as a sum over $d' \in (\Z/c'\Z)^{\times}$ and make the change of variables $a' \mapsto a' b' d'$, yielding
\[\frac{\varphi(c_1')}{\varphi(c')} \sum_{d' \in (\Z/c'\Z)^{\times}} e\left(\frac{d'}{c'}\right) \sum_{b' \in (\Z/c'\Z)^{\times}} S\left(b' \frac{\ell'}{c'} n_{1,0}, \mp \overline{\frac{c_{1,0}}{n_{1,0}}} n_2' n_{2,0}; \frac{\ell'}{c_2' n_1'}\right) \sum_{a' \in \Z/c'\Z} e\left(\frac{a' (b' + c_2' c_{2,0}^2 d')}{c'}\right).\]
The innermost sum over $a' \in \Z/c'\Z$ vanishes unless $b' \equiv -c_2' c_{2,0}^2 d' \pmod{c'}$, in which case it is equal to $c'$. Since $c_1' c_2' = c'$ and $(b',c') = 1$, this congruence can only hold when $c_1' = c'$ and $c_2' = 1$, in which case $b' \equiv -c_{2,0}^2 d' \pmod{c'}$. Thus this becomes
\[\delta_{c_2',1} c' \sum_{d' \in (\Z/c'\Z)^{\times}} e\left(\frac{d'}{c'}\right) S\left(-c_{2,0}^2 d' \frac{\ell'}{c'} n_{1,0}, \mp \overline{\frac{c_{1,0}}{n_{1,0}}} n_2' n_{2,0}; \frac{\ell'}{n_1'}\right).\]
We open up the Kloosterman sum as a sum over $a' \in (\Z/\frac{\ell'}{n_1'}\Z)^{\times}$ and then inflate this sum to run over elements of $(\Z/\ell'\Z)^{\times}$, at a cost of multiplying through by $\varphi(\frac{\ell'}{n_1'})/\varphi(\ell')$. The ensuing sum over $d' \in (\Z/c'\Z)^{\times}$ is the Ramanujan sum $R_{c'}(1 - a' c_{2,0}^2 n_1' n_{1,0})$, where
\begin{equation}
\label{eqn:Rqndefeq}
R_q(n) \coloneqq \sum_{d \in (\Z/q\Z)^{\times}} e\left(\frac{dn}{q}\right) = \sum_{d \mid (q,n)} d \mu\left(\frac{q}{d}\right)
\end{equation}
We therefore arrive at the expression
\begin{equation}
\label{eqn:twistedmult2}
\delta_{c_2',1} c' \frac{\varphi\left(\frac{\ell'}{n_1'}\right)}{\varphi(\ell')} \sum_{a' \in (\Z/\ell'\Z)^{\times}} e\left(\mp \frac{\overline{a' \frac{c_{1,0}}{n_{1,0}}} n_1' n_2' n_{2,0}}{\ell'}\right) \sum_{d' \mid (c', 1 - a' c_{2,0}^2 n_1' n_{1,0})} d' \mu\left(\frac{c'}{d'}\right)
\end{equation}
for the first line of \eqref{eqn:twistedmult}.

We insert the expression \eqref{eqn:twistedmult2} back into \eqref{eqn:twistedmult}, which in turn we insert back into \eqref{eqn:openVoronoi1}. We find that \eqref{eqn:openVoronoi1} is equal to
\begin{multline}
\label{eqn:openVoronoi2}
\sum_{\substack{\ell' = 1 \\ (\ell',q) = 1}}^{\infty} \frac{1}{{\ell'}^{\frac{s}{2} + 3w}} \sum_{n_1' \mid \ell'} \frac{\varphi\left(\frac{\ell'}{n_1'}\right)}{\varphi(\ell')} \sum_{\substack{n_2' = 1 \\ (n_2',q) = 1}}^{\infty} \frac{A_F(n_2',n_1')}{{n_2'}^{1 - \frac{s}{2} - w} {n_1'}^{1 - s - 2w}}	\\
\times \sum_{c_0 \mid q^{\infty}} c_0^{s - 2} \sum_{c_{1,0} c_{2,0} = c_0 q} c_{1,0}^{1 - \frac{3s}{2} - 3w} \sum_{n_{1,0} \mid c_{1,0}} \sum_{n_{2,0} \mid q^{\infty}} \frac{A_F(n_{2,0},n_{1,0})}{n_{2,0}^{1 - \frac{s}{2} - w} n_{1,0}^{1 - s - 2w}}	\\
\times \sum_{a_0 \in (\Z/c_0 q\Z)^{\times}} \chi(a_0) S_{\overline{\chi}^2}(1,a_0;c_0 q) \sum_{b_0 \in (\Z/c_{1,0}\Z)^{\times}} e\left(\frac{a_0 \overline{b_0}}{c_{1,0}}\right) S\left(b_0 n_1', \mp \overline{\frac{\ell'}{n_1'}} n_2' n_{2,0}; \frac{c_{1,0}}{n_{1,0}}\right)	\\
\times \sum_{c' \mid \ell'} \sum_{a' \in (\Z/\ell'\Z)^{\times}} e\left(\mp \frac{\overline{a' \frac{c_{1,0}}{n_{1,0}}} n_1' n_2' n_{2,0}}{\ell'}\right) \sum_{d' \mid (c', 1 - a' c_{2,0}^2 n_1' n_{1,0})} d' \mu\left(\frac{c'}{d'}\right).
\end{multline}
In the last line, we interchange the order of summation and make the change of variables $c' \mapsto c' d'$, so that $c' \mid \frac{\ell'}{d'}$ and $d' \mid (\ell',1 - a' c_{2,0}^2 n_1' n_{1,0})$. Since $\sum_{c' \mid \frac{\ell'}{d'}} \mu(c')$ is $1$ if $d' = \ell'$ and is $0$ otherwise, the last line of \eqref{eqn:openVoronoi2} becomes
\[\ell' \sum_{\substack{a' \in (\Z/\ell'\Z)^{\times} \\ a' c_{2,0}^2 n_1' n_{1,0} \equiv 1 \hspace{-.25cm} \pmod{\ell'}}} e\left(\mp \frac{\overline{a' \frac{c_{1,0}}{n_{1,0}}} n_1' n_2' n_{2,0}}{\ell'}\right).\]
The congruence condition $a' c_{2,0}^2 n_1' n_{1,0} \equiv 1 \pmod{\ell'}$ subject to the constraint $n_1' \mid \ell'$ can only be met if $n_1' = 1$, in which case $a' \equiv \overline{c_{2,0}^2 n_{1,0}} \pmod{\ell'}$, and thus this is
\[\delta_{n_1',1} \ell' e\left(\mp \frac{\overline{\frac{c_{1,0}}{n_{1,0}}} c_{2,0}^2 n_{1,0} n_2' n_{2,0}}{\ell'}\right).\]
We have therefore shown that \eqref{eqn:openVoronoi2} is equal to
\begin{multline}
\label{eqn:openVoronoi3}
\sum_{\substack{\ell' = 1 \\ (\ell',q) = 1}}^{\infty} \frac{1}{{\ell'}^{\frac{s}{2} + 3w - 1}} \sum_{\substack{n_2' = 1 \\ (n_2',q) = 1}}^{\infty} \frac{A_F(n_2',1)}{{n_2'}^{1 - \frac{s}{2} - w}}	\\
\times \sum_{c_0 \mid q^{\infty}} c_0^{s - 2} \sum_{c_{1,0} c_{2,0} = c_0 q} c_{1,0}^{1 - \frac{3s}{2} - 3w} \sum_{n_{1,0} \mid c_{1,0}} \sum_{n_{2,0} \mid q^{\infty}} \frac{A_F(n_{2,0},n_{1,0})}{n_{2,0}^{1 - \frac{s}{2} - w} n_{1,0}^{1 - s - 2w}} e\left(\mp \frac{\overline{\frac{c_{1,0}}{n_{1,0}}} c_{2,0}^2 n_{1,0} n_2' n_{2,0}}{\ell'}\right)	\\
\times \sum_{a_0 \in (\Z/c_0 q\Z)^{\times}} \chi(a_0) S_{\overline{\chi}^2}(1,a_0;c_0 q) \sum_{b_0 \in (\Z/c_{1,0}\Z)^{\times}} e\left(\frac{a_0 \overline{b_0}}{c_{1,0}}\right) S\left(b_0, \mp \overline{\ell'} n_2' n_{2,0}; \frac{c_{1,0}}{n_{1,0}}\right).
\end{multline}

We now deal with the last line of \eqref{eqn:openVoronoi3}. We open up the first Kloosterman sum as a sum over $d_0 \in (\Z/c_0 q\Z)^{\times}$ and open up the second Kloosterman sum as a sum over $d_1 \in (\Z/\frac{c_{1,0}}{n_{1,0}}\Z)^{\times}$, inflate the sums over $b_0 \in (\Z/c_{1,0}\Z)^{\times}$ and $d_1 \in (\Z/\frac{c_{1,0}}{n_{1,0}}\Z)^{\times}$ to run over elements of $(\Z/c_0 q\Z)^{\times}$, at the cost of multiplying through by $\varphi(c_{1,0}) \varphi(\frac{c_{1,0}}{n_{1,0}}) / \varphi(c_0 q)^2$, and make the change of variables $a_0 \mapsto a_0 b_0 d_0 d_1$, $b_0 \mapsto b_0 \overline{d_1}$, $d_0 \mapsto d_0 \overline{d_1}$, and $d_1 \mapsto \overline{d_1}$. We find that the last line of \eqref{eqn:openVoronoi3} is
\begin{multline*}
\frac{\varphi(c_{1,0}) \varphi\left(\frac{c_{1,0}}{n_{1,0}}\right)}{\varphi(c_0 q)^2} \sum_{a_0 \in \Z/c_0 q\Z} \chi(a_0) \sum_{d_0 \in \Z/c_0 q\Z} \overline{\chi}(d_0) e\left(\frac{a_0 c_{2,0} d_0}{c_0 q}\right)	\\
\times \sum_{b_0 \in \Z/c_0 q\Z} \chi(b_0) e\left(\frac{b_0 (a_0 + c_{2,0} n_{1,0})}{c_0 q}\right) \sum_{d_1 \in \Z/c_0 q\Z} \overline{\chi}(d_1) e\left(\frac{d_1 (d_0 \mp c_{2,0} \overline{\ell'} n_{1,0} n_2' n_{2,0})}{c_0 q}\right)
\end{multline*}
Here we have extended each sum to be over $\Z/c_0 q\Z$ instead of $(\Z/c_0 q\Z)^{\times}$ since $\chi(a) = 0$ whenever $(a,c_0 q) \neq 1$ due to the fact that $c_0 \mid q^{\infty}$.

The sum over $b_0 \in \Z/c_0 q\Z$ vanishes unless $a_0 \equiv - c_{2,0} n_{1,0} \pmod{c_0}$ (which can only occur when $(c_0,c_{2,0} n_{1,0}) = 1$), in which case it is
\[c_0 \tau\left(\chi,\frac{a_0 + c_{2,0} n_{1,0}}{c_0}\right).\]
Similarly, the sum over $d_1 \in \Z/c_0 q\Z$ vanishes unless $d_0 \equiv \pm c_{2,0} \overline{\ell'} n_{1,0} n_2' n_{2,0} \pmod{c_0}$ (which can only occur if $(c_0,c_{2,0} n_{1,0} n_{2,0}) = 1$), in which case it is
\[c_0 \tau\left(\overline{\chi},\frac{d_0 \mp c_{2,0} \overline{\ell'} n_{1,0} n_2' n_{2,0}}{c_0}\right).\]

After making the change of variables $a_0 \mapsto a_0 c_0 - c_{2,0} n_{1,0}$ and $d_0 \mapsto c_0 d_0 \overline{\ell'} n_2' \pm c_{2,0} \overline{\ell'} n_{1,0} n_2' n_{2,0}$, where now $a_0, d_0 \in \Z/q\Z$, and using \eqref{eqn:GausssumCOV}, we find that the last line of \eqref{eqn:openVoronoi3} is
\begin{multline*}
\delta_{(c_0,c_{2,0} n_{1,0} n_{2,0}),1} \frac{\varphi(c_{1,0}) \varphi\left(\frac{c_{1,0}}{n_{1,0}}\right)}{\varphi(c_0 q)^2} c_0^2 e\left(\mp \frac{\overline{\ell'} c_{2,0}^2 n_{1,0} n_2' n_{2,0}}{\frac{c_{1,0}}{n_{1,0}}}\right)	\\
\times \sum_{a_0 \in \Z/q\Z} \sum_{d_0 \in \Z/q\Z} \tau(\overline{\chi},d_0) \overline{\chi}(c_0 d_0 \pm c_{2,0} n_{1,0} n_{2,0}) \tau(\chi,a_0) \chi(a_0 c_0 - c_{2,0} n_{1,0})	\\
\times e\left(\frac{((a_0 c_0 - c_{2,0} n_{1,0}) d_0 \pm a_0 c_{2,0} n_{1,0} n_{2,0}) c_{2,0} \overline{\ell'} n_2'}{q}\right).
\end{multline*}
Since the double sum over $a_0,d_0 \in \Z/q\Z$ vanishes unless $(a_0 c_0 - c_{2,0} n_{1,0},q) = 1$, we may make the change of variables $d_0 \mapsto \mp(\overline{a_0 c_0 - c_{2,0} n_{1,0}}) (d_0 + a_0 c_{2,0} n_{1,0} n_{2,0})$. Via \eqref{eqn:GausssumCOV}, this yields
\begin{multline*}
\delta_{(c_0,c_{2,0} n_{1,0} n_{2,0}),1} \frac{\varphi(c_{1,0}) \varphi\left(\frac{c_{1,0}}{n_{1,0}}\right)}{\varphi(c_0 q)^2} c_0^2 e\left(\mp \frac{\overline{\ell'} c_{2,0}^2 n_{1,0} n_2' n_{2,0}}{\frac{c_{1,0}}{n_{1,0}}}\right)	\\
\times \sum_{a_0 \in \Z/q\Z} \sum_{d_0 \in \Z/q\Z} \tau(\overline{\chi}, d_0 + a_0 c_{2,0} n_{1,0} n_{2,0}) \overline{\chi}(c_0 d_0 + c_{2,0}^2 n_{1,0}^2 n_{2,0}) \tau(\chi,a_0) \chi(a_0 c_0 - c_{2,0} n_{1,0})	\\
\times e\left(\mp \frac{c_{2,0} d_0 \overline{\ell'} n_2'}{q}\right).
\end{multline*}

Finally, we use character orthogonality to write
\[e\left(\mp \frac{c_{2,0} d_0 \overline{\ell'} n_2'}{q}\right) = \frac{1}{\varphi(q)} \sum_{\psi \hspace{-.25cm} \pmod{q}} \psi(\mp 1) \overline{\psi}(\ell') \psi(n_2') \tau(\overline{\psi},c_{2,0} d_0).\]
Recalling \eqref{eqn:VVchidefeq}, we have therefore shown the last line of \eqref{eqn:openVoronoi3} is equal to
\begin{multline}
\label{eqn:openVoronoi4}
\delta_{(c_0,c_{2,0} n_{1,0} n_{2,0}),1} \frac{\varphi(c_{1,0}) \varphi\left(\frac{c_{1,0}}{n_{1,0}}\right) q}{\varphi(c_0 q)^2 \varphi(q)} c_0^2 e\left(\mp \frac{\overline{\ell'} c_{2,0}^2 n_{1,0} n_2' n_{2,0}}{\frac{c_{1,0}}{n_{1,0}}}\right) \sum_{\psi \hspace{-.25cm} \pmod{q}} \psi(\mp 1) \overline{\psi}(\ell') \psi(n_2')	\\
\times \VV_{\chi}(\psi; c_{2,0} n_{1,0}, c_{2,0} n_{1,0} n_{2,0}, c_{2,0}, c_0).
\end{multline}

We insert the expression \eqref{eqn:openVoronoi4} back into \eqref{eqn:openVoronoi3} and use the additive reciprocity formula
\begin{equation}
\label{eqn:additivereciprocity}
e\left(\mp \frac{\overline{\frac{c_{1,0}}{n_{1,0}}} c_{2,0}^2 n_{1,0} n_2' n_{2,0}}{\ell'}\right) e\left(\mp \frac{\overline{\ell'} c_{2,0}^2 n_{1,0} n_2' n_{2,0}}{\frac{c_{1,0}}{n_{1,0}}}\right) = e\left(\mp \frac{c_{2,0}^2 n_{1,0}^2 n_2' n_{2,0}}{c_{1,0} \ell'}\right).
\end{equation}
We additionally note that the conditions $(c_0,c_{2,0} n_{1,0}) = 1$ and $c_{1,0} c_{2,0} = c_0 q$ with $n_{1,0} \mid c_{1,0}$ can only be met if $c_{1,0} \equiv 0 \pmod{c_0 n_{1,0}}$, and so we make the change of variables $c_{1,0} \mapsto c_0 c_{1,0} n_{1,0}$. Applying the Hecke relations \cite[Theorem 6.4.11]{Gol06}
\[A_F(n_{2,0},n_{1,0}) = \sum_{d_0 \mid (n_{2,0},n_{1,0})} \mu(d_0) A_F\left(\frac{n_{2,0}}{d_0},1\right) A_F\left(1,\frac{n_{1,0}}{d_0}\right)\]
and making the change of variables $n_{1,0} \mapsto d_0 n_{1,0}$ and $n_{2,0} \mapsto d_0 n_{2,0}$, we see that \eqref{eqn:openVoronoi3} is equal to
\begin{multline*}
\frac{q^{1 - s}}{\varphi(q)} \sum_{\psi \hspace{-.25cm} \pmod{q}} \psi(\mp 1) \sum_{\substack{n_2' = 1 \\ (n_2',q) = 1}}^{\infty} \frac{A_F(n_2',1) \psi(n_2')}{{n_2'}^{1 - \frac{s}{2} - w}} \sum_{\substack{\ell' = 1 \\ (\ell',q) = 1}}^{\infty} \frac{\overline{\psi}(\ell')}{{\ell'}^{\frac{s}{2} + 3w - 1}}	\\
\times \sum_{c_0 \mid q^{\infty}} \frac{1}{\varphi(c_0 q)^2 c_0^{\frac{s}{2} + 3w - 1}} \sum_{\substack{c_{1,0} c_{2,0} d_0 n_{1,0} = q \\ (c_{2,0} d_0 n_{1,0},c_0) = 1}} \frac{\varphi(c_0 c_{1,0}) \varphi(c_0 c_{1,0} d_0 n_{1,0}) \mu(d_0) A_F(1,n_{1,0})}{c_{1,0}^{\frac{s}{2} + 3w - 1} c_{2,0}^{-s} d_0^{1 - s} n_{1,0}^{w - \frac{s}{2}}}	\\
\times \sum_{\substack{n_{2,0} \mid q^{\infty} \\ (n_{2,0},c_0) = 1}} \frac{A_F(n_{2,0},1)}{n_{2,0}^{1 - \frac{s}{2} - w}} \VV_{\chi}(\psi; c_{2,0} d_0 n_{1,0}, c_{2,0} d_0^2 n_{1,0} n_{2,0}, c_{2,0}, c_0) e\left(\mp \frac{c_{2,0}^2 d_0^2 n_{1,0} n_2' n_{2,0}}{c_0 c_{1,0} \ell'}\right).
\end{multline*}

Our final step is to insert the analytic reciprocity formula
\begin{equation}
\label{eqn:analyticreciprocity}
e\left(\mp \frac{c_{2,0}^2 d_0^2 n_{1,0} n_2' n_{2,0}}{c_0 c_{1,0} \ell'}\right) = \frac{1}{2\pi i} \int_{\CC_0} G^{\mp}(z) \left(\frac{c_{2,0}^2 d_0^2 n_{1,0} n_2' n_{2,0}}{c_0 c_{1,0} \ell'}\right)^{-z} \, dz.
\end{equation}
Here $G^{\pm}$ is as in \eqref{eqn:Gpm}, while $\CC_0$ is the contour consisting of the straight lines connecting the points $x_0 - i\infty$, $x_0 - i$, $\delta$, $x_0 + i$, and $x_0 + i\infty$, with $x_0 < -1/2$ and $\delta > 0$; this ensures, via \eqref{eqn:Gpmbounds}, that this integral converges absolutely but that the poles of the integrand are to the left of the contour. To see the equality \eqref{eqn:analyticreciprocity}, one can shift the contour of integration off to the left, picking up residues at $-\ell$ for each $\ell \in \N_0$, and noting that the resulting expression is the Taylor series expression for the left-hand side. Finally, we make the change of variables $z \mapsto z + s/2 + w - 1/2$.

After interchanging the order of summation and integration, we arrive at the desired identity \eqref{eqn:Voronoiidentity}, where we require the condition $x_1 < -\sigma/2 - u$ in order to ensure the absolute convergence of the integral over $z$, the condition $x_1 > 1/2$ in order to ensure the absolute convergence of the sum over $n_2'$, namely the Dirichlet series for $L(1/2 + z, \widetilde{F} \otimes \psi)$, and the condition $\delta < \sigma/2 + 3u - 2$ in order to ensure the absolute convergence of the sum over $\ell'$, namely the Dirichlet series for $L(2w - 1/2 - z,\overline{\psi})$.
\end{proof}

\section{The First Moment in the Region of Absolute Convergence}

We now prove our desired $\GL_3 \times \GL_2 \leftrightsquigarrow \GL_4 \times \GL_1$ spectral reciprocity formula, with the caveat that we prove this initially in the region of absolute convergence. The following result gives an \emph{equality} of moments of $L$-functions involving the $\GL_3 \times \GL_2$ Rankin--Selberg $L$-function $L(w,F \otimes f \otimes \chi_1)$, where $7/4 < \Re(w) < 2$. In \hyperref[sect:firstmomentcentral]{Section \ref*{sect:firstmomentcentral}}, we holomorphically extend this identity to the central value $w = 1/2$. The identity that we give is valid for a tuple of test functions $(h,h^{\hol})$ satisfying certain stringent conditions; we show in \hyperref[sect:testfunction]{Section \ref*{sect:testfunction}} that these conditions are met for a flexible family of tests functions.

\begin{proposition}
\label{prop:absconv}
Let $F$ be a Hecke--Maa\ss{} cusp form for $\SL_3(\Z)$ and let $q = q_1 q_2$ be a positive integer with $(q_1,q_2) = 1$. Let $\chi_1$ be a primitive Dirichlet character of conductor $q_1$, and set $\chi \coloneqq \chi_1 \chi_{0(q_2)}$. Let $h$ be an even function that is holomorphic in the strip $|\Im(t)| < 1/2 + \delta$ in which it satisfies $h(t) \ll (1 + |t|)^{-2 - \delta}$ for some $\delta > 0$ and let $h^{\hol} : 2\N \to \C$ be a sequence satisfying $h^{\hol}(k) \ll k^{-2 - \delta}$ for some $\delta > 0$. Define the transform
\begin{equation}
\label{eqn:HdefeqKscr}
H(x) \coloneqq (\Kscr h)(x) + (\Kscr^{\hol} h^{\hol})(x),
\end{equation}
and suppose that the Mellin transform $\widehat{H}(s) \coloneqq \int_{0}^{\infty} H(x) x^s \, \frac{dx}{x}$ is holomorphic in the strip $-5 < \Re(s) < 1$, in which it satisfies the bounds
\begin{equation}
\label{eqn:HpmMellindecay}
\widehat{H}(s) \ll (1 + |\Im(s)|)^{\Re(s) - 1}.
\end{equation}
Then for $w = u + iv$ with $7/4 < u < 2$, the moment
\begin{multline}
\label{eqn:absoluteLHS}
\sum_{\substack{q' \mid q \\ q' \equiv 0 \hspace{-.25cm} \pmod{q_{\overline{\chi_1}^2}}}} \alpha(q,q',q_{\overline{\chi_1}^2}) \sum_{f \in \BB_0^{\ast}(q',\overline{\chi_1}^2)} \frac{L^q(w,F \otimes f \otimes \chi_1)}{L^q(1,\ad f)} h(t_f)	\\
+ \sum_{\substack{q' \mid q \\ q' \equiv 0 \hspace{-.25cm} \pmod{q_{\overline{\chi_1}^2}}}} \alpha(q,q',q_{\overline{\chi_1}^2}) \sum_{\substack{\psi_1,\psi_2 \hspace{-.25cm} \pmod{q} \\ \psi_1 \psi_2 = \overline{\chi_1}^2 \\ q_{\psi_1} q_{\psi_2} = q'}} \frac{1}{2\pi} \int_{-\infty}^{\infty} \frac{L^q(w + it,F \otimes \psi_1 \chi_1) L^q(w - it,F \otimes \overline{\psi_1 \chi_1})}{L^q(1 + 2it,\psi_1\overline{\psi_2}) L^q(1 - 2it,\overline{\psi_1}\psi_2)} h(t) \, dt	\\
+ \sum_{\substack{q' \mid q \\ q' \equiv 0 \hspace{-.25cm} \pmod{q_{\overline{\chi_1}^2}}}} \alpha(q,q',q_{\overline{\chi_1}^2}) \sum_{f \in \BB_{\hol}^{\ast}(q',\overline{\chi_1}^2)} \frac{L^q(w,F \otimes f \otimes \chi_1)}{L^q(1,\ad f)} h^{\hol}(k_f)
\end{multline}
is equal to the sum of the main term
\begin{equation}
\label{eqn:absolutemain}
q L^q(2w,\widetilde{F}) \left(\frac{1}{2\pi^2} \int_{-\infty}^{\infty} h(r) r \tanh\pi r \, dr + \sum_{\substack{k = 2 \\ k \equiv 0 \hspace{-.25cm} \pmod{2}}}^{\infty} \frac{k - 1}{2\pi^2} h^{\hol}(k)\right)
\end{equation}
and of the dual moment
\begin{multline}
\label{eqn:absolutedual}
\frac{1}{\varphi(q)} \sum_{\psi \hspace{-.25cm} \pmod{q}} \frac{1}{2\pi i} \int_{x_1 - i\infty}^{x_1 + i\infty} L\left(\frac{1}{2} + z, \widetilde{F} \otimes \psi\right) L\left(2w - \frac{1}{2} - z,\overline{\psi}\right) \ZZ_{\chi}(\psi;w,z)	\\
\times \sum_{\pm} \psi(\mp 1) \HH_{\mu_F}^{\pm}(w,z) \, dz,
\end{multline}
where $1/2 < x_1 < 5/2 - u$, $\ZZ_{\chi}(\psi;w,z)$ is as in \eqref{eqn:ZZchipsiwzdefeq}, and for $1 - 2u - 2x_1 < \sigma_2 < 2 - 2u$,
\begin{equation}
\label{eqn:HHwzdefeq}
\HH_{\mu_F}^{\pm}(w,z) \coloneqq \frac{1}{2\pi i} \int_{\sigma_2 - i\infty}^{\sigma_2 + i\infty} \widehat{H}(s) \Gscr_{\mu_F}^{\pm}\left(1 - \frac{s}{2} - w\right) G^{\mp}\left(\frac{s}{2} + w - \frac{1}{2} + z\right) \, ds,
\end{equation}
with $\Gscr_{\mu_F}^{\pm}$ as in \eqref{eqn:Gscrmupm} and $G^{\pm}$ as in \eqref{eqn:Gpm}.
\end{proposition}

\begin{proof}
We sum together the Kuznetsov and Petersson formul\ae{}, \eqref{eqn:Kuznetsov} and \eqref{eqn:Petersson}, with $q = q_1 q_2$, $m = 1$, and $\chi$ replaced by the primitive Dirichlet character inducing $\overline{\chi_1}^2$, which by abuse of notation we also denote by $\overline{\chi_1}^2$. We then multiply through by $A_F(\ell,n) \chi_1(n) \ell^{-2w} n^{-w}$, with $w = u + iv$ such that $u > 1$, and sum over $\ell,n \in \N$ with $(\ell n,q) = 1$. As $\Re(w) > 1$, the ensuing expression converges absolutely, which allows us to interchange the order of summation.

For $f \in \BB_0^{\ast}(q',\overline{\chi_1}^2)$ or $f \in \BB_{\hol}^{\ast}(q',\overline{\chi_1}^2)$ with $q' \mid q$ such that $q' \equiv 0 \pmod{q_{\overline{\chi_1}^2}}$, we have that\footnote{Note that the condition $(\ell,q) = 1$ is mistakenly omitted in several previous works, such as \cite[Section 2]{Blo12}, \cite[Lemma 2.1]{Hua21}, and \cite[(1.2)]{Qi19}.}
\[\sum_{\substack{\ell,n = 1 \\ (\ell n, q) = 1}}^{\infty} \frac{A_F(\ell,n) \lambda_f(n) \chi_1(n)}{\ell^{2w} n^w} = L^q(w,F \otimes f \otimes \chi_1).\]
Similarly, for $\psi_1,\psi_2$ primitive Dirichlet characters modulo $q_{\psi_1},q_{\psi_2}$ satisfying $\psi_1 \psi_2 = \overline{\chi_1}^2$ and $q_{\psi_1} q_{\psi_2} = q'$, we have that
\[\sum_{\substack{\ell,n = 1 \\ (\ell n, q) = 1}}^{\infty} \frac{A_F(\ell,n) \lambda_{\psi_1,\psi_2}(n,t) \chi_1(n)}{\ell^{2w} n^w} = L^q(w + it,F \otimes \psi_1 \chi_1) L^q(w - it,F \otimes \overline{\psi_1 \chi_1}).\]
From these identities, the left-hand sides of the Kuznetsov and Petersson formul\ae{} give us \eqref{eqn:absoluteLHS}. The diagonal terms are equal to the main term \eqref{eqn:absolutemain}, since
\[\sum_{\substack{\ell = 1 \\ (\ell, q) = 1}}^{\infty} \frac{A_F(\ell,1)}{\ell^{2w}} = L^q(2w,\widetilde{F}).\]
Thus it remains to show that the Kloosterman terms are equal to the dual moment \eqref{eqn:absolutedual}.

After applying the Mellin inversion formula to the function $H(x)$ and interchanging the order of integration and summation, we deduce that the Kloosterman terms are equal to
\begin{equation}
\label{eqn:Kloostermanterm}
\frac{1}{2\pi i} \int_{\sigma_0 - i\infty}^{\sigma_0 + i\infty} \widehat{H}(s) q \sum_{\substack{c,\ell = 1 \\ c \equiv 0 \hspace{-.25cm} \pmod{q} \\ (\ell,q) = 1}}^{\infty} \frac{c^{s - 1}}{\ell^{2w}} \sum_{\substack{n = 1 \\ (n,q) = 1}}^{\infty} \frac{A_F(\ell,n) \chi_1(n)}{n^{\frac{s}{2} + w}} S_{\overline{\chi_1}^2}(1,n;c) \, ds.
\end{equation}
This identity is valid so long as $2 - 2u < \sigma_0 < -1$, which requires that $u > 3/2$. Indeed, the process of Mellin inversion is valid for $-5 < \Re(s) < 0$ via \eqref{eqn:HpmMellindecay}. The trivial bound $|S_{\overline{\chi_1}^2}(1,n;c)| \leq \varphi(c)$ for the Kloosterman sum ensures that the sum over $c$ converges absolutely since $\sigma_0 < -1$. The sum over $n$ converges absolutely since $\sigma_0 > 2 - 2u$. Finally, the sum over $\ell$ converges absolutely since $u > 1$.

We now restrict our attention to the integrand in \eqref{eqn:Kloostermanterm}. We make the change of variables $c \mapsto cq$ and break the sum over $n \in \N$ into residue classes $a$ modulo $cq$, yielding
\begin{multline*}
q \sum_{\substack{c,\ell = 1 \\ c \equiv 0 \hspace{-.25cm} \pmod{q} \\ (\ell,q) = 1}}^{\infty} \frac{c^{s - 1}}{\ell^{2w}} \sum_{\substack{n = 1 \\ (n,q) = 1}}^{\infty} \frac{A_F(\ell,n) \chi_1(n)}{n^{\frac{s}{2} + w}} S_{\overline{\chi_1}^2}(1,n;c)	\\
= q^s \sum_{\substack{c,\ell = 1 \\ (\ell,q) = 1}}^{\infty} \frac{c^{s - 1}}{\ell^{2w}} \sum_{\substack{a \in \Z/cq\Z \\ (a,q) = 1}} \chi_1(a) S_{\overline{\chi_1}^2}(1,a;cq) \sum_{\substack{n = 1 \\ n \equiv a \hspace{-.25cm} \pmod{cq}}}^{\infty} \frac{A_F(\ell,n)}{n^{\frac{s}{2} + w}}.
\end{multline*}
The condition $n \equiv a \pmod{cq}$ can be enforced by inserting the double sum
\[\frac{1}{cq} \sum_{c_1 \mid cq} \sum_{b \in (\Z/c_1\Z)^{\times}} e\left(\frac{(a - n)\overline{b}}{c_1}\right).\]
Recalling the definition \eqref{eqn:PhiFdefeq} of the Vorono\u{\i} series $\Phi_F$, this leads us to the expression
\begin{equation}
\label{eqn:Kloostermanintegrand}
q^{s - 1} \sum_{\substack{c,\ell = 1 \\ (\ell,q) = 1}}^{\infty} \frac{c^{s - 2}}{\ell^{2w}} \sum_{\substack{a \in \Z/cq\Z \\ (a,q) = 1}} \chi_1(a) S_{\overline{\chi_1}^2}(1,a;cq) \sum_{c_1 \mid cq} \sum_{b \in (\Z/c_1\Z)^{\times}} e\left(\frac{a\overline{b}}{c_1}\right) \Phi_F\left(c_1,-b,\ell;\frac{s}{2} + w\right).
\end{equation}

We insert the expression \eqref{eqn:Kloostermanintegrand} back into \eqref{eqn:Kloostermanterm} and shift the contour to $\Re(s) = \sigma_1$ with $-5 < \sigma_1 < -2u - 1$; this process is valid so long as $7/4 < u < 2$, since the ensuing sums over $c,\ell \in \N$ and integral over $\Re(s) = \sigma_1$ converge absolutely via \eqref{eqn:PhiPL} and \eqref{eqn:HpmMellindecay}. We then apply the Vorono\u{\i} summation formula \eqref{eqn:GL3Voronoi}, which shows that \eqref{eqn:Kloostermanterm} is equal to
\begin{multline}
\label{eqn:postVoronoitotal}
\frac{1}{2\pi i} \int_{\sigma_1 - i\infty}^{\sigma_1 + i\infty} \widehat{H}(s) \sum_{\pm} \Gscr_{\mu_F}^{\pm}\left(1 - \frac{s}{2} - w\right) q^{s - 1}	\\
\times \sum_{\substack{c,\ell = 1 \\ (\ell,q) = 1}}^{\infty} \frac{c^{s - 2}}{\ell^{2w}} \sum_{\substack{a \in \Z/cq\Z \\ (a,q) = 1}} \chi_1(a) S_{\overline{\chi_1}^2}(1,a;cq) \sum_{c_1 \mid cq} \sum_{b \in (\Z/c_1\Z)^{\times}} e\left(\frac{a\overline{b}}{c_1}\right) \Xi_F\left(c_1,\mp b,\ell;-\frac{s}{2} - w\right) \, ds,
\end{multline}
where $\Gscr_{\mu_F}^{\pm}$ is as in \eqref{eqn:Gscrmupm}. We write $c = c' c_0$, where $(c',q) = 1$ and $c_0 \mid q^{\infty}$. In anticipation of future simplifications, we let $\ell' = c' \ell$, so that the last line of \eqref{eqn:postVoronoitotal} becomes
\begin{multline}
\label{eqn:postVoronoi1}
\sum_{\substack{\ell' = 1 \\ (\ell',q) = 1}}^{\infty} \frac{1}{{\ell'}^{2w}} \sum_{c_0 \mid q^{\infty}} c_0^{s - 2} \sum_{c' \mid \ell'} {c'}^{s + 2w - 2} \sum_{\substack{a \in \Z/c' c_0 q\Z \\ (a,q) = 1}} \chi_1(a) S_{\overline{\chi_1}^2}(1,a;c' c_0 q)	\\
\times \sum_{c_1 \mid c' c_0 q} \sum_{b \in (\Z/c_1\Z)^{\times}} e\left(\frac{a\overline{b}}{c_1}\right) \Xi_F\left(c_1,\mp b,\frac{\ell'}{c'};-\frac{s}{2} - w\right).
\end{multline}
The condition $(a,q) = 1$ ensures that we may replace $\chi_1$ with the Dirichlet character $\chi \coloneqq \chi_1 \chi_{0(q_2)}$ modulo $q$, so that \eqref{eqn:postVoronoi1} is precisely the left-hand side of \eqref{eqn:Voronoiidentity}.

We insert the identity \eqref{eqn:Voronoiidentity} for \eqref{eqn:postVoronoi1} back into \eqref{eqn:postVoronoitotal}. The ensuing expression for the Kloosterman term is
\begin{multline*}
\frac{1}{\varphi(q)} \sum_{\psi \hspace{-.25cm} \pmod{q}} \frac{1}{2\pi i} \int_{\sigma_1 - i\infty}^{\sigma_1 + i\infty} \widehat{H}(s) \sum_{\pm} \Gscr_{\mu_F}^{\pm}\left(1 - \frac{s}{2} - w\right)	\\
\times \frac{1}{2\pi i} \int_{\CC_1} L\left(\frac{1}{2} + z, \widetilde{F} \otimes \psi\right) L\left(2w - \frac{1}{2} - z,\overline{\psi}\right) \ZZ_{\chi}(\psi;w,z)	\\
\times \psi(\mp 1) G^{\mp}\left(\frac{s}{2} + w - \frac{1}{2} + z\right) \, dz \, ds.
\end{multline*}
Here $-5 < \sigma_1 < -2u - 1$ and $\CC_1$ is the contour as in \hyperref[lem:Voronoiidentity]{Lemma \ref*{lem:Voronoiidentity}}. We straighten the inner contour to the line $\Re(z) = x_1$, where $1/2 < x_1 < -\sigma_1/2 - u$; in doing so, we pick up a residue at $z = 1/2 - s/2 - w$ due to the pole of $G^{\mp}$. We obtain the expression
\begin{multline}
\label{eqn:absolutedualshift}
\frac{1}{\varphi(q)} \sum_{\psi \hspace{-.25cm} \pmod{q}} \frac{1}{2\pi i} \int_{\sigma_1 - i\infty}^{\sigma_1 + i\infty} \widehat{H}(s) \sum_{\pm} \psi(\mp 1) \Gscr_{\mu_F}^{\pm}\left(1 - \frac{s}{2} - w\right)	\\
\times \frac{1}{2\pi i} \int_{x_1 - i\infty}^{x_1 + i\infty} L\left(\frac{1}{2} + z, \widetilde{F} \otimes \psi\right) L\left(2w - \frac{1}{2} - z,\overline{\psi}\right) \ZZ_{\chi}(\psi;w,z) G^{\mp}\left(\frac{s}{2} + w - \frac{1}{2} + z\right) \, dz \, ds \\
+ \frac{1}{\varphi(q)} \sum_{\psi \hspace{-.25cm} \pmod{q}} \frac{1}{2\pi i} \int_{\sigma_1 - i\infty}^{\sigma_1 + i\infty} L\left(1 - \frac{s}{2} - w, \widetilde{F} \otimes \psi\right) L\left(\frac{s}{2} + 3w - 1,\overline{\psi}\right) \ZZ_{\chi}\left(\psi;w,\frac{1}{2} - \frac{s}{2} - w\right)	\\
\times \widehat{H}(s) \sum_{\pm} \psi(\mp 1) \Gscr_{\mu_F}^{\pm}\left(1 - \frac{s}{2} - w\right) \, ds.
\end{multline}
For the first term in \eqref{eqn:absolutedualshift}, we observe that this double integral is absolutely convergent via the bounds \eqref{eqn:Gpmbounds}, \eqref{eqn:Gscrmubounds}, and \eqref{eqn:HpmMellindecay}. Thus we may interchange the order of integration, so that now $1/2 < x_1 < 5/2 - u$ and $-5 < \sigma_1 < -2u - 2x_1$. We may then shift the inner contour of integration to the line $\Re(s) = \sigma_2$ with $1 - 2u - 2x_1 < \sigma_2 < 2 - 2u$; in doing so, we pick up a residue at $s = 1 - 2w - 2z$ due to the pole of $G^{\mp}$. This residue, however, is exactly cancelled out by the second term in \eqref{eqn:absolutedualshift} once we make the change of variables $s = 1 - 2w - 2z$ in the latter. In this way, we arrive at the identity \eqref{eqn:absolutedual} for the Kloosterman term.
\end{proof}

\section{Character Sums I}
\label{sect:charsumsI}

Our next goal is to analytically continue the $\GL_3 \times \GL_2 \leftrightsquigarrow \GL_4 \times \GL_1$ spectral reciprocity formula derived in \hyperref[prop:absconv]{Proposition \ref*{prop:absconv}} to the central value $w = 1/2$. In order to do so, we require some stringent control over the behaviour of $\ZZ_{\chi}(\psi_{0(q)};w,2w - 3/2)$ in order to precisely extract a secondary main term arising from shifting the contour in \eqref{eqn:absolutedual} and picking up a residue at $z = 2w - 3/2$ when $\psi$ is the principal character $\psi_{0(q)}$ modulo $q$.

To begin, we highlight some properties of the character sum \eqref{eqn:VVchidefeq}.

\begin{lemma}[Cf.~{\cite[Section 5.2]{PY20}}]\hspace{1em}
\label{lem:VVchiproperties}
\begin{enumerate}[leftmargin=*,label=\textup{(\arabic*)}]
\item\label{lemno:VVchicoprime} If $(m_1' m_2' m_3' r',q) = 1$, then
\[\VV_{\chi}(\psi;m_1 m_1',m_2 m_2',m_3 m_3',r r') = \psi(m_1' m_2' m_3') \overline{\psi}(r') \VV_{\chi}(\psi;m_1,m_2,m_3,r).\]
\item\label{lemno:VVchimult} If $q = q_1 q_2$ with $(q_1,q_2) = 1$, so that $\chi = \chi_1 \chi_2$ and $\psi = \psi_1 \psi_2$ with $\chi_1,\psi_1$ characters modulo $q_1$ and $\chi_2,\psi_2$ characters modulo $q_2$, and if $m_j = m_{j,1} m_{j,2}$, and $r = r_1 r_2$ with $m_{j,1},r_1 \mid q_1^{\infty}$ and $m_{j,2}, r_2 \mid q_2^{\infty}$, then
\begin{multline*}
\VV_{\chi}(\psi;m_1,m_2,m_3,r) = \psi_1(m_{1,2} m_{2,2} m_{3,2}) \overline{\psi_1}(q_2 r_2) \psi_2(m_{1,1} m_{2,1} m_{3,1}) \overline{\psi_2}(q_1 r_1)	\\
\times \VV_{\chi_1}(\psi_1;m_{1,1},m_{2,1},m_{3,1},r_1) \VV_{\chi_2}(\psi_2;m_{1,2},m_{2,2},m_{3,2},r_2).
\end{multline*}
\item\label{lemno:VVchig} If $\chi,\psi$ are primitive, then $\VV_{\chi}(\psi;1,1,1,1) = \chi(-1) \tau(\overline{\psi}) g(\chi,\psi)$, where
\begin{equation}
\label{eqn:gchipsidefeq}
g(\chi,\psi) \coloneqq \sum_{t,u \in \Z/q\Z} \overline{\chi}(t) \chi(t + 1) \chi(u) \overline{\chi}(u + 1) \psi(ut - 1).
\end{equation}
\end{enumerate}
\end{lemma}

Here $\tau(\chi) \coloneqq \tau(\chi,1)$ denotes the (standard) Gauss sum. If $\chi$ is a primitive Dirichlet character modulo $q$, then $|\tau(\chi)| = \sqrt{q}$.

\begin{proof}[Proof of {\hyperref[lem:VVchiproperties]{Lemma \ref*{lem:VVchiproperties}}}]\hspace{1em}

\begin{enumerate}[leftmargin=*,label=\textup{(\arabic*)}]
\item This follows upon making the change of variables $t \mapsto m_1' m_2' \overline{r'} t$ and $u \mapsto m_1' \overline{r'} u$ and recalling \eqref{eqn:GausssumCOV}.
\item This follows upon writing $t = q_1 t_2 + q_2 t_1$ and $u = q_1 u_2 + q_2 u_1$, making the change of variables $t_1 \mapsto m_{1,2} m_{2,2} \overline{q_2 r_2} t_1$ and $u_1 \mapsto m_{1,2} \overline{q_2 r_2} u_1$, and recalling \eqref{eqn:GausssumCOV}.
\item This follows upon making the change of variables $u \mapsto u + 1$ and $t \mapsto ut - 1$, noting that if $\chi$ is primitive, then the generalised Gauss sum $\tau(\chi,n)$ as in \eqref{eqn:tauchiadefeq} satisfies $\tau(\chi,n) = \overline{\chi}(n) \tau(\chi,1)$ for \emph{all} $n \in \Z$, together with the fact that $|\tau(\chi)| = \sqrt{q}$.\qedhere
\end{enumerate}
\end{proof}

From \hyperref[lemno:VVchicoprime]{Lemma \ref*{lem:VVchiproperties} \ref*{lemno:VVchicoprime}} and \ref{lemno:VVchimult}, in order to further analyse properties of the character sum \eqref{eqn:VVchidefeq}, it suffices to suppose that $q = p^{\beta}$ and that $m_1,m_2,m_3,r \mid p^{\infty}$. For our applications, $\chi$ is either principal or primitive, and we have that $m_3 \mid m_1$ and $m_1 \mid m_2$. The exact behaviour of this character sum depends delicately on the conductors of $\chi$ and $\psi$, and so we analyse this behaviour by treating each case separately.

\begin{lemma}
\label{lem:VVchiprincprinc}
Let $\chi_{0(p^{\beta})},\psi_{0(p^{\beta})}$ both be the principal character modulo $p^{\beta}$. Suppose that $m_3 \mid m_1$, $m_1 \mid m_2$, and $m_1,m_2,m_3,r \mid p^{\infty}$. Then
\[\VV_{\chi_{0(p^{\beta})}}(\psi_{0(p^{\beta})};m_1,m_2,m_3,r) = \begin{dcases*}
p^{-1} (p - 1)^3 & if $r = 1$, $p \mid m_1,m_2,m_3$, and $\beta = 1$,	\\
-p^{-1} (p - 1)^2 & if $m_3, r = 1$, $p \mid m_1,m_2$, and $\beta = 1$,	\\
p^{-1} (p - 1) & if $m_1, m_3, r = 1$, $p \mid m_2$, and $\beta = 1$,	\\
p^{2\beta - 1} (p - 1) & if $m_1, m_2, m_3 = 1$ and $p \mid r$,	\\
p^{-1} (p^3 - p^2 - p - 1) & if $m_1, m_2, m_3, r = 1$ and $\beta = 1$,	\\
p^{2\beta - 1}(p - 1) & if $m_1, m_2, m_3, r = 1$ and $\beta \geq 2$,	\\
0 & otherwise.
\end{dcases*}\]
In particular, $\VV_{\chi_{0(p^{\beta})}}(\psi_{0(p^{\beta})};m_1,m_2,m_3,r) \ll p^{2\beta}$.
\end{lemma}

\begin{proof}
The character sum of interest is
\begin{multline}
\label{eqn:VVchiprincprincdefeq}
\VV_{\chi_{0(p^{\beta})}}(\psi_{0(p^{\beta})};m_1,m_2,m_3,r)	\\
= \frac{1}{p^{\beta}} \sum_{t,u \in \Z/p^{\beta}\Z} R_{p^{\beta}}(t + m_2 u) \chi_{0(p^{\beta})}(rt + m_1 m_2) R_{p^{\beta}}(u) \chi_{0(p^{\beta})}(ru - m_1) R_{p^{\beta}}(m_3 t),
\end{multline}
where $R_q(n)$ denotes the Ramanujan sum, as in \eqref{eqn:Rqndefeq}. To determine its exact value, we must treat this on a case-by-case basis. Below, we freely use the fact that
\begin{equation}
\label{eqn:Rpbetapalphadefeq}
R_{p^{\beta}}(p^{\alpha}) = \begin{dcases*}
0 & if $0 \leq \alpha \leq \beta - 2$,	\\
-p^{\beta - 1} & if $\alpha = \beta - 1$,	\\
p^{\beta - 1}(p - 1) & if $\alpha \geq \beta$.
\end{dcases*}
\end{equation}
\begin{itemize}[leftmargin=*]
\item If $m_1,m_2,m_3 \equiv 0 \pmod{p}$, then the summand in \eqref{eqn:VVchiprincprincdefeq} vanishes unless $r = 1$ and $(tu,p) = 1$, in which case it is
\[\frac{1}{p^{\beta}} \sum_{t,u \in (\Z/p^{\beta}\Z)^{\times}} R_{p^{\beta}}(t + m_2 u) R_{p^{\beta}}(u) R_{p^{\beta}}(m_3 t).\]
The second term in the sum vanishes unless $u \equiv 0 \pmod{p^{\beta - 1}}$, which can only occur if $\beta = 1$. In this case, the first term and second terms are both $-1$, while the third is $p - 1$, and so we obtain $p^{-1} (p - 1)^3$.
\item If $m_1,m_2 \equiv 0 \pmod{p}$ and $m_3 = 1$, then we follow the same argument above except the third term in the sum is $-1$, and so we obtain $-p^{-1} (p - 1)^2$.
\item If $m_1, m_3 = 1$ and $m_2 \equiv 0 \pmod{p}$, then the summand vanishes unless $r = 1$ and $(t,p) = 1$. We make the change of variables $t \mapsto \overline{r} t$ and $u \mapsto \overline{r} (u + 1)$, yielding
\[\frac{1}{p^{\beta}} \sum_{t,u \in (\Z/p^{\beta}\Z)^{\times}} R_{p^{\beta}}(t + m_2 u + m_2) R_{p^{\beta}}(u + 1) R_{p^{\beta}}(t).\]
The third term in the sum vanishes unless $\beta = 1$, in which case the first and third terms are $-1$, while the second is $-1$ unless $u = p - 1$, in which case it is $p - 1$, and so we obtain $p^{-1} (p - 1)$.
\item If $m_1, m_2, m_3 = 1$ and $r \equiv 0 \pmod{p}$, this is
\[\frac{1}{p^{\beta}} \sum_{t,u \in \Z/p^{\beta}\Z} R_{p^{\beta}}(t - u) R_{p^{\beta}}(u) R_{p^{\beta}}(t).\]
The summand vanishes unless $t,u \equiv 0 \pmod{p^{\beta - 1}}$. Making the change of variables $t \mapsto p^{\beta - 1}t$ and $u \mapsto p^{\beta - 1}u$, this becomes
\[p^{2\beta - 3} \sum_{t,u \in \Z/p\Z} R_p(t - u) R_p(u) R_p(t),\]
which is $p^{2\beta - 1} (p - 1)$.
\item If $m_1, m_2, m_3, r = 1$, then we instead make the change of variables $t \mapsto t - 1$ and $u \mapsto 1 - u$, yielding
\[\frac{1}{p^{\beta}} \sum_{t,u \in (\Z/p^{\beta}\Z)^{\times}} R_{p^{\beta}}(t - u) R_{p^{\beta}}(u - 1) R_{p^{\beta}}(t - 1).\]
For $\beta = 1$, this is $p^{-1} (p^3 - p^2 - p - 1)$. For $\beta \geq 2$, the summand vanishes unless $t,u \equiv 1 \pmod{p^{\beta - 1}}$. Making the change of variables $t \mapsto 1 + p^{\beta - 1}t$ and $u \mapsto 1 + p^{\beta - 1}u$, this becomes
\[p^{2\beta - 3} \sum_{t,u \in \Z/p\Z} R_p(t - u) R_p(u) R_p(t),\]
which is again $p^{2\beta - 1} (p - 1)$.\qedhere
\end{itemize}
\end{proof}

\begin{lemma}
\label{lem:VVchiprincnonprinc}
Let $\chi_{0(p^{\beta})}$ be the principal character modulo $p^{\beta}$ and let $\psi_{p^{\beta}}$ be a nonprincipal character modulo $p^{\beta}$. Suppose that $m_3 \mid m_1$, $m_1 \mid m_2$, and $m_1,m_2,m_3,r \mid p^{\infty}$. Then
\[\VV_{\chi_{0(p^{\beta})}}(\psi_{p^{\beta}};m_1,m_2,m_3,r) = \begin{dcases*}
\overline{\tau(\psi_p)} p^{-1} (p + 1) & if $m_1, m_2, m_3, r = 1$ and $\beta = 1$,	\\
0 & otherwise.
\end{dcases*}\]
In particular, $\VV_{\chi_{0(p^{\beta})}}(\psi_{p^{\beta}};m_1,m_2,m_3,r) \ll p^{\beta/2}$.
\end{lemma}

\begin{proof}
The proof follows the same lines as that of \hyperref[lem:VVchiprincprinc]{Lemma \ref*{lem:VVchiprincprinc}} (namely a case-by-case treatment) except that in place of the Ramanujan sum $R_{p^{\beta}}(m_3 t)$ in \eqref{eqn:VVchiprincprincdefeq}, we instead have the generalised Gauss sum $\tau(\overline{\psi_{p^{\beta}}},m_3 t)$ as in \eqref{eqn:tauchiadefeq}. The strategy is identical except that we use character orthogonality for the nonprincipal character $\psi_{p^{\beta}}$ as well as the fact that if $\chi$ modulo $q$ is induced from a primitive Dirichlet character $\chi^{\star}$ modulo $d$ for some $d \mid q$, then
\[\tau(\chi,a) = \begin{dcases*}
\overline{\chi}^{\star}\left(\frac{a}{(q,a)}\right) \chi^{\star}\left(\frac{q}{d(q,a)}\right) \mu\left(\frac{q}{d(q,a)}\right) \frac{\varphi(q)}{\varphi\left(\frac{q}{(q,a)}\right)} \tau(\chi^{\star}) & if $d \mid \frac{q}{(q,a)}$,	\\
0 & otherwise.
\end{dcases*}\]
\par\vspace{-1.33\baselineskip}\qedhere
\end{proof}

When $\chi_{p^{\beta}}$ is primitive, so that $\tau(\chi_{p^{\beta}},n) = \overline{\chi_{p^{\beta}}}(n) \tau(\chi_{p^{\beta}})$ for \emph{any} $n \in \Z$, we have that $\VV_{\chi_{p^{\beta}}}(\psi_{p^{\beta}};m_1,m_2,m_3,r) = \chi_{p^{\beta}}(-1) \widehat{H}(\psi_{p^{\beta}},\chi_{p^{\beta}},m_1,m_2,m_3,r)$, where $\widehat{H}(\psi,\chi,m_1,m_2,m_3,r)$ is the character sum as in \cite[(5.13)]{PY20}, namely
\[\widehat{H}(\psi,\chi,m_1,m_2,m_3,r) \coloneqq \sum_{t,u,v \in \Z/q\Z} \chi(t + m_2 u) \overline{\chi}(rt + m_1 m_2) \overline{\chi}(u) \chi(ru - m_1) \overline{\psi}(v) e\left(\frac{m_3 tv}{q}\right).\]
We may therefore appeal to earlier work of Petrow and Young to determine the behaviour of the character sum \eqref{eqn:VVchidefeq} in this setting.

\begin{lemma}[Petrow--Young {\cite{PY20}}]
\label{lem:VVchiprimprinc}
Let $\chi_{p^{\beta}}$ be a primitive character modulo $p^{\beta}$ and let $\psi_{0(p^{\beta})}$ be the principal character modulo $p^{\beta}$. Suppose that $m_3 \mid m_1$, $m_1 \mid m_2$, and $m_1,m_2,m_3,r \mid p^{\infty}$. Then
\begin{multline*}
\VV_{\chi_{p^{\beta}}}(\psi_{0(p^{\beta})}; m_1,m_2,m_3,r)	\\
= \begin{dcases*}
\chi_{p^{\beta}}(-1) p^{3\beta - 3} (p - 1)^3 & if $r = 1$ and $p^{\beta} \mid m_1,m_2,m_3$,	\\
-\chi_{p^{\beta}}(-1) p^{3\beta - 3} (p - 1)^2 & if $r = 1$, $p^{\beta - 1} \parallel m_3$, and $p^{\beta} \mid m_1,m_2$,	\\
\chi_{p^{\beta}}(-1) p^{3\beta - 3} (p - 1) & if $r = 1$, $p^{\beta - 1} \parallel m_1,m_3$, and $p^{\beta} \mid m_2$,	\\
-\chi_{p^{\beta}}(-1) p^{3\beta - 3} & if $r = 1$, $p^{\beta - 1} \parallel m_1,m_2,m_3$, and $\beta \geq 2$,	\\
p^{2\beta - 1}(p - 1) & if $m_1, m_2, m_3 = 1$ and $p^{\beta} \mid r$,	\\
-p^{2\beta - 1} & if $m_1, m_2, m_3 = 1$, $p^{\beta - 1} \parallel r$, and $\beta \geq 2$,	\\
-p - \chi_p(-1) & if $m_1, m_2, m_3, r = 1$ and $\beta = 1$,	\\
0 & otherwise.
\end{dcases*}
\end{multline*}
In particular, $\VV_{\chi_{p^{\beta}}}(\psi_{0(p^{\beta})}; m_1,m_2,m_3,r) \ll p^{3\beta}$.
\end{lemma}

\begin{proof}
It is shown in \cite[Lemma 6.5]{PY20} that for $m_1,m_2,m_3,r \mid p^{\infty}$,
\begin{multline*}
\VV_{\chi_{p^{\beta}}}(\psi_{0(p^{\beta})}; m_1,m_2,m_3,r)	\\
= p^{\beta} R_{p^{\beta}}(r) \psi_{0(p^{\beta})}(m_1 m_2 m_3) + \chi_{p^{\beta}}(-1) \psi_{0(p^{\beta})}(r) R_{p^{\beta}}(m_1) R_{p^{\beta}}(m_2) R_{p^{\beta}}(m_3),
\end{multline*}
which implies the result by \eqref{eqn:Rpbetapalphadefeq}.
\end{proof}

When $\chi$ is primitive and $\psi$ is nonprincipal, we do not require exact identities in all cases; upper bounds suffice. When both $\chi$ and $\psi$ are primitive, we have an exact identity involving the character sum $g(\chi,\psi)$ given by \eqref{eqn:gchipsidefeq}.

\begin{lemma}[Petrow--Young {\cite{PY20,PY23}}]
\label{lem:VVchiprimnonprinc}
Let $\chi_{p^{\beta}}$ be a primitive character modulo $p^{\beta}$ and let $\psi_{p^{\beta}}$ be a nonprincipal character modulo $p^{\beta}$ of conductor $p^{\alpha}$ for some $\alpha \in \{1,\ldots,\beta\}$. Suppose that $m_3 \mid m_1$, $m_1 \mid m_2$, and $m_1,m_2,m_3,r \mid p^{\infty}$. Then
\[\VV_{\chi_{p^{\beta}}}(\psi_{p^{\beta}};m_1,m_2,m_3,r) = \begin{dcases*}
\chi_{p^{\beta}}(-1) \tau(\overline{\psi_{p^{\beta}}}) g(\chi_{p^{\beta}},\psi_{p^{\beta}}) \hspace{-1.1cm} & \hspace{1.1cm} if $m_1, m_2, m_3, r = 1$ and $\alpha = \beta$,	\\
O\left(p^{2\beta - \frac{\alpha}{2}}\right) & if $m_1, m_2, m_3 = 1$, $p^{\beta - \alpha} \parallel r$, and $1 \leq \alpha < \beta$,	\\
O\left(p^{3\beta - \frac{3\alpha}{2}}\right) & if $r = 1$, $p^{\beta - \alpha} \parallel m_1, m_2, m_3$, and $1 \leq \alpha < \beta$,	\\
0 & otherwise.
\end{dcases*}\]
\end{lemma}

\begin{proof}
If $\alpha = \beta$, so that $\psi_{p^{\beta}}$ is primitive, this follows from \cite[Lemma 6.4]{PY20}. If $\psi_{p^{\beta}}$ is imprimitive but nonprincipal, so that $\beta \geq 2$ and $\alpha \in \{1,\ldots,\beta - 1\}$, this follows from \cite[Lemma 6.8]{PY20}, which is dependent on a conjecture stated in \cite[Conjecture 6.6]{PY20} and proven in \cite[Lemma 2.8]{PY23}.
\end{proof}

We now focus on $\ZZ_{\chi}(\psi_{0(q)};w,2w - 3/2)$. From \hyperref[lemno:VVchimult]{Lemma \ref*{lem:VVchiproperties} \ref*{lemno:VVchimult}}, we have that
\begin{equation}
\label{eqn:ZZchipsiww-32defeq1}
\ZZ_{\chi}\left(\psi_{0(q)};w,2w - \frac{3}{2}\right) = \prod_{p^{\beta} \parallel q} \ZZ_{\chi_{p^{\beta}}}\left(\psi_{0(p^{\beta})};w,2w - \frac{3}{2}\right),
\end{equation}
where we have factorised $\chi$ and $\psi_{0(q)}$ as a product of Dirichlet characters $\chi_{p^{\beta}}$ and $\psi_{0(p^{\beta})}$ modulo $p^{\beta}$. Here
\begin{multline}
\label{eqn:ZZchipsiww-32defeq2}
\ZZ_{\chi_{p^{\beta}}}\left(\psi_{0(p^{\beta})};w,2w - \frac{3}{2}\right) = \sum_{c_0 \mid p^{\infty}} \frac{1}{\varphi(c_0 p^{\beta})^2 c_0} \sum_{\substack{c_{1,0} c_{2,0} d_0 n_{1,0} = p^{\beta} \\ (c_{2,0} d_0 n_{1,0},c_0) = 1}} \frac{\varphi(c_0 c_{1,0}) \varphi(c_0 c_{1,0} d_0 n_{1,0}) \mu(d_0)}{c_{1,0} c_{2,0}^{6w - 4} d_0^{6w - 3}}	\\
\times \frac{A_F(1,n_{1,0})}{n_{1,0}^{4w - 2}} \sum_{\substack{n_{2,0} \mid p^{\infty} \\ (n_{2,0},c_0) = 1}} \frac{A_F(n_{2,0},1)}{n_{2,0}^{2w - 1}} \VV_{\chi_{p^{\beta}}}(\psi_{0(p^{\beta})}; c_{2,0} d_0 n_{1,0}, c_{2,0} d_0^2 n_{1,0} n_{2,0}, c_{2,0}, c_0).
\end{multline}
Again, the behaviour of this depends on whether $\chi_{p^{\beta}}$ is primitive or principal; moreover, it depends on whether $\beta = 1$ or $\beta \geq 2$.

\begin{lemma}\hspace{1em}
\label{lem:ZZchipsiww-32identity}
\begin{enumerate}[leftmargin=*,label=\textup{(\arabic*)}]
\item Let $\chi_{0(p^{\beta})},\psi_{0(p^{\beta})}$ both be the principal character modulo $p^{\beta}$. Then
\begin{multline*}
\ZZ_{\chi_{0(p^{\beta})}}\left(\psi_{0(p^{\beta})};w,2w - \frac{3}{2}\right)	\\
= \begin{dcases*}
p^{4 - 6w} L_p(2w - 1,\widetilde{F}) \left(1 - A_F(1,p) p^{2w - 2} + A_F(1,p) p^{2w - 3} + p^{6w - 5} - p^{-2} - p^{6w - 6}\right) \hspace{-1cm} & \\
\qquad + p^{-1}(p^2 - 2) & if $\beta = 1$,	\\
p^{\beta} & if $\beta \geq 2$.
\end{dcases*}
\end{multline*}
\item Let $\chi_{p^{\beta}}$ be a primitive character modulo $p^{\beta}$ and let $\psi_{0(p^{\beta})}$ be the principal character modulo $p^{\beta}$. Then
\begin{multline*}
\ZZ_{\chi_{p^{\beta}}}\left(\psi_{0(p^{\beta})};w,2w - \frac{3}{2}\right)	\\
= \chi_{p^{\beta}}(-1) p^{(5 - 6w)\beta} L_p(2w - 1,\widetilde{F}) \left(1 - A_F(1,p) p^{2w - 2} + A_F(1,p) p^{2w - 3} + p^{6w - 5} - p^{-2} - p^{6w - 6}\right)	\\
- \begin{dcases*}
\chi_p(-1) (1 - p^{-1}) & if $\beta = 1$,	\\
\chi_{p^{\beta}}(-1) p^{(5 - 6w)(\beta - 1)} - 1 + p^{2\beta} & if $\beta \geq 2$.
\end{dcases*}
\end{multline*}
\end{enumerate}
\end{lemma}

\begin{proof}
The desired identities follow by a case-by-case analysis using the identities in \hyperref[lem:VVchiprincprinc]{Lemmata \ref*{lem:VVchiprincprinc}} and \ref{lem:VVchiprimprinc} for the character sum \eqref{eqn:VVchidefeq} together with the identity \eqref{eqn:ZZchipsiww-32defeq2}.
\end{proof}

From this, we can precisely describe the behaviour of $\ZZ_{\chi}(\psi_{0(q)};w,2w - 3/2)$. This description is simplified under the assumption that $F$ is selfdual.

\begin{corollary}
\label{cor:ZZchipsi0hol}
Let $F$ be a selfdual Hecke--Maa\ss{} cusp form for $\SL_3(\Z)$. Let $q = q_1 q_2$ be a positive integer with $(q_1,q_2) = 1$. Let $\chi_1$ be a primitive Dirichlet character of conductor $q_1$, and set $\chi \coloneqq \chi_1 \chi_{0(q_2)}$. Then the quotient $\ZZ_{\chi}(\psi_{0(q)};w,2w - 3/2) / L_q(2w - 1,\widetilde{F})$ extends holomorphically to an entire function, and we have that
\[\lim_{w \to \frac{1}{2}} \frac{\ZZ_{\chi}\left(\psi_{0(q)};w,2w - \frac{3}{2}\right)}{L_q(2w - 1,\widetilde{F})} = \begin{dcases*}
\frac{\chi_1(-1) q_1^2 q_2}{L_q(1,F)} & if $q_2$ is squarefree,	\\
0 & otherwise.
\end{dcases*}\]
\end{corollary}

\begin{proof}
Since $F$ is selfdual, we have that $A_F(1,p) = A_F(p,1)$ for every prime $p$, so that for $\Re(s) > 0$,
\[L_p(s,F) = L_p(s,\widetilde{F}) = \frac{1}{1 - A_F(1,p) p^{-s} + A_F(1,p) p^{-2s} - p^{-3s}}.\]
In particular, $\lim_{w \to 1/2} L_p(2w - 1,\widetilde{F})^{-1} = 0$. The result then follows from \eqref{eqn:ZZchipsiww-32defeq1} in conjunction with \hyperref[lem:ZZchipsiww-32identity]{Lemma \ref*{lem:ZZchipsiww-32identity}}.
\end{proof}

\section{Test Functions and Transforms I}
\label{sect:testfunction}

In the process of analytically continuing the $\GL_3 \times \GL_2 \leftrightsquigarrow \GL_4 \times \GL_1$ spectral reciprocity formula derived in \hyperref[prop:absconv]{Proposition \ref*{prop:absconv}} to the central value $w = 1/2$, we require good control over the transform $\HH_{\mu_F}^{\pm}(w,z)$ given by \eqref{eqn:HHwzdefeq}. In particular, we need to ensure that this expression is holomorphic in both $w$ and $z$ in certain regions and additionally ensure that it decays sufficiently rapidly in $|\Im(z)|$. These conditions are met once we enforce particular conditions on our tuple of test functions $(h,h^{\hol})$.

Following work of Blomer and Khan \cite[Section 6]{BK19}, we call a function $H : (0,\infty) \to \C$ \emph{admissible of type $(A,B)$} for some $A,B > 5$ if it is a linear combination of functions of the following two types:
\begin{enumerate}[leftmargin=*,label=\textup{(\arabic*)}]
\item $H(x) = (\Kscr h)(x)$ for some function $h$ that is even, holomorphic in the strip $|\Im(t)| < A$ with zeroes at $\pm(n + 1/2)i$ for $n \in \{0,\ldots,\lfloor A - 1/2\rfloor\}$, and satisfies the bound $h(t) \ll (1 + |t|)^{-B - 2}$ in this strip;
\item There exist constants $a,b \in \N$ with $a - b \geq A$ and $\alpha_0,\alpha_k \in \C$ for all $k \in 2\N$ with $k > a - b$ and $\alpha_k \ll k^{-B - 2}$ such that
\[H(x) = \alpha_0 \JJ_{a + 1}^{\hol}(x) x^{-b} + \sum_{\substack{k = 2 \\ k \equiv 0 \hspace{-.25cm} \pmod{2} \\ k > a - b}}^{\infty} \alpha_k \JJ_k^{\hol}(x).\]
\end{enumerate}

Similarly, let $h$ be an even function that is holomorphic in the strip $|\Im(t)| < 1/2 + \delta$ in which it satisfies $h(t) \ll (1 + |t|)^{-2 - \delta}$ for some $\delta > 0$ and let $h^{\hol} : 2\N \to \C$ be a sequence satisfying $h^{\hol}(k) \ll k^{-2 - \delta}$ for some $\delta > 0$. We say that the tuple of test functions $(h,h^{\hol})$ is \emph{admissible of type $(A,B)$} for some $A,B > 5$ if the transform $H : (0,\infty) \to \C$ given by \eqref{eqn:HdefeqKscr} is admissible of type $(A,B)$.

From \cite[Lemma 9]{BK19}, the Mellin transform $\widehat{H}(s)$ of an admissible function $H$ of type $(A,B)$ is holomorphic in the open half-plane $\Re(s) > -A$, in which it satisfies the bound
\begin{equation}
\label{eqn:hatHbounds}
\widehat{H}(s) \ll_{\sigma} (1 + |\tau|)^{\sigma - 1}.
\end{equation}
Furthermore, for $\sigma > -A$ and $|\tau| \geq 1$, there exists a smooth function $j_{\sigma}(\tau)$ satisfying the bound $|\tau|^m j_{\sigma}^{(m)}(\tau) \ll_{\sigma} 1$ for all nonnegative integers $m \leq B$ such that we have the asymptotic formula
\begin{equation}
\label{eqn:hatHasymp}
\widehat{H}(s) = |\tau|^{\sigma - 1} \exp\left(i\tau \log \frac{|\tau|}{4\pi e}\right) j_{\sigma}(\tau) + O_{\sigma}(|\tau|^{\sigma - 1 - B}).
\end{equation}
Using this, we show the following.

\begin{lemma}
\label{lem:HHpmwzbounds}
Let $H$ be admissible of type $(A,B)$ for some $A,B > 5$. Then for $w = u + iv$ and $z = x + iy$, the function $\HH_{\mu_F}^{\pm}(w,z)$ defined as in \eqref{eqn:HHwzdefeq} for $7/4 < u < 2$ and $1/2 < x < 5/2 - u$ extends holomorphically to $1/2 \leq u < 2$ and $x < 2u - 1/2$. Moreover, for $w$ lying in a compact subset $K$ of the vertical strip $1/2 \leq u < 2$, we have that
\begin{equation}
\label{eqn:HHpmwzdecay}
\HH_{\mu_F}^{\pm}(w,z) \ll_{\mu_F,K,\e} (1 + |y|)^{-\min\left\{A + \frac{1}{2},\frac{B}{4}\right\} + \e}.
\end{equation}
\end{lemma}

\begin{proof}
By \cite[Lemma 10]{BK19}, $\HH_{\mu_F}^{\pm}(w,z)$ is holomorphic for $x < 2u - 1/2$ and has a meromorphic continuation to $x < 2u + 1/2$ with a simple polar divisor at most at $z = 2w - 1/2$. To prove the desired bounds \eqref{eqn:HHpmwzdecay} for $\HH_{\mu_F}^{\pm}(w,z)$, we deal only with the case $\pm = +$; the case $\pm = -$ follows by a similar argument, noting that $\HH_{\mu_F}^{-}(w,z) = \overline{\HH_{\overline{\mu_F}}^{+}(\overline{w},\overline{z})}$. We must separately deal with the cases $y \leq -1$ and $y \geq 1$; note that we can estimate trivially for $-1 \leq y \leq 1$. We follow a strategy of Blomer and Khan \cite[Proof of Lemma 11]{BK19}.

For $y \leq -1$, we make the change of variables $s \mapsto 2s - 2w + 1 - 2z$ in \eqref{eqn:HHwzdefeq}, so that
\[\HH_{\mu_F}^{+}(w,z) = \frac{1}{\pi i} \int_{\sigma_3 - i\infty}^{\sigma_3 + i\infty} \widehat{H}(2s - 2w + 1 - 2z) \Gscr_{\mu_F}^{+}\left(\frac{1}{2} + z - s\right) G^{-}(s) \, ds,\]
where $0 < \sigma_3 < 1/2 + x$. We then shift the contour to the line $\Re(s) = A_1$, where $A_1$ is a large positive constant such that no poles of the integrand lie on this line. We pick up residues at $s = 1/2 + z + \mu_j + \ell$ for each $\ell \in \N_0$; by \eqref{eqn:Gpmbounds} and \eqref{eqn:hatHbounds}, these are $O_{\mu_F,K}(|y|^{A_1 + x} e^{-\pi |y|})$. Next, we let $\Omega : \R \to [0,1]$ be a smooth function equal to $1$ on $(-\infty,1]$, supported on $(-\infty,2]$, and having bounded derivatives. We then write the remaining integral as $\II_1 + \II_2$, where
\begin{align*}
\II_1 & \coloneqq \frac{1}{\pi i} \int_{A_1 - i\infty}^{A_1 + i\infty} \widehat{H}(2s - 2w + 1 - 2z) \Gscr_{\mu_F}^{+}\left(\frac{1}{2} + z - s\right) G^{-}(s) \Omega\left(\frac{\tau}{y_0}\right) \, ds,	\\
\II_2 & \coloneqq \frac{1}{\pi i} \int_{A_1 - i\infty}^{A_1 + i\infty} \widehat{H}(2s - 2w + 1 - 2z) \Gscr_{\mu_F}^{+}\left(\frac{1}{2} + z - s\right) G^{-}(s) \left(1 - \Omega\left(\frac{\tau}{y_0}\right)\right) \, ds,
\end{align*}
where $y_0 \coloneqq |y|^{\delta}$ with $\delta \in (0,1)$ a parameter to be chosen. By \eqref{eqn:Gscrmubounds}, \eqref{eqn:Gpmbounds}, and \eqref{eqn:hatHbounds}, we have that
\[\II_1 \ll_{\mu_F,K} |y|^{-A_1 - 2u + x} y_0^{A_1 + \frac{1}{2}}.\]
For $\II_2$, we insert the asymptotic expansions \eqref{eqn:Gscrpmasymp}, \eqref{eqn:Gpmasymp}, and \eqref{eqn:hatHasymp}. The contribution from the error terms is $O_{\mu_F,K}(|y|^{- B - 2u + x + \frac{1}{2}})$. The main term is of the shape
\begin{equation}
\label{eqn:II2mainterm}
\int_{0}^{\infty} (\tau + |y|)^{-A_1 - 2u + x} \tau^{A_1 - \frac{1}{2}} e^{i\Phi_y(\tau)} W_y(\tau) \left(1 - \Omega\left(\frac{\tau}{y_0}\right)\right) \, d\tau,
\end{equation}
where $W_y(\tau)$ is a $1$-inert function (in the sense of \cite[Definition 2.1]{KPY19}) and
\begin{equation}
\label{eqn:Phiytaudefeq}
\Phi_y(\tau) \coloneqq \tau \log \frac{\tau}{2\pi e} - (\tau + |y|) \log \frac{\tau + |y|}{2\pi e}.
\end{equation}
Since $\tau \geq y_0$ and $y \leq -1$, the derivatives satisfy
\[|\Phi_y'(\tau)| \asymp \begin{dcases*}
1 & if $\tau \leq |y|$,	\\
\frac{|y|}{\tau} & if $\tau \geq |y|$,
\end{dcases*} \qquad \Phi_y^{(j)}(\tau) \ll_j \begin{dcases*}
\frac{1}{\tau^{j - 1}} & if $j \geq 2$ and $\tau \leq |y|$,	\\
\frac{|y|}{\tau^j} & if $j \geq 2$ and $\tau \geq |y|$.
\end{dcases*}\]
We insert a dyadic smooth partition of unity into the integral \eqref{eqn:II2mainterm}, dividing it into dyadic ranges of length $Z = 2^{k - 1} y_0$ with $k \in \N$. We then estimate each dyadic portion via the integration by parts estimate in \cite[Lemma 2]{BK19} with the parameters given by
\[Y = \begin{dcases*}
Z & if $Z \leq |y|$,	\\
|y| & if $Z \geq |y|$,
\end{dcases*} \quad R = \begin{dcases*}
1 & if $Z \leq |y|$,	\\
\frac{|y|}{Z} & if $Z \geq |y|$,
\end{dcases*} \quad X = \begin{dcases*}
|y|^{-A_1 - 2u + x} Z^{A_1 - \frac{1}{2}} & if $Z \leq |y|$,	\\
Z^{-2u + x - \frac{1}{2}} & if $Z \geq |y|$.
\end{dcases*}\]
and also $\beta - \alpha \asymp U = Q = Z$. We deduce that \eqref{eqn:II2mainterm} is
\begin{align*}
& \ll_{\mu_F,K} |y|^{-A_1 - 2u + x} \sum_{k \leq \log_2 \frac{|y|}{y_0}} (2^k y_0)^{A_1 - \frac{B}{4} + \frac{1}{2}} + |y|^{-\frac{B}{4}} \sum_{k > \log_2 \frac{|y|}{y_0}} (2^k y_0)^{-2u + x + \frac{1}{2}}	\\
& \ll_{\mu_F,K} |y|^{-\min\left\{A_1,\frac{B}{4} - \frac{1}{2}\right\} - 2u + x}.
\end{align*}
Taking $y_0 = |y|^{\e/(A_1 + 1/2)}$ with $A_1 = A + 1/2$ yields \eqref{eqn:HHpmwzdecay} for $y \leq -1$.

Next, we consider the case $y \geq 1$. We instead make the change of variables $s \mapsto 2s - 2w + 1$ in \eqref{eqn:HHwzdefeq}, so that
\[\HH_{\mu_F}^{+}(w,z) = \frac{1}{\pi i} \int_{\sigma_4 - i\infty}^{\sigma_4 + i\infty} \widehat{H}(2s - 2w + 1) \Gscr_{\mu_F}^{+}\left(\frac{1}{2} - s\right) G^{-}(s + z) \, ds,\]
where $-x < \sigma_4 < 1/2$. We shift the contour to the line $\Re(s) = -A_2$, where $A_2 < A$ is a large positive constant such that no poles of the integrand lie on this line. We pick up residues at $s = -z - \ell$ for $\ell \in \N_0$; by \eqref{eqn:Gscrmubounds} and \eqref{eqn:hatHbounds}, these are $O_{\mu_F,K} (y^{4A_2 - 2u + x} e^{-3\pi y})$. We again let $\Omega$ be as above and write the remaining integral as $\JJ_1 + \JJ_2$, where
\begin{align*}
\JJ_1 & \coloneqq \frac{1}{\pi i} \int_{-A_2 - i\infty}^{-A_2 + i\infty} \widehat{H}(2s - 2w + 1) \Gscr_{\mu_F}^{+}\left(\frac{1}{2} - s\right) G^{-}(s + z) \Omega\left(\frac{\tau}{y_0}\right) \, ds,	\\
\JJ_2 & \coloneqq \frac{1}{\pi i} \int_{-A_2 - i\infty}^{-A_2 + i\infty} \widehat{H}(2s - 2w + 1) \Gscr_{\mu_F}^{+}\left(\frac{1}{2} - s\right) G^{-}(s + z) \left(1 - \Omega\left(\frac{\tau}{y_0}\right)\right) \, ds,
\end{align*}
where $y_0 \coloneqq y^{\delta}$ with $\delta \in (0,1)$ a parameter to be chosen. By \eqref{eqn:Gscrmubounds}, \eqref{eqn:Gpmbounds}, and \eqref{eqn:hatHbounds}, we have that
\[\JJ_1 \ll_{\mu_F,K} y^{-A_2 - \frac{1}{2}} y_0^{A_2 - 2u + x + 1}.\]
For $\JJ_2$, we insert the asymptotic expansions \eqref{eqn:Gscrpmasymp}, \eqref{eqn:Gpmasymp}, and \eqref{eqn:hatHasymp}. The contribution from the error terms is $O_{\mu_F,K}(y^{-B - 2u + x + \frac{1}{2}})$. The main term is of the shape
\begin{equation}
\label{eqn:JJ2mainterm}
\int_{0}^{\infty} \tau^{A_2 - 2u + x} (\tau + y)^{-A_2 - \frac{1}{2}} e^{-i\Phi_y(\tau)} W_y(\tau) \left(1 - \Omega\left(\frac{\tau}{y_0}\right)\right) \, d\tau,
\end{equation}
where $W_y(\tau)$ is again a $1$-inert function and $\Phi_y(\tau)$ is as in \eqref{eqn:Phiytaudefeq}. The same integration by parts argument as before shows that \eqref{eqn:JJ2mainterm} is
\[\ll_{\mu_F,K} y^{-\min\left\{A_2,\frac{B}{4} + 2u - x - 1\right\} - \frac{1}{2}}.\]
Taking $y_0 = y^{\e/(A_2 - 2u + x + 1)}$ with $A_2 = A - \e$ yields \eqref{eqn:HHpmwzdecay} for $y \geq 1$.
\end{proof}

\section{The First Moment at the Central Point}
\label{sect:firstmomentcentral}

Our goal is to analytically continue the identity given in \hyperref[prop:absconv]{Proposition \ref*{prop:absconv}} to the central point $w = 1/2$. In order to do so, we first require the following second moment bounds.

\begin{lemma}[Cf.\ {\cite[Lemma 3.8]{HK25}}]
\label{lem:integralsecondmomentbounds}
Given a Dirichlet character $\psi$ modulo $q$, we have the bounds
\begin{align}
\label{eqn:Lspsi2ndmoment}
\int_{U}^{2U} |L(\sigma + it,\psi)|^2 \, dt & \ll_{q,\e} U^{1 + \e} \quad \text{for $\sigma \geq \tfrac{1}{2}$},	\\
\label{eqn:LsFpsi2ndmoment}
\int_{U}^{2U} |L(\sigma + it,F \otimes \psi)|^2 \, dt & \ll_{F,q,\e} \begin{dcases*}
U^{3(1 - \sigma) + \e} & if $\frac{1}{2} \leq \sigma \leq \frac{2}{3}$,	\\
U^{1 + \e} & if $\sigma \geq \frac{2}{3}$.
\end{dcases*}
\end{align}
\end{lemma}

Under the assumption of the generalised Lindel\"{o}f hypothesis, the bound \eqref{eqn:Lspsi2ndmoment} is essentially optimal, whereas \eqref{eqn:LsFpsi2ndmoment} falls shy of the conjecturally optimal upper bound $O_{F,q,\e}(U^{1 + \e})$ when $1/2 \leq \sigma < 2/3$.

\begin{proof}
This follows by using the approximate functional equation \cite[Theorem 5.3]{IK04} to write $L(\sigma + it,\psi)$ and $L(\sigma + it,F \otimes \psi)$ in terms of Dirichlet polynomials and then invoking the Montgomery--Vaughan mean value theorem for Dirichlet polynomials \cite[Corollary 3]{MV74}.
\end{proof}

We now prove a $\GL_3 \times \GL_2 \leftrightsquigarrow \GL_4 \times \GL_1$ spectral reciprocity formula for the central value $w = 1/2$. 

\begin{theorem}
\label{thm:central}
Let $F$ be a selfdual Hecke--Maa\ss{} cusp form for $\SL_3(\Z)$. Let $q = q_1 q_2$ be a positive integer with $(q_1,q_2) = 1$. Let $\chi_1$ be a primitive Dirichlet character of conductor $q_1$, and set $\chi \coloneqq \chi_1 \chi_{0(q_2)}$. Let $(h,h^{\hol})$ be admissible of type $(A,B)$ for some $A,B > 5$. Then the moment
\begin{multline}
\label{eqn:centralLHS}
\sum_{\substack{q' \mid q \\ q' \equiv 0 \hspace{-.25cm} \pmod{q_{\overline{\chi_1}^2}}}} \alpha(q,q',q_{\overline{\chi_1}^2}) \sum_{f \in \BB_0^{\ast}(q',\overline{\chi_1}^2)} \frac{L^q\left(\frac{1}{2},F \otimes f \otimes \chi_1\right)}{L^q(1,\ad f)} h(t_f)	\\
+ \sum_{\substack{q' \mid q \\ q' \equiv 0 \hspace{-.25cm} \pmod{q_{\overline{\chi_1}^2}}}} \alpha(q,q',q_{\overline{\chi_1}^2}) \sum_{\substack{\psi_1,\psi_2 \hspace{-.25cm} \pmod{q} \\ \psi_1 \psi_2 = \overline{\chi_1}^2 \\ q_{\psi_1} q_{\psi_2} = q'}} \frac{1}{2\pi} \int_{-\infty}^{\infty} \left|\frac{L^q\left(\frac{1}{2} + it,F \otimes \psi_1 \chi_1\right)}{L^q(1 + 2it,\psi_1\overline{\psi_2})}\right|^2 h(t) \, dt	\\
+ \sum_{\substack{q' \mid q \\ q' \equiv 0 \hspace{-.25cm} \pmod{q_{\overline{\chi_1}^2}}}} \alpha(q,q',q_{\overline{\chi_1}^2}) \sum_{f \in \BB_{\hol}^{\ast}(q',\overline{\chi_1}^2)} \frac{L^q\left(\frac{1}{2},F \otimes f \otimes \chi_1\right)}{L^q(1,\ad f)} h^{\hol}(k_f).
\end{multline}
is equal to the sum of the primary main term
\begin{equation}
\label{eqn:centralmain}
q L^q(1,F) \left(\frac{1}{2\pi^2} \int_{-\infty}^{\infty} h(r) r \tanh\pi r \, dr + \sum_{\substack{k = 2 \\ k \equiv 0 \hspace{-.25cm} \pmod{2}}}^{\infty} \frac{k - 1}{2\pi^2} h^{\hol}(k)\right),
\end{equation}
the secondary main term
\begin{equation}
\label{eqn:centralsecondmain}
\begin{dcases*}
\chi_1(-1) q_1 L^q(1,F) \sum_{\substack{k = 2 \\ k \equiv 0 \hspace{-.25cm} \pmod{2}}}^{\infty} \frac{k - 1}{2\pi^2} i^{-k} h^{\hol}(k) & if $q_2$ is squarefree,	\\
0 & otherwise,
\end{dcases*}
\end{equation}
and the dual moment
\begin{equation}
\label{eqn:centraldual}
\frac{1}{\varphi(q)} \sum_{\psi \hspace{-.25cm} \pmod{q}} \frac{1}{2\pi} \int_{-\infty}^{\infty} L\left(\frac{1}{2} + it, F \otimes \psi\right) L\left(\frac{1}{2} - it,\overline{\psi}\right) \ZZ_{\chi}\left(\psi;\frac{1}{2},it\right) \sum_{\pm} \psi(\mp 1) \HH_{\mu_F}^{\pm}\left(\frac{1}{2},it\right) \, dt,
\end{equation}
where $\ZZ_{\chi}(\psi;w,z)$ is as in \eqref{eqn:ZZchipsiwzdefeq} and $\HH_{\mu_F}^{\pm}(w,z)$ is as in \eqref{eqn:HHwzdefeq}.
\end{theorem}

For notational simplicity, we define $\ZZ_{\chi}(\psi;t) \coloneqq \ZZ_{\chi}(\psi;1/2,it)$, so that
\begin{multline}
\label{eqn:ZZchipsitdefeq}
\ZZ_{\chi}(\psi;t) \coloneqq \sum_{c_0 \mid q^{\infty}} \frac{1}{\varphi(c_0 q)^2 c_0^{\frac{1}{2} - it}} \sum_{\substack{c_{1,0} c_{2,0} d_0 n_{1,0} = q \\ (c_{2,0} d_0 n_{1,0},c_0) = 1}} \frac{\varphi(c_0 c_{1,0}) \varphi(c_0 c_{1,0} d_0 n_{1,0}) \mu(d_0) A_F(1,n_{1,0})}{c_{1,0}^{\frac{1}{2} - it} c_{2,0}^{2it} d_0^{1 + 2it} n_{1,0}^{\frac{1}{2} + it}}	\\
\times \sum_{\substack{n_{2,0} \mid q^{\infty} \\ (n_{2,0},c_0) = 1}} \frac{A_F(n_{2,0},1)}{n_{2,0}^{\frac{1}{2} + it}} \VV_{\chi}(\psi; c_{2,0} d_0 n_{1,0}, c_{2,0} d_0^2 n_{1,0} n_{2,0}, c_{2,0}, c_0)
\end{multline}
and $\HH_{\mu_F}^{\pm}(t) \coloneqq \HH_{\mu_F}^{\pm}(1/2,it)$, so that
\begin{equation}
\label{eqn:HHtdefeq}
\HH_{\mu_F}^{\pm}(t) \coloneqq \frac{1}{2\pi i} \int_{\sigma_2 - i\infty}^{\sigma_2 + i\infty} \widehat{H}(s) \Gscr_{\mu_F}^{\pm}\left(\frac{1 - s}{2}\right) G^{\mp}\left(\frac{s}{2} + it\right) \, ds,
\end{equation}
where $0 < \sigma_2 < 1$.

\begin{proof}
We analytically continue the identity given in \hyperref[prop:absconv]{Proposition \ref*{prop:absconv}} to $w = 1/2$. The analytic continuation of the $\GL_2$ moment of $\GL_3 \times \GL_2$ Rankin--Selberg $L$-functions \eqref{eqn:absoluteLHS} is straightforward: via the Cauchy--Schwarz inequality, the approximate functional equation for $L(w,F \otimes f \otimes \chi_1)$, and the Weyl law, these expressions converge absolutely provided that $h(t) \ll (1 + |t|)^{-5/2 - \delta}$ and $h^{\hol}(k) \ll k^{-5/2 - \delta}$ for some $\delta > 0$ (cf.\ \cite[Proposition 6.1 (1)]{HK25}). The ensuing expression is \eqref{eqn:centralLHS}. The analytic continuation of the primary main term \eqref{eqn:absolutemain} is simply \eqref{eqn:centralmain}.

For the analytic continuation of the $\GL_4 \times \GL_1$ moment \eqref{eqn:absolutedual}, we first shift the contour of integration to the line $\Re(z) = 0$. We then analytically continue this expression to $w = 1/2$, keeping $w = u + iv$ at all times inside a compact subset $K$ of the closed vertical strip $1/2 \leq u \leq 2$ containing $1/2$ (in particular, we assume that $v$ is bounded). To ensure that this process of analytic continuation is valid, we use the following two facts:
\begin{itemize}
\item $\HH_{\mu_F}^{\pm}(w,z)$ is holomorphic as a function of $w$ and $z$;
\item the integral over $z$ is absolutely convergent for all $w \in K$, in which it defines a holomorphic function of $w \in K$.
\end{itemize}
The former fact follows from \hyperref[lem:HHpmwzbounds]{Lemma \ref*{lem:HHpmwzbounds}}. For the latter fact, we break up the integral over $z$ into dyadic ranges and bound $\HH_{\mu_F}^{\pm}(w,z)$ pointwise over these ranges by \hyperref[lem:HHpmwzbounds]{Lemma \ref*{lem:HHpmwzbounds}}. The absolute convergence of this dyadic sum is ensured via the bounds
\begin{multline*}
\int_{U}^{2U} \sum_{\pm} \left|L\left(\frac{1}{2} + x \pm iy,F \otimes \psi\right) L\left(2w - \frac{1}{2} - x \mp iy,\overline{\psi}\right)\right| \, dy	\\
\ll_{F,q,K,\e} \begin{dcases*}
U^{\frac{5}{4} - \frac{3x}{2} + \e} & if $0 \leq x \leq \min\{\frac{1}{6},2u - 1\}$,	\\
U^{1 + \e} & if $\frac{1}{6} \leq x \leq 2u - 1$,
\end{dcases*}
\end{multline*}
which are immediate consequences of the Cauchy--Schwarz inequality coupled with \hyperref[lem:integralsecondmomentbounds]{Lemma \ref*{lem:integralsecondmomentbounds}}. 

There is one last subtlety. This shifting of the contour to $\Re(z) = 0$ means that once we subsequently analytically continue to $1/2 < u < 3/4$, the $\GL_4 \times \GL_1$ moment gives an additional term arising from the residue\footnote{To be more precise, fix $\e \in (0,\frac{1}{100})$. We shift the contour to $\Re(z) = 0$. The ensuing expression is holomorphic, as a function of $w$, for $3/4 < u < 2$. For $3/4 < u < 3/4 + \e$, we then shift the contour from $\Re(z) = 0$ to $\Re(z) = 3\e$, picking up a residue at $z = 2w - 3/2$. The ensuing expression is holomorphic, as a function of $w$, for $3/4 - \e < u < 3/4 + \e$. For $3/4 - \e < u < 3/4$, we finally shift the contour back to $\Re(z) = 0$, crossing no poles. The ensuing expression is holomorphic, as a function of $w$, for $1/2 < u < 3/4$.} at the pole at $z = 2w - 3/2$ of $L(2w - 1/2 - z,\overline{\psi})$ with $\psi = \psi_{0(q)}$, the principal character modulo $q$, which is
\[\frac{1}{q} L(2w - 1, \widetilde{F}) \frac{\ZZ_{\chi}\left(\psi_{0(q)};w,2w - \frac{3}{2}\right)}{L_q(2w - 1, \widetilde{F})} \sum_{\pm} \HH_{\mu_F}^{\pm}\left(w,2w - \frac{3}{2}\right).\]
From \eqref{eqn:HHwzdefeq},
\[\sum_{\pm} \HH_{\mu_F}^{\pm}\left(w,2w - \frac{3}{2}\right) = \frac{1}{2\pi i} \int_{\sigma_2 - i\infty}^{\sigma_2 + i\infty} \widehat{H}(s) \sum_{\pm} \Gscr_{\mu_F}^{\pm}\left(1 - \frac{s}{2} - w\right) G^{\mp}\left(\frac{s}{2} + 3w - 2\right) \, ds,\]
where $4 - 6u < \sigma_2 < 2 - 2u$. We shift the contour to the left to the line $\Re(s) = \sigma_3$ with $2 - 6u < \sigma_3 < 4 - 6u$, which picks up a residue at $s = 4 - 6w$ due to the poles of $G^{\mp}(s/2 + 3w - 2)$, so that
\begin{multline*}
\sum_{\pm} \HH_{\mu_F}^{\pm}\left(w,2w - \frac{3}{2}\right) = \frac{1}{2\pi i} \int_{\sigma_3 - i\infty}^{\sigma_3 + i\infty} \widehat{H}(s) \sum_{\pm} \Gscr_{\mu_F}^{\pm}\left(1 - \frac{s}{2} - w\right) G^{\mp}\left(\frac{s}{2} + 3w - 2\right) \, ds	\\
+ 2 \widehat{H}(4 - 6w) \sum_{\pm} \Gscr_{\mu_F}^{\pm}(2w - 1)
\end{multline*}
Via the functional equation $L(2 - 2w,F) = \sum_{\pm} \Gscr_{\mu_F}^{\pm}(2w - 1) L(2w - 1,\widetilde{F})$, we deduce that this secondary main term is
\begin{multline*}
\frac{1}{q} L(2w - 1, \widetilde{F}) \frac{\ZZ_{\chi}\left(\psi_{0(q)};w,2w - \frac{3}{2}\right)}{L_q(2w - 1, \widetilde{F})} \frac{1}{2\pi i} \int_{\sigma_3 - i\infty}^{\sigma_3 + i\infty} \widehat{H}(s) \sum_{\pm} \Gscr_{\mu_F}^{\pm}\left(1 - \frac{s}{2} - w\right) G^{\mp}\left(\frac{s}{2} + 3w - 2\right) \, ds	\\
+ \frac{2}{q} L(2 - 2w,F) \frac{\ZZ_{\chi}\left(\psi_{0(q)};w,2w - \frac{3}{2}\right)}{L_q(2w - 1, \widetilde{F})} \widehat{H}(4 - 6w).
\end{multline*}
From \hyperref[cor:ZZchipsi0hol]{Corollary \ref*{cor:ZZchipsi0hol}}, the first term above extends holomorphically to $w = 1/2$, where it vanishes, since $L(0,\widetilde{F}) = 0$ as $F$ is selfdual. The second term above extends holomorphically to $w = 1/2$, where it is equal to
\[\begin{dcases*}
2 \chi_1(-1) q_1 L^q(1,F) \widehat{H}(1) & if $q_2$ is squarefree,	\\
0 & otherwise.
\end{dcases*}\]
It remains to note that
\begin{align*}
\widehat{H}(1) & = \frac{1}{2\pi^2} \int_{-\infty}^{\infty} \widehat{\JJ_r^+}(1) h(r) r \tanh \pi r \, dr + \sum_{\substack{k = 2 \\ k \equiv 0 \hspace{-.25cm} \pmod{2}}}^{\infty} \frac{k - 1}{2\pi^2} \widehat{\JJ_k^{\hol}}(1) h^{\hol}(k)	\\
& = \sum_{\substack{k = 2 \\ k \equiv 0 \hspace{-.25cm} \pmod{2}}}^{\infty} \frac{k - 1}{4\pi^2} i^{-k} h^{\hol}(k)
\end{align*}
as from \cite[(3.13)]{BK19}, we have that
\begin{align*}
\widehat{\JJ_r^+}(s) & = (2\pi)^{-s} \Gamma\left(\frac{s}{2} + ir\right) \Gamma\left(\frac{s}{2} - ir\right) \cos \frac{\pi s}{2},	\\
\widehat{\JJ_k^{\hol}}(s) & = \pi i^{-k} (2\pi)^{-s} \frac{\Gamma\left(\frac{s + k - 1}{2}\right)}{\Gamma\left(\frac{1 - s + k}{2}\right)}.
\qedhere
\end{align*}
\end{proof}

With \hyperref[thm:central]{Theorem \ref*{thm:central}} in hand, we direct our attention towards proving \hyperref[thm:firstmomentbounds]{Theorem \ref*{thm:firstmomentbounds}}. The proof is given in \hyperref[sect:prooffirstmoment]{Section \ref*{sect:prooffirstmoment}} and relies upon first choosing a specific choice of tuple of test functions $(h,h^{\hol})$, described in \hyperref[sect:testfunctionsII]{Section \ref*{sect:testfunctionsII}}, and then proving upper bounds for the primary main term \eqref{eqn:centralmain}, the secondary main term \eqref{eqn:centralsecondmain}, and the dual moment \eqref{eqn:centraldual}, of which the latter requires the most work.

Our strategy towards bounding the dual moment \eqref{eqn:centraldual} involves first bounding pointwise both the quantity $\ZZ_{\chi}(\psi;t)$ defined in \eqref{eqn:ZZchipsitdefeq}, which we detail in \hyperref[sect:charsumsII]{Section \ref*{sect:charsumsII}}, and the quantity $\HH_{\mu_F}^{\pm}(t)$ defined in \eqref{eqn:HHtdefeq}, which we detail in \hyperref[sect:testfunctionsII]{Section \ref*{sect:testfunctionsII}}. We then break up the integral over $t \in \R$ in \eqref{eqn:centraldual} into dyadic ranges, invoke pointwise bounds for $\HH_{\mu_F}^{\pm}(t)$, and apply the Cauchy--Schwarz inequality. In this way, the problem is reduced to proving second moment bounds for $L(1/2 + it,F \otimes \psi)$ and $L(1/2 - it,\overline{\psi}) \ZZ_{\chi}(\psi;t)$, which we detail in \hyperref[sect:secondmomentbounds]{Section \ref*{sect:secondmomentbounds}}.

\section{Character Sums II}
\label{sect:charsumsII}

We first focus on the quantity $\ZZ_{\chi}(\psi;t)$ defined in \eqref{eqn:ZZchipsitdefeq}. From \hyperref[lemno:VVchimult]{Lemma \ref*{lem:VVchiproperties} \ref*{lemno:VVchimult}} and multiplicativity, we have the factorisation $\ZZ_{\chi}(\psi;t) = \prod_{p^{\beta} \parallel q} \widetilde{\ZZ}_{\chi_{p^{\beta}}}(\psi_{p^{\beta}};t)$, where we have factorised $\chi$ and $\psi$ as a product of Dirichlet characters $\chi_{p^{\beta}}$ and $\psi_{p^{\beta}}$ modulo $p^{\beta}$, with
\begin{multline}
\label{eqn:tildeZZchipsitdefeq}
\widetilde{\ZZ}_{\chi_{p^{\beta}}}(\psi_{p^{\beta}};t) \coloneqq \overline{\psi_{qp^{-\beta}}}(p^{\beta}) \sum_{c_0 \mid p^{\infty}} \frac{\overline{\psi_{qp^{-\beta}}}(c_0)}{\varphi(c_0 p^{\beta})^2 c_0^{\frac{1}{2} - it}}	\\
\times \sum_{\substack{c_{1,0} c_{2,0} d_0 n_{1,0} = p^{\beta} \\ (c_{2,0} d_0 n_{1,0},c_0) = 1}} \frac{\mu(d_0) \varphi(c_0 c_{1,0} d_0 n_{1,0}) \varphi(c_0 c_{1,0}) A_F(1,n_{1,0}) \psi_{qp^{-\beta}}(c_{2,0}^3 d_0^3 n_{1,0}^2)}{c_{1,0}^{\frac{1}{2} - it} c_{2,0}^{2it} d_0^{1 + 2it} n_{1,0}^{\frac{1}{2} + it}}	\\
\times \sum_{\substack{n_{2,0} \mid p^{\infty} \\ (n_{2,0},c_0) = 1}} \frac{A_F(n_{2,0},1) \psi_{qp^{-\beta}}(n_{2,0})}{n_{2,0}^{\frac{1}{2} + it}} \VV_{\chi_{p^{\beta}}}(\psi_{p^{\beta}}; c_{2,0} d_0 n_{1,0}, c_{2,0} d_0^2 n_{1,0} n_{2,0}, c_{2,0}, c_0).
\end{multline}
The bounds that we obtain for $\widetilde{\ZZ}_{\chi_{p^{\beta}}}(\psi_{p^{\beta}};t)$ depend delicately on the conductors of the Dirichlet characters $\chi_{p^{\beta}}$ and $\psi_{p^{\beta}}$.

\begin{lemma}\hspace{1em}
\label{lem:tildeZbounds}
\begin{enumerate}[leftmargin=*,label=\textup{(\arabic*)}]
\item\label{lemno:tildeZboundsprincprinc} Let $\chi_{0(p^{\beta})},\psi_{0(p^{\beta})}$ both be the principal character modulo $p^{\beta}$. Then $\widetilde{\ZZ}_{\chi_{0(p^{\beta})}}(\psi_{0(p^{\beta})};t) \ll p^{3\beta/2}$.
\item\label{lemno:tildeZboundsprincnonprinc} Let $\chi_{0(p^{\beta})}$ be the principal character modulo $p^{\beta}$ and let $\psi_{p^{\beta}}$ be a nonprincipal character modulo $p^{\beta}$. Then
\[\widetilde{\ZZ}_{\chi_{0(p^{\beta})}}(\psi_{p^{\beta}};t) = \begin{dcases*}
O(1) & if $\beta = 1$,	\\
0 & otherwise.
\end{dcases*}\]
\item\label{lemno:tildeZboundsprimarb} Let $\chi_{p^{\beta}}$ be a primitive character modulo $p^{\beta}$ and let $\psi_{p^{\beta}}$ be a Dirichlet character modulo $p^{\beta}$. Then
\[\widetilde{\ZZ}_{\chi_{p^{\beta}}}(\psi_{p^{\beta}};t) \ll \begin{dcases*}
|g(\chi_{p^{\beta}},\psi_{p^{\beta}})| & if $\psi_{p^{\beta}}$ is primitive,	\\
p^{\beta} & otherwise,
\end{dcases*}\]
where $g(\chi_{p^{\beta}},\psi_{p^{\beta}})$ is as in \eqref{eqn:gchipsidefeq}.
\end{enumerate}
\end{lemma}

\begin{proof}\hspace{1em}

\begin{enumerate}[leftmargin=*,label=\textup{(\arabic*)}]
\item The desired bound follows by a case-by-case analysis using the bounds in \hyperref[lem:VVchiprincprinc]{Lemma \ref*{lem:VVchiprincprinc}} for the character sum \eqref{eqn:VVchidefeq}; the main contribution comes from when $c_{1,0} = p^{\beta}$ and $c_0 = c_{2,0} = d_0 = n_{1,0} = n_{2,0} = 1$.
\item From \hyperref[lem:VVchiprincnonprinc]{Lemma \ref*{lem:VVchiprincnonprinc}}, the character sum \eqref{eqn:VVchidefeq} appearing in \eqref{eqn:tildeZZchipsitdefeq} vanishes unless $\beta = 1$, $c_{1,0} = p$, and $c_0 = c_{2,0} = d_0 = n_{1,0} = n_{2,0} = 1$, in which case
\[\widetilde{\ZZ}_{\chi_{0(p)}}(\psi_p;t) = \overline{\psi_{qp^{-1}}}(p) \overline{\tau(\psi_p)} p^{-\frac{3}{2} + it} (p + 1).\]
\item Suppose first that $\psi_{p^{\beta}}$ is primitive. By \hyperref[lem:VVchiprimnonprinc]{Lemma \ref*{lem:VVchiprimnonprinc}}, the only contribution is from when $c_{1,0} = p^{\beta}$ and $c_0 = c_{2,0} = d_0 = n_{1,0} = n_{2,0} = 1$, so that
\[\widetilde{\ZZ}_{\chi_{p^{\beta}}}(\psi_{p^{\beta}};t) = \chi_{p^{\beta}}(-1) \overline{\psi_{qp^{-\beta}}}(p^{\beta}) \tau(\overline{\psi_{p^{\beta}}}) p^{-\frac{\beta}{2} + \beta it} g(\chi_{p^{\beta}},\psi_{p^{\beta}}).\]

Next, suppose that $\psi_{p^{\beta}}$ is imprimitive but nonprincipal, so that $\beta \geq 2$ and the conductor of $\psi_{p^{\beta}}$ is $p^{\alpha}$ for some $\alpha \in \{1,\ldots,\beta - 1\}$. From \hyperref[lem:VVchiprimnonprinc]{Lemma \ref*{lem:VVchiprimnonprinc}}, the character sum vanishes unless either $c_0 = p^{\beta - \alpha}$, $c_{1,0} = p^{\beta}$, and $d_0 = n_{1,0} = n_{2,0} = 1$, or $c_{1,0} = p^{\alpha}$, $c_{2,0} = p^{\beta - \alpha}$, and $c_0 = d_0 = n_{1,0} = n_{2,0} = 1$. Both cases contribute terms of size $O(p^{\beta})$.

Finally, if $\psi_{p^{\beta}} = \psi_{0(p^{\beta})}$ is principal, the desired bound follows by a case-by-case analysis using the bounds in \hyperref[lem:VVchiprimprinc]{Lemma \ref*{lem:VVchiprimprinc}} for the character sum \eqref{eqn:VVchidefeq}.
\qedhere
\end{enumerate}
\end{proof}

We use \hyperref[lem:tildeZbounds]{Lemma \ref*{lem:tildeZbounds}} in conjunction with the factorisation $\ZZ_{\chi}(\psi;t) = \prod_{p^{\beta} \parallel q} \widetilde{\ZZ}_{\chi_{p^{\beta}}}(\psi_{p^{\beta}};t)$ to bound $\ZZ_{\chi}(\psi;t)$. These bounds depend on the factorisations of both $q$ and $\psi$.

\begin{corollary}
\label{cor:Zbounds}
Let $q = q_1 q_2$ be a positive integer with $(q_1,q_2) = 1$, and write $q_2 = q_3 q_4$ with $q_3$ squarefree, $q_4$ squarefull, and $(q_3,q_4) = 1$. Let $\chi_1$ be a primitive Dirichlet character of conductor $q_1$, and set $\chi \coloneqq \chi_1 \chi_{0(q_2)}$. We have that $\ZZ_{\chi}(\psi;t) = 0$ unless there exists a Dirichlet character $\psi_1$ modulo $q_1$ and a primitive Dirichlet character $\psi_3'$ modulo $q_3'$ for some divisor $q_3'$ of $q_3$ such that $\psi = \psi_1 \psi_3' \psi_{0(q_2/q_3')}$. Moreover,
\[\ZZ_{\chi}\left(\psi_1 \psi_3' \psi_{0(q_2/q_3')};t\right) \ll_{\varepsilon} q_1 \left(\frac{q_2}{q_3'}\right)^{\frac{3}{2}} q^{\e} \prod_{p^{\beta} \parallel q_1} \left(\delta_{\psi_{p^{\beta}}',\star} \frac{|g(\chi_{p^{\beta}},\psi_{p^{\beta}}')|}{p^{\beta}} + 1\right).\]
\end{corollary}

Here $\delta_{\psi_{p^{\beta}}',\star}$ is $1$ if $\psi_{p^{\beta}}'$ is primitive and is $0$ otherwise.

\section{Test Functions and Transforms II}
\label{sect:testfunctionsII}

With \hyperref[thm:firstmomentbounds]{Theorem \ref*{thm:firstmomentbounds}} in mind, we now specify an explicit tuple of test functions $(h,h^{\hol})$, namely
\begin{equation}
\label{eqn:testfunctionchoice}
h(t) \coloneqq \prod_{n = 1}^{C} \left(\frac{t^2 + \left(n - \frac{1}{2}\right)^2}{T^2}\right) \left(\sum_{\pm} \exp\left(-\left(\frac{t \pm T}{U}\right)^2\right)\right)^2, \qquad h^{\hol}(k) \coloneqq \Omega\left(\frac{k - 1 - T}{U}\right).
\end{equation}
Here $T^{\e} \leq U \leq T^{1 - \e}$ for some fixed $\e > 0$, $C < \frac{T}{2} - U$ is a large fixed positive integer (for example, $C = 500$ suffices), while $\Omega$ is a fixed nonnegative smooth function equal to $1$ on $[-1,1]$, supported on $[-2,2]$, and having bounded derivatives. Note that we may view $h^{\hol}$ not just as a function on $2\N$ but as a compactly supported smooth function on $\R$. Our goal for this section is to bound the primary main term \eqref{eqn:centralmain}, the secondary main term \eqref{eqn:centralsecondmain}, and the transform $\HH_{\mu_F}^{\pm}(t)$ given by \eqref{eqn:HHtdefeq} with this choice of tuple of test funactions. The bounds that we obtain depend on $\e$, $C$, and the Langlands parameters $\mu_F = (2it_g,0,-2it_g)$ of the fixed selfdual Hecke--Maa\ss{} cusp form $F = \ad g$ for $\SL_3(\Z)$, where $g \in \BB_0(1,\chi_0)$; we suppress these dependencies from the notation throughout this section.

An immediate consequence of the definition \eqref{eqn:testfunctionchoice} is the following.

\begin{lemma}
\label{lem:testfunction}
The tuple $(h,h^{\hol})$ given by \eqref{eqn:testfunctionchoice} is admissible of type $(A,B)$ for any $A < C + 1/2$ and any $B \geq 0$. Moreover, $h$ is nonnegative on $\R \cup i(-\frac{1}{2},\frac{1}{2})$ and $h(t) \asymp 1$ for $t \in [-T - U,-T + U] \cup [T - U,T + U]$. Similarly, $h^{\hol}$ is nonnegative and $h^{\hol}(k) \asymp 1$ for $k \in 2\N \cap [T - U,T + U]$. Finally,
\begin{equation}
\label{eqn:maintermbounds}
\begin{drcases*}
\frac{1}{2\pi^2} \int_{-\infty}^{\infty} h(r) r \tanh\pi r \, dr \\
\sum_{\substack{k = 2 \\ k \equiv 0 \hspace{-.25cm} \pmod{2}}}^{\infty} \frac{k - 1}{2\pi^2} h^{\hol}(k)
\end{drcases*} \ll TU.
\end{equation}
\end{lemma}

We turn to the issue of bounding $\HH_{\mu_F}^{\pm}(t)$. We shall prove the following.

\begin{lemma}
\label{lem:HHpmtbounds}
Let $(h,h^{\hol})$ be the tuple of test functions given by \eqref{eqn:testfunctionchoice}. Then the transform $\HH_{\mu_F}^{\pm}(t)$ given by \eqref{eqn:HHtdefeq} satisfies
\begin{equation}
\label{eqn:HHpmbounds}
\HH_{\mu_F}^{\pm}(t) \ll \begin{dcases*}
U & if $|t| \leq \frac{T}{U}$,	\\
U \left(\frac{U|t|}{T}\right)^{-C} & if $|t| \geq \frac{T}{U}$.
\end{dcases*}
\end{equation}
\end{lemma}

Our first step is to rewrite the Mellin transform $\widehat{H}(s)$ as a sum of integrals involving certain functions $(\Fscr h)(u), (\Fscr^{\hol} h^{\hol})(u)$ defined in the following lemma.

\begin{lemma}
Let $(h,h^{\hol})$ be the tuple \eqref{eqn:testfunctionchoice}, and for $u \in \R$, define
\begin{align}
\label{eqn:Fscrhdefeq}
(\Fscr h)(u) & \coloneqq \int_{-\infty}^{\infty} h(r) r \tanh \pi r \, e(-ru) \, dr,	\\
\label{eqn:Fscrholhholdefeq}
(\Fscr^{\hol} h^{\hol})(u) & \coloneqq -2 \int_{-\infty}^{\infty} h^{\hol}(2r + 1) r e(-ru) \, dr.
\end{align}
Then there exist smooth functions $g_{+},g_{-},g_{\hol} : \R \to \C$ satisfying
\begin{equation}
\label{eqn:gdecay}
\begin{drcases*}
g_{+}^{(j)}(u) \\
g_{-}^{(j)}(u) \\
g_{\hol}^{(j)}(u)
\end{drcases*} \ll_{j,N} \begin{dcases*}
1 & if $|u| \leq 1$,	\\
|u|^{-N} & if $|u| \geq 1$
\end{dcases*}
\end{equation}
for all $j \in \N_0$ and $N \geq 0$ such that
\begin{align}
\label{eqn:Fscrhdefeq2}
(\Fscr h)(u) & = \sum_{\pm} TU e(\pm Tu) g_{\pm}(Uu) (\sech \pi u)^{2C + \frac{1}{2}},	\\
\label{eqn:Fscrholhholdefeq2}
(\Fscr^{\hol} h^{\hol})(u) & = TU e\left(-\frac{Tu}{2}\right) g_{\hol}(Uu).
\end{align}
Moreover, we have that
\begin{equation}
\label{eqn:FscrFourierinversion}
\int_{-\infty}^{\infty} (\Fscr h)(u) e(i\ell u) \, du = \int_{-\infty}^{\infty} (\Fscr^{\hol} h^{\hol})(u) e(\ell u) \, du = 0
\end{equation}
for all $\ell \in \{-C,\ldots,C\}$.
\end{lemma}

\begin{proof}
From the definitions \eqref{eqn:testfunctionchoice} of $h^{\hol}$ and \eqref{eqn:Fscrholhholdefeq} of $(\Fscr^{\hol} h^{\hol})(u)$, we have that
\begin{align*}
(\Fscr^{\hol} h^{\hol})(u) & = TU e\left(-\frac{Tu}{2}\right) g_{\hol}(Uu),	\\
g_{\hol}(u) & \coloneqq -\frac{1}{2} \int_{-\infty}^{\infty} \Omega(r) e\left(-\frac{ru}{2}\right) \, dr + \frac{U}{2\pi i T} \frac{\dee}{\dee u} \int_{-\infty}^{\infty} \Omega(r) e\left(-\frac{ru}{2}\right) \, dr.
\end{align*}
The desired bound \eqref{eqn:gdecay} for $g_{\hol}$ then follows by repeated differentiation under the integral sign and integration by parts, noting that $\Omega$ is smooth and compactly supported. The identity \eqref{eqn:FscrFourierinversion} for $(\Fscr^{\hol} h^{\hol})(u)$ holds by the Fourier inversion formula as $h^{\hol}(2\ell + 1) = 0$ for $\ell \in \{-C,\ldots,C\}$. 

The analogous result for $g_{+}$ and $g_{-}$ follow similarly via the definitions \eqref{eqn:testfunctionchoice} of $h$ and \eqref{eqn:Fscrhdefeq} of $(\Fscr h)(u)$. In this case, we additionally obtain a factor of $(\sech \pi u)^{2C + 1/2}$ in the identity \eqref{eqn:Fscrhdefeq2} by making the change of variables $r \mapsto r - i\sgn(u)(C + \frac{1}{4})$ in \eqref{eqn:Fscrhdefeq} and then shifting the contour of integration back to $\Im(r) = 0$, noting that the zeroes of $h(r)$ cancel out the poles of $\tanh \pi r$. Finally, the identity \eqref{eqn:FscrFourierinversion} for $(\Fscr h)(u)$ once more holds by the Fourier inversion formula as $\tanh \pi i \ell = 0$ for $\ell \in \{-C,\ldots,C\}$.
\end{proof}

\begin{lemma}
Let $(h,h^{\hol})$ be the tuple \eqref{eqn:testfunctionchoice} and define
\begin{align}
\label{eqn:Dscrhdefeq}
(\Dscr h)(s) & \coloneqq \int_{-\infty}^{\infty} (\Fscr h)(u) (\cosh^2 \pi u)^{-s} \, du,	\\
\label{eqn:Dscrholhholdefeq}
(\Dscr^{\hol} h^{\hol})(s) & \coloneqq \sum_{k = -\infty}^{\infty} (-1)^k \int_{k - \frac{1}{2}}^{k + \frac{1}{2}} (\Fscr^{\hol} h^{\hol})(u) (\cos^2 \pi u)^{-s} \, du.
\end{align}
For $-C - 1/4 < \sigma < 1/2$, we have that
\begin{equation}
\label{eqn:HMellinFourier}
\widehat{H}(2s) = \frac{1}{2\sqrt{\pi}} (2\pi)^{-2s} \frac{\Gamma(s)}{\Gamma\left(\frac{1}{2} - s\right)} (\Dscr h)(s) + \frac{1}{2\sqrt{\pi}} (2\pi)^{-2s} \frac{\Gamma\left(s + \frac{1}{2}\right)}{\Gamma(1 - s)} (\Dscr^{\hol} h^{\hol})(s).
\end{equation}
\end{lemma}

\begin{proof}
We recall from \eqref{eqn:HdefeqKscr} that $\widehat{H}(s) = \widehat{\Kscr h}(s) + \widehat{\Kscr^{\hol} h^{\hol}}(s)$. Via \cite[(3.13)]{BK19}, Euler's reflection formula, and the Legendre duplication formula, we have that
\begin{align*}
\widehat{\Kscr h}(2s) & = \frac{1}{2\pi^2} \int_{-\infty}^{\infty} \widehat{\JJ_r^+}(2s) h(r) r \tanh \pi r \, dr	\\
& = \frac{1}{2\sqrt{\pi}} (2\pi)^{-2s} \frac{\Gamma(s)}{\Gamma\left(\frac{1}{2} - s\right)} \int_{-\infty}^{\infty} \frac{2^{2s - 1}}{\pi} \frac{\Gamma(s + ir) \Gamma(s - ir)}{\Gamma(2s)} h(r) r \tanh \pi r \, dr.
\end{align*}
From \cite[3.985.1]{GR15}, we have that for $\Re(s) > 0$,
\[\frac{2^{2s - 1}}{\pi} \frac{\Gamma(s + ir) \Gamma(s - ir)}{\Gamma(2s)} = \int_{-\infty}^{\infty} (\cosh^2 \pi u)^{-s} e(-ru) \, du.\]
Interchanging the order of integration, we find that
\[\widehat{\Kscr h}(2s) = \frac{1}{2\sqrt{\pi}} (2\pi)^{-2s} \frac{\Gamma(s)}{\Gamma\left(\frac{1}{2} - s\right)} (\Dscr h)(s).\]
This extends holomorphically to $-C - 1/4 < \Re(s) < C + 1/4$ since the identity \eqref{eqn:Fscrhdefeq2} together with the bounds \eqref{eqn:gdecay} ensure that $(\Dscr h)(s)$ converges absolutely in this region, while \eqref{eqn:FscrFourierinversion} implies that $(\Dscr h)(s)$ vanishes at $s = 0,-1,\ldots,-C$, which cancels out the poles of $\Gamma(s)$.

Similarly, we have via \cite[(3.13)]{BK19}, Euler's reflection formula, and the Legendre duplication formula that
\begin{align*}
\widehat{\Kscr^{\hol} h^{\hol}}(2s) & = \sum_{\substack{k = 2 \\ k \equiv 0 \hspace{-.25cm} \pmod{2}}}^{\infty} \widehat{\JJ_k^{\hol}}(2s) \frac{k - 1}{2\pi^2} h^{\hol}(k)	\\
& = -\frac{1}{2\sqrt{\pi}} (2\pi)^{-2s} \frac{\Gamma\left(s + \frac{1}{2}\right)}{\Gamma(1 -s)} \sum_{k = -\infty}^{\infty} 2^{2s} \frac{\Gamma(1 - 2s)}{\Gamma\left(\frac{1}{2} - s + k\right) \Gamma\left(\frac{3}{2} - s - k\right)} h^{\hol}(2k) (2k - 1).
\end{align*}
From \cite[3.892.2]{GR15}, we have that for $\Re(s) < 1$,
\[2^{2s} \frac{\Gamma(1 - 2s)}{\Gamma\left(\frac{1}{2} - s + k\right) \Gamma\left(\frac{3}{2} - s - k\right)} = \int_{-\frac{1}{2}}^{\frac{1}{2}} (\cos^2 \pi u)^{-s} e\left(\frac{u}{2}\right) e(-ku) \, du.\]
Via the Poisson summation formula, we find that $\widehat{\Kscr^{\hol} h^{\hol}}(2s)$ is equal to
\[-\frac{1}{2 \sqrt{\pi}} (2\pi)^{-2s} \frac{\Gamma\left(s + \frac{1}{2}\right)}{\Gamma(1 - s)} \sum_{k = -\infty}^{\infty} \int_{-\infty}^{\infty} h^{\hol}(2r) (2r - 1) e(-kr) \int_{-\frac{1}{2}}^{\frac{1}{2}} (\cos^2 \pi u)^{-s} e\left(\frac{u}{2}\right) e(-ru) \, du \, dr.\]
Interchanging the order of integration and making the change of variables $r \mapsto r + 1/2$, $u \mapsto u - k$, and $k \mapsto -k$, we deduce that
\[\widehat{\Kscr^{\hol} h^{\hol}}(2s) = \frac{1}{2\sqrt{\pi}} (2\pi)^{-2s} \frac{\Gamma\left(s + \frac{1}{2}\right)}{\Gamma(1 - s)} (\Dscr^{\hol} h^{\hol})(s).\]
This extends holomorphically to $-C - 1/4 < \Re(s) < 1/2$ since the identity \eqref{eqn:Fscrholhholdefeq2} together with the bounds \eqref{eqn:gdecay} ensure that $(\Dscr^{\hol} h^{\hol})(s)$ converges absolutely in this region, while \eqref{eqn:FscrFourierinversion} implies that $(\Dscr^{\hol} h^{\hol})(s)$ vanishes at $s = -1/2,-3/2,\ldots,-C + 1/2$, which cancels out the poles of $\Gamma(s + \frac{1}{2})$.
\end{proof}

We proceed to bounding $\HH_{\mu_F}^{\pm}(t)$. Throughout, we shall only deal with the test function $h(r)$ and its corresponding transforms $(\Fscr h)(u)$ given by \eqref{eqn:Fscrhdefeq} and $(\Dscr h)(s)$ given by \eqref{eqn:Dscrhdefeq}. The proofs for the test function $h^{\hol}(k)$ follow analogously, bearing in mind its transforms $(\Fscr^{\hol} h^{\hol})(u)$ given by \eqref{eqn:Fscrholhholdefeq} and $(\Dscr^{\hol} h^{\hol})(s)$ given by \eqref{eqn:Dscrholhholdefeq}.

To bound $\HH_{\mu_F}^{\pm}(t)$ when $|t|$ is much larger than $T$, we use the same method as the proof of \hyperref[lem:HHpmwzbounds]{Lemma \ref*{lem:HHpmwzbounds}}, which requires bounds for the derivatives of $(\Dscr h)(s)$ that are uniform with respect to $\tau$, $T$, and $U$.

\begin{lemma}
For $j \in \N_0$, we have that for any $N \geq 0$,
\begin{equation}
\label{eqn:Dscrhbounds}
\frac{d^j}{d\tau^j} (\Dscr h)(s) \ll_{j,N} \begin{dcases*}
T^{-N} & if $|\tau| \leq T$,	\\
\frac{TU}{\sqrt{|\tau|}} \left(\frac{TU}{|\tau|}\right)^{-N} \left(\frac{|\tau|}{T}\right)^{-2j} & if $T < |\tau| < TU$,	\\
\frac{TU}{\sqrt{|\tau|}} \left(\frac{|\tau|}{T}\right)^{-2j} & if $|\tau| \geq TU$.
\end{dcases*}
\end{equation}
\end{lemma}

\begin{proof}
From \eqref{eqn:Dscrhdefeq} and \eqref{eqn:Fscrhdefeq2}, we have that
\begin{equation}
\label{eqn:Dscrhjint1}
\frac{d^j}{d\tau^j} (\Dscr h)(s) = \sum_{\pm} T i^{-j} \int_{-\infty}^{\infty} g_{\pm}(u) \left(\sech \frac{\pi u}{U}\right)^{2C + \frac{1}{2} + 2\sigma} \left(\log \cosh^2 \frac{\pi u}{U}\right)^j e^{i \Psi_{\tau,\pm}(u)} \, du,	
\end{equation}
where the phase $\Psi_{\tau,\pm}(u)$ satisfies
\begin{align}
\label{eqn:Psitaupmdefeq}
\Psi_{\tau,\pm}(u) & \coloneqq \pm \frac{2\pi T u}{U} - \tau \log \cosh^2 \frac{\pi u}{U},	\\
\notag
\Psi_{\tau,\pm}'(u) & = \pm \frac{2\pi T}{U} - \frac{2\pi \tau}{U} \tanh \frac{\pi u}{U},	\\
\label{eqn:Psitaupmmdecay}
\Psi_{\tau,\pm}^{(m)}(u) & \ll_m \frac{|\tau|}{U^m} \sech^2 \frac{\pi u}{U} \quad \text{for $m \geq 2$.}
\end{align}
If $|\tau| > T$, this phase has a stationary point at $u_{0,\tau,\pm} \coloneqq \pm \frac{U}{\pi} \artanh \frac{T}{\tau}$, with
\[\Psi_{\tau,\pm}''(u_{0,\tau,\pm}) = -\frac{2\pi^2 \tau}{U^2} \left(1 - \frac{T^2}{\tau^2}\right).\]
If $|\tau| \leq T$, this phase has no stationary point.

In the latter case, we may simply repeatedly integrate by parts, where we antidifferentiate $\Psi_{\tau,\pm}'(u) e^{i\Psi_{\tau,\pm}(u)}$ and differentiate the rest (using, say, \cite[Lemma 3.1 (1)]{KPY19}), and we make use of the bounds \eqref{eqn:gdecay} for $g_{\pm}^{(m)}(u)$ and \eqref{eqn:Psitaupmmdecay} for $\Psi_{\tau,\pm}^{(m)}(u)$. Recalling our assumption that $T^{\e} \leq U \leq T^{1 - \e}$, this shows that if $|\tau| \leq T$, then for all for all $N \geq 0$,
\[\frac{d^j}{d\tau^j}(\Dscr h)(s) \ll_{j,N} T^{-N}.\]

In the former case, we let $\widetilde{\Omega}$ be a smooth compactly supported function equal to $1$ on $[-1,1]$, supported on $[-2,2]$, and having bounded derivatives, and we write the integral on the right-hand side of \eqref{eqn:Dscrhjint1} as
\begin{multline}
\label{eqn:Dscrhjint2}
\int_{-\infty}^{\infty} g_{\pm}(u) \left(\sech \frac{\pi u}{U}\right)^{2C + \frac{1}{2} + 2\sigma} \left(\log \cosh^2 \frac{\pi u}{U}\right)^j e^{i \Psi_{\tau,\pm}(u)} \widetilde{\Omega}\left(\frac{u - u_{0,\tau,\pm}}{\sqrt{|\Psi_{\tau,pm}''(u_{0,\tau,\pm})|}}\right) \, du	\\
+ \int_{-\infty}^{\infty} g_{\pm}(u) \left(\sech \frac{\pi u}{U}\right)^{2C + \frac{1}{2} + 2\sigma} \left(\log \cosh^2 \frac{\pi u}{U}\right)^j e^{i \Psi_{\tau,\pm}(u)} \left(1 - \widetilde{\Omega}\left(\frac{u - u_{0,\tau,\pm}}{\sqrt{|\Psi_{\tau,pm}''(u_{0,\tau,\pm})|}}\right)\right) \, du.
\end{multline}
Since the stationary point does not occur in the support of the integrand of the second term in \eqref{eqn:Dscrhjint2}, we may bound this second term using the same integration by parts argument. On the other hand, for the first term in \eqref{eqn:Dscrhjint2}, we simply bound this integral trivially, using the pointwise bounds \eqref{eqn:gdecay} for $g_{\pm}(u)$ and observing that $\log \cosh^2 \frac{\pi u}{U} \asymp \frac{T^2}{|\tau|^2}$ for all $u$ lying in the support of the integrand. There are two regimes of interest depending on the behaviour of $g_{\pm}(u)$. If $T < |\tau| < TU$, then for all $N \geq 0$, we have that $g_{\pm}(u) \ll_{N} |u|^{-N}$ for all $u$ lying in the support of the integrand by \eqref{eqn:gdecay}, so that
\[\frac{d^j}{d\tau^j}(\Dscr h)(s) \ll_{j,N} \frac{TU}{\sqrt{|\tau|}} \left(\frac{TU}{|\tau|}\right)^{-N} \left(\frac{|\tau|}{T}\right)^{-2j}.\]
If $|\tau| \geq TU$, then instead $g_{\pm}(u) \ll 1$, and so
\[\frac{d^j}{d\tau^j} (\Dscr h)(s) \ll_j \frac{TU}{\sqrt{|\tau|}} \left(\frac{|\tau|}{T}\right)^{-2j}.\qedhere\]
\end{proof}

\begin{proof}[Proof of {\hyperref[lem:HHpmtbounds]{Lemma \ref*{lem:HHpmtbounds}}} for $|t| \geq TU$]
We follow the exact same method of proof as that of \hyperref[lem:HHpmwzbounds]{Lemma \ref*{lem:HHpmwzbounds}}, replacing $w = u + iv$ with $1/2$ and $z = x + iy$ with $it$. Since $\HH_{\mu_F}^{-}(t) = \overline{\HH_{\mu_F}^{+}(-t)}$, it suffices to prove this for $\pm = +$. Moreover, we prove this only for $t > 0$, since an analogous (but easier) argument yields the same result for $t < 0$. For $\pm = +$ and $t > 0$, we write
\[\HH_{\mu_F}^{+}(t) = \frac{1}{\pi i} \int_{\sigma - i\infty}^{\sigma + i\infty} \frac{1}{2\sqrt{\pi}} (2\pi)^{-2s} \frac{\Gamma(s)}{\Gamma\left(\frac{1}{2} - s\right)} (\Dscr h)(s) \Gscr_{\mu_F}^{+}\left(\frac{1}{2} - s\right) G^{-}(s + it) \, ds,\]
where $0 < \sigma < 1/2$. We shift the contour of integration to $\Re(s) = -C + 1/4$, which picks up residues at $s = -it - \ell$ for $\ell \in \N_0$; by \eqref{eqn:Gscrmubounds} and \eqref{eqn:Dscrhbounds}, these are $\ll TU |t|^{2C - 3} e^{-3\pi t}$. Next, we let $\Omega : \R \to [0,1]$ be a smooth function equal to $1$ on $(-\infty,1]$, supported on $(-\infty,2]$, and having bounded derivatives. We then write the remaining integral as $\JJ_1 + \JJ_2$, where
\begin{align*}
\JJ_1 & \coloneqq \frac{1}{\pi i} \int_{-C + \frac{1}{4} - i\infty}^{-C + \frac{1}{4} + i\infty} \frac{1}{2\sqrt{\pi}} (2\pi)^{-2s} \frac{\Gamma(s)}{\Gamma\left(\frac{1}{2} - s\right)} (\Dscr h)(s) \Gscr_{\mu_F}^{+}\left(\frac{1}{2} - s\right) G^{-}(s + it) \Omega\left(\frac{\tau}{T}\right) \, ds,	\\
\JJ_2 & \coloneqq \frac{1}{\pi i} \int_{-C + \frac{1}{4} - i\infty}^{-C + \frac{1}{4} + i\infty} \frac{1}{2\sqrt{\pi}} (2\pi)^{-2s} \frac{\Gamma(s)}{\Gamma\left(\frac{1}{2} - s\right)} (\Dscr h)(s) \Gscr_{\mu_F}^{+}\left(\frac{1}{2} - s\right) G^{-}(s + it) \left(1 - \Omega\left(\frac{\tau}{T}\right)\right) \, ds.
\end{align*}
By \eqref{eqn:Gscrpmasymp}, \eqref{eqn:Gpmasymp}, and \eqref{eqn:Dscrhbounds}, we have that
\[\JJ_1 \ll_N T^{-N} t^{-C - \frac{1}{4}}\]
for any $N \geq 0$. For $\JJ_2$, we insert the asymptotic expansions \eqref{eqn:Gscrpmasymp}, \eqref{eqn:Gpmasymp}, as well as the asymptotic expansion (see, for example, \cite[(2.4)]{BK19})
\[(2\pi)^{-2s} \frac{\Gamma(s)}{\Gamma\left(\frac{1}{2} - s\right)} = |\tau|^{2\sigma - \frac{1}{2}} \exp\left(2i\tau \log \frac{|\tau|}{\pi e}\right) \alpha_{\sigma,M}(\tau) + O_{\sigma,M}((1 + |\tau|)^{-M})\]
for any $M > 0$, where $\alpha_{\sigma,M}(\tau)$ satisfies $|\tau|^j \alpha_{\sigma,M}^{(j)}(\tau) \ll_{j,\sigma,M} 1$ for all $j \in \N_0$. Bearing in mind the bounds \eqref{eqn:Dscrhbounds} for $(\Dscr h)(s)$, we see that the contribution from the error terms is $\ll_N t^{-N}$ for any $N \geq 0$, while the main term is of the shape
\[TU \int_{0}^{\infty} \tau^{C - \frac{5}{4}} (\tau + t)^{-C - \frac{1}{4}} e^{-i\Phi_t(\tau)} W_t(\tau) \left(1 - \Omega\left(\frac{\tau}{T}\right)\right) \, d\tau,\]
where $W_t(\tau)$ is a $1$-inert function and $\Phi_t(\tau)$ is as in \eqref{eqn:Phiytaudefeq}. We use the same integration by parts method as in the proof of \hyperref[lem:HHpmwzbounds]{Lemma \ref*{lem:HHpmwzbounds}} except with
\[X = \begin{dcases*}
TU Z^{C - \frac{5}{4}} t^{-C - \frac{1}{4}} & if $Z \leq t$,	\\
TU Z^{-\frac{3}{2}} & if $Z \geq t$.
\end{dcases*}\]
We deduce that this main term is $\ll_N T^{-N} t^{-C - 1/4}$ for any $N \geq 0$. Combined, this implies the desired bound $\HH_{\mu_F}^{+}(t) \ll U (Ut/T)^{-C}$ for $t \geq TU$.
\end{proof}

For $|t| \leq TU$, a more delicate analysis is required. Our first step for this range of $t$ is to insert the identity \eqref{eqn:HMellinFourier} for $\widehat{H}(2s)$ into the definition \eqref{eqn:HHtdefeq} of $\HH_{\mu_F}^{\pm}(t)$ and then truncate the integral over $u \in \R$.

\begin{lemma}
Fix $\delta \in (0,1)$ and let $\Omega$ be a smooth compactly supported function equal to $1$ on $[-1,1]$, supported on $[-2,2]$, and having bounded derivatives. We have that
\begin{multline}
\label{eqn:HHtdoubleint}
\HH_{\mu_F}^{\pm}(t) = \frac{1}{2\pi i} \int_{\sigma - i\infty}^{\sigma + i\infty} \frac{1}{\sqrt{\pi}} (2\pi)^{-2s} \frac{\Gamma(s)}{\Gamma\left(\frac{1}{2} - s\right)} \Gscr_{\mu_F}^{\pm}\left(\frac{1}{2} - s\right) G^{\mp}(s + it)	\\
\times \sum_{\pm_1} T \int_{-\infty}^{\infty} g_{\pm_1}(u) \left(\sech \frac{\pi u}{U}\right)^{2C + \frac{1}{2}} e\left(\pm_1 \frac{Tu}{U}\right) \left(\cosh^2 \frac{\pi u}{U}\right)^{-s} \Omega\left(\frac{u}{U^{\delta}}\right) \, du \, ds	\\
+ O_{\delta,N}\left((1 + |t|)^{-C} T^{-N}\right)
\end{multline}
for all $N \geq 0$.
\end{lemma}

\begin{proof}
From \eqref{eqn:Fscrhdefeq2}, we have that
\begin{multline*}
(\Dscr h)(s) = \sum_{\pm} T \int_{-\infty}^{\infty} g_{\pm}(u) \left(\sech \frac{\pi u}{U}\right)^{2C + \frac{1}{2}} e\left(\pm \frac{Tu}{U}\right) \left(\cosh^2 \frac{\pi u}{U}\right)^{-s} \Omega\left(\frac{u}{U^{\delta}}\right) \, du	\\
+ \sum_{\pm} T \int_{-\infty}^{\infty} g_{\pm}(u) \left(\sech \frac{\pi u}{U}\right)^{2C + \frac{1}{2}} e\left(\pm \frac{Tu}{U}\right) \left(\cosh^2 \frac{\pi u}{U}\right)^{-s} \left(1 - \Omega\left(\frac{u}{U^{\delta}}\right)\right) \, du.
\end{multline*}
For $j \in \N_0$, the $j$-th derivative with respect to $\tau$ of the second term is
\[\sum_{\pm} T i^{-j} \int_{-\infty}^{\infty} g_{\pm}(u) \left(\sech \frac{\pi u}{U}\right)^{2C + 2\sigma + \frac{1}{2}} \left(\log \cosh^2 \frac{\pi u}{U}\right)^j e^{i\Psi_{\tau,\pm}(u)} \left(1 - \Omega\left(\frac{u}{U^{\delta}}\right)\right) \, du,\]
where $\Psi_{\tau,\pm}(u)$ is as in \eqref{eqn:Psitaupmdefeq}. For $|\tau| \leq TU^{1 - 2\delta}$, we bound this trivially via \eqref{eqn:gdecay} and use the assumption that $U \geq T^{\e}$; we find that this is $\ll_{j,\delta,N} ((1 + |\tau|) T)^{-N}$ for all $N \geq 0$. For $|\tau| \geq TU^{1 - 2\delta}$, we integrate by parts and use the assumption that $U \leq T^{1 - \e}$. We again find that this is $\ll_{j,\delta,N} ((1 + |\tau|) T)^{-N}$ for all $N \geq 0$.

With this in hand, we follow the exact same method as the proof of \hyperref[lem:HHpmwzbounds]{Lemma \ref*{lem:HHpmwzbounds}} in order to find that
\begin{multline*}
\frac{1}{2\pi i} \int_{\sigma - i\infty}^{\sigma + i\infty} \frac{1}{\sqrt{\pi}} (2\pi)^{-2s} \frac{\Gamma(s)}{\Gamma\left(\frac{1}{2} - s\right)} \Gscr_{\mu_F}^{\pm}\left(\frac{1}{2} - s\right) G^{\mp}(s + it)	\\
\times \sum_{\pm} T \int_{-\infty}^{\infty} g_{\pm}(u) \left(\sech \frac{\pi u}{U}\right)^{2C + \frac{1}{2}} e\left(\pm \frac{Tu}{U}\right) \left(\cosh^2 \frac{\pi u}{U}\right)^{-s} \left(1 - \Omega\left(\frac{u}{U^{\delta}}\right)\right) \, du \, ds	\\
\ll_{\delta,N} (1 + |t|)^{-C} T^{-N}
\end{multline*}
for all $N \geq 0$.
\end{proof}

We are left with dealing with the double integral on the right-hand side of \eqref{eqn:HHtdoubleint}. We interchange the order of integration and focus on the integral over $\Re(s) = \sigma$. We may express this in terms of hypergeometric functions, which are defined by
\begin{equation}
\label{eqn:hypergeometric}
\prescript{}{2}{F}_1(a,b;c;x) \coloneqq \frac{\Gamma(c)}{\Gamma(a) \Gamma(b)} \sum_{\ell = 0}^{\infty} \frac{\Gamma(a + \ell) \Gamma(b + \ell)}{\Gamma(c + \ell)} \frac{x^{\ell}}{\ell!}.
\end{equation}

\begin{lemma}
For $0 < \sigma < 1/2$ and $|u| < \frac{1}{\pi} \log(1 + \sqrt{2})$, we have that
\begin{multline}
\label{eqn:contourintegral1}
\frac{1}{2\pi i} \int_{\sigma - i\infty}^{\sigma + i\infty} \frac{1}{\sqrt{\pi}} (2\pi)^{-2s} \frac{\Gamma(s)}{\Gamma\left(\frac{1}{2} - s\right)} \Gscr_{\mu_F}^{\pm}\left(\frac{1}{2} - s\right) G^{\mp}(s + it) (\cosh^2 \pi u)^{-s} \, ds	\\
= (1 \pm i) (2\pi)^{-1 - it} e^{\pm \frac{\pi t}{2}} \Gamma(it) (\tanh^2 \pi u)^{-it} \prescript{}{2}{F}_1\left(\frac{1}{2} + 2it_g,\frac{1}{2} - 2it_g;1 - it; -\sinh^2 \pi u\right)	\\
+ 2(1 \pm i) (2\pi)^{-2 - it} e^{\pm \frac{\pi t}{2}} \Gamma(-it) \Gamma\left(\frac{1}{2} + it + 2it_g\right) \Gamma\left(\frac{1}{2} + it - 2it_g\right) (e^{\mp \pi t} \cosh 2\pi t_g \mp \sinh \pi t)	\\
\times \prescript{}{2}{F}_1\left(\frac{1}{2} + 2it_g,\frac{1}{2} - 2it_g;1 + it; -\sinh^2 \pi u\right).
\end{multline}
\end{lemma}

\begin{proof}
We first restrict to the case that $u \neq 0$. From the definitions \eqref{eqn:Gpm} of $G^{\pm}(s)$ and \eqref{eqn:Gscrmupm} of $\Gscr_{\mu_F}^{\pm}(s)$, the left-hand side of \eqref{eqn:contourintegral1} is equal to
\begin{multline*}
(1 \mp i) (2\pi)^{-2 - it} e^{\pm \frac{\pi t}{2}} \frac{1}{2\pi i} \int_{\sigma - i\infty}^{\sigma + i\infty} \Gamma(s) \Gamma(s + it) \Gamma\left(\frac{1}{2} - s + 2it_g\right) \Gamma\left(\frac{1}{2} - s - 2it_g\right)	\\
\times (1 + 2 \cosh 2\pi t_g - e(\mp s)) (\cosh^2 \pi u)^{-s} \, ds.
\end{multline*}
Since $\cosh^2 \pi u > 1$, we shift the contour off to the right to the line $\Re(s) = \sigma'$, picking up residues at the poles at $s = 1/2 + \ell + 2it_g$ and $s = 1/2 + \ell - 2it_g$ for each $\ell \in \N_0$ with $\ell < \sigma' - 1/2$. The remaining integral over the line $\Re(s) = \sigma'$ tends to $0$ as $\sigma'$ tends to infinity via Stirling's formula, while each residue may be simplified via Euler's reflection formula. In this way, we find that the left-hand side of \eqref{eqn:contourintegral1} is equal to
\begin{multline*}
\pi (i \pm 1) (2\pi)^{-2 - it} e^{\pm \frac{\pi t}{2}} \cosech 4\pi t_g \sum_{\pm_1} \pm_1 (1 + 2 \cosh 2\pi t_g + e^{\pm \pm_1 4\pi t_g}) (\cosh^2 \pi u)^{-\frac{1}{2} \mp_1 2it_g}	\\
\times \sum_{\ell = 0}^{\infty} \frac{\Gamma\left(\frac{1}{2} + \ell \pm_1 2it_g\right) \Gamma\left(\frac{1}{2} + \ell + it \pm_1 2it_g\right)}{\Gamma\left(1 + \ell \pm_1 4it_g\right)} \frac{(\sech^2 \pi u)^{\ell}}{\ell!}.
\end{multline*}
The sum over $\ell \in \N_0$ is
\[\frac{\Gamma\left(\frac{1}{2} \pm_1 2it_g\right) \Gamma\left(\frac{1}{2} + it \pm_1 2it_g\right)}{\Gamma(1 \pm_1 4it_g)} \prescript{}{2}{F}_1\left(\frac{1}{2} \pm_1 2it_g, \frac{1}{2} + it \pm_1 2it_g; 1 \pm_1 4it_g; \sech^2 \pi u\right).\]
From the definition \eqref{eqn:hypergeometric} of the hypergeometric function together with \cite[9.131.1, 9.131.2]{GR15}, the sum over $\ell \in \N_0$ is equal to
\begin{multline*}
(\cosh^2 \pi u)^{\frac{1}{2} \pm_1 2it_g} (\tanh^2 \pi u)^{-it} \Gamma(it) \prescript{}{2}{F}_1\left(\frac{1}{2} + 2it_g,\frac{1}{2} - 2it_g;1 - it; -\sinh^2 \pi u\right)	\\
+ (\cosh^2 \pi u)^{\frac{1}{2} \pm_1 2it_g} \frac{\Gamma\left(\frac{1}{2} + it \pm_1 2it_g\right)}{\Gamma\left(\frac{1}{2} - it \pm_1 2it_g\right)} \Gamma(-it) \prescript{}{2}{F}_1\left(\frac{1}{2} + 2it_g,\frac{1}{2} - 2it_g;1 + it; -\sinh^2 \pi u\right).
\end{multline*}
This is valid for $|u| < \frac{1}{\pi} \log(1 + \sqrt{2})$, so that $\sinh^2 \pi u < 1$. Next, we have that
\[\sum_{\pm_1} \pm_1 (1 + 2 \cosh 2\pi t_g + e^{\pm \pm_1 4\pi t_g}) = \pm 2 \sinh 4\pi t_g,\]
while by Euler's reflection formula,
\begin{multline*}
\sum_{\pm_1} \pm_1 (1 + 2 \cosh 2\pi t_g + e^{\pm \pm_1 4\pi t_g}) \frac{\Gamma\left(\frac{1}{2} + it \pm_1 2it_g\right)}{\Gamma\left(\frac{1}{2} - it \pm_1 2it_g\right)}	\\
= \pm \frac{2}{\pi} \sinh 4\pi t_g \Gamma\left(\frac{1}{2} + it + 2it_g\right) \Gamma\left(\frac{1}{2} + it - 2it_g\right) (e^{\mp \pi t} \cosh 2\pi t_g \mp \sinh \pi t).
\end{multline*}
This yields the identity \eqref{eqn:contourintegral1} for $u \neq 0$; finally, by continuity, it additionally extends to the value $u = 0$.
\end{proof}

\begin{remark}
The hypergeometric functions appearing on the right-hand side of \eqref{eqn:contourintegral1} can be expressed in terms of associated Legendre functions of the first kind: from \cite[8.702]{GR15},
\[\prescript{}{2}{F}_1\left(\frac{1}{2} + 2it_g,\frac{1}{2} - 2it_g;1 \mp it; -\sinh^2 \pi u\right) = (\tanh^2\pi u)^{\pm \frac{it}{2}} \Gamma(1 \mp it) P_{-\frac{1}{2} + 2it_g}^{\pm it}(\cosh 2\pi u).\]
\end{remark}

\begin{proof}[Proof of {\hyperref[lem:HHpmtbounds]{Lemma \ref*{lem:HHpmtbounds}}} for $|t| \leq TU$]
From \eqref{eqn:HHtdoubleint}, \eqref{eqn:contourintegral1}, and Stirling's formula, it suffices to show that the quantities
\begin{multline}
\label{eqn:oscillatory}
\int_{-\infty}^{\infty} g_{\pm}(u) \left(\sech \frac{\pi u}{U}\right)^{2C + \frac{1}{2}} e\left(\pm \frac{Tu}{U}\right) \left(\tanh^2 \frac{\pi u}{U}\right)^{-it} \Omega\left(\frac{u}{U^{\delta}}\right)	\\
\times \prescript{}{2}{F}_1\left(\frac{1}{2} + 2it_g,\frac{1}{2} - 2it_g;1 - it; -\sinh^2 \frac{\pi u}{U}\right) \, du
\end{multline}
and
\begin{multline}
\label{eqn:nonoscillatory}
\int_{-\infty}^{\infty} g_{\pm}(u) \left(\sech \frac{\pi u}{U}\right)^{2C + \frac{1}{2}} e\left(\pm \frac{Tu}{U}\right) \Omega\left(\frac{u}{U^{\delta}}\right)	\\
\times \prescript{}{2}{F}_1\left(\frac{1}{2} + 2it_g,\frac{1}{2} - 2it_g;1 + it; -\sinh^2 \frac{\pi u}{U}\right) \, du
\end{multline}
are both $\ll_{\delta} \frac{U}{T\sqrt{|t|}}$ for $|t| \leq \frac{T}{U}$ and are $\ll_{\delta} \frac{U}{T\sqrt{|t|}} (\frac{U|t|}{T})^{-C/2}$ for $\frac{T}{U} < |t| \leq TU$.

By \cite[9.103.1]{GR15}, we have that for any $m \in \N_0$,
\begin{multline*}
\frac{d}{dx} \prescript{}{2}{F}_1\left(m - \frac{1}{2} + 2it_g,m - \frac{1}{2} - 2it_g;m \mp it; x\right)	\\
= \frac{\left(m - \frac{1}{2}\right)^2 + 4t_g^2}{m \mp it} \prescript{}{2}{F}_1\left(m + \frac{1}{2} + 2it_g,m + \frac{1}{2} - 2it_g;m + 1 \mp it;x\right).
\end{multline*}
Since
\[\prescript{}{2}{F}_1\left(m - \frac{1}{2} + 2it_g,m - \frac{1}{2} - 2it_g;m \mp it;x\right) \ll_m 1\]
for $m \in \N_0$ and $|x| \leq 1/2$, as is immediate from the series representation \eqref{eqn:hypergeometric} of this hypergeometric function, we deduce that for all $m \in \N_0$ and for $|u| \leq 2U^{\delta}$,
\begin{equation}
\label{eqn:2F1derivbounds}
\frac{d^m}{du^m} \prescript{}{2}{F}_1\left(\frac{1}{2} + 2it_g,\frac{1}{2} - 2it_g;1 \pm it; -\sinh^2 \frac{\pi u}{U}\right) \ll_m U^{-m}.
\end{equation}

With this in hand, we proceed to bound \eqref{eqn:oscillatory}. We write this integral as
\begin{equation}
\label{eqn:gpmhypergeometricint1}
\int_{-\infty}^{\infty} g_{\pm}(u) \left(\sech \frac{\pi u}{U}\right)^{2C + \frac{1}{2}} e^{i\Phi_{\pm}(u)} \Omega\left(\frac{u}{U^{\delta}}\right) \prescript{}{2}{F}_1\left(\frac{1}{2} + 2it_g,\frac{1}{2} - 2it_g;1 - it; -\sinh^2 \frac{\pi u}{U}\right) \, du,
\end{equation}
where
\[\Phi_{\pm}(u) \coloneqq \pm \frac{2\pi Tu}{U} - t \log \tanh^2 \frac{\pi u}{U},\]
so that
\begin{equation}
\label{eqn:Phipmmdeceay}
\Phi_{\pm}'(u) = \pm \frac{2\pi T}{U} - \frac{4\pi t}{U} \cosech \frac{2\pi u}{U}, \qquad \left|\Phi_{\pm}^{(m)}(u)\right| \asymp_m \frac{|t|}{|u|^m} \quad \text{for $m \geq 2$ and $|u| \leq U$.}
\end{equation}
If $|t| \ll TU^{\delta - 1}$, this phase has a stationary point at $u_{0,\pm} \coloneqq \pm \frac{U}{2\pi} \arsinh \frac{2t}{T}$ with
\[\Phi_{\pm}''(u_{0,\pm}) = -\frac{2\pi^2 T^2}{tU^2} \sqrt{\frac{4t^2}{T^2} + 1},\]
while if $|t| \gg TU^{\delta - 1}$, this stationary point does not occur in the support of the integrand.

In the latter case, we may simply repeatedly integrate by parts by, where we antidifferentiate $\Phi_{\pm}'(u) e^{i\Phi_{\pm}(u)}$ and differentiate the rest (using, say, \cite[Lemma 3.1 (1)]{KPY19}), and we make use of the bounds \eqref{eqn:2F1derivbounds} for derivatives of the hypergeometric function, \eqref{eqn:gdecay} for $g_{\pm}^{(m)}(u)$, and \eqref{eqn:Phipmmdeceay} for $\Phi_{\pm}^{(m)}(u)$. This shows that if $TU^{\delta - 1} \ll |t| \ll TU$, then for all for all $N \geq 0$, \eqref{eqn:gpmhypergeometricint1} is
\[\ll_N \frac{U}{T\sqrt{|t|}} \left(\frac{U|t|}{T}\right)^{-N}.\]

In the former case, we let $\widetilde{\Omega}$ be a smooth compactly supported function equal to $1$ on $[-1,1]$, supported on $[-2,2]$, and having bounded derivatives, and we write \eqref{eqn:gpmhypergeometricint1} as
\begin{multline}
\label{eqn:gpmhypergeometricint2}
\int_{-\infty}^{\infty} g_{\pm}(u) \left(\sech \frac{\pi u}{U}\right)^{2C + \frac{1}{2}} e^{i\Phi_{\pm}(u)} \Omega\left(\frac{u}{U^{\delta}}\right) \widetilde{\Omega}\left(\frac{u - u_{0,\pm}}{\sqrt{|\Phi_{\pm}''(u_{0,\pm})|}}\right)	\\
\times \prescript{}{2}{F}_1\left(\frac{1}{2} + 2it_g,\frac{1}{2} - 2it_g;1 - it; -\sinh^2 \frac{\pi u}{U}\right) \, du	\\
+ \int_{-\infty}^{\infty} g_{\pm}(u) \left(\sech \frac{\pi u}{U}\right)^{2C + \frac{1}{2}} e^{i\Phi_{\pm}(u)} \Omega\left(\frac{u}{U^{\delta}}\right) \left(1 - \widetilde{\Omega}\left(\frac{u - u_{0,\pm}}{\sqrt{|\Phi_{\pm}''(u_{0,\pm})|}}\right)\right)	\\
\times \prescript{}{2}{F}_1\left(\frac{1}{2} + 2it_g,\frac{1}{2} - 2it_g;1 - it; -\sinh^2 \frac{\pi u}{U}\right) \, du.
\end{multline}
Since the stationary point does not occur in the support of the integrand of the second term in \eqref{eqn:gpmhypergeometricint2}, we may bound this second term using the same integration by parts argument. On the other hand, for the first term in \eqref{eqn:gpmhypergeometricint2}, we simply bound this integral trivially, using the pointwise bounds \eqref{eqn:gdecay} for $g_{\pm}(u)$ and \eqref{eqn:2F1derivbounds} for the hypergeometric function. There are two regimes of interest depending on the behaviour of $g_{\pm}(u)$. If $T/U < |t| \ll TU^{1 - \delta}$, then for all $N \geq 0$, we have that $g_{\pm}(u) \ll_{N} |u|^{-N}$ for all $u$ lying in the support of the integrand by \eqref{eqn:gdecay}, so that \eqref{eqn:gpmhypergeometricint1} is
\[\ll_N \frac{U}{T\sqrt{|t|}} \left(\frac{U|t|}{T}\right)^{-N}.\]
If $|t| \leq T/U$, then instead $g_{\pm}(u) \ll 1$, and so \eqref{eqn:gpmhypergeometricint1} is
\[\ll \frac{U}{T\sqrt{|t|}}.\]

To bound \eqref{eqn:nonoscillatory}, on the other hand, we merely use integration by parts, integrating $e(\pm Tu/U)$ and differentiating the rest, and bearing in mind the assumption that $U \leq T^{1 - \e}$ as well as the bounds \eqref{eqn:2F1derivbounds} for derivatives of the hypergeometric function and the bounds \eqref{eqn:gdecay} for $g_{\pm}^{(m)}(u)$. We find that the integral \eqref{eqn:nonoscillatory} is $\ll_{\delta,N} T^{-N}$ for all $N \geq 0$, which is sufficient since $|t| \leq TU$.
\end{proof}

\section{Bounds for Second Moments of \texorpdfstring{$L$}{L}-Functions}
\label{sect:secondmomentbounds}

We next focus on bounding the dual moment \eqref{eqn:centraldual} with $\HH_{\mu_F}^{\pm}(t)$ replaced by the indicator function of the interval $[-T,T]$. Towards this, we prove the following.

\begin{proposition}
\label{prop:centraldualbounds}
Let $F$ be a Hecke--Maa\ss{} cusp form for $\SL_3(\Z)$. Let $q = q_1 q_2$ be a positive integer with $(q_1,q_2) = 1$. Let $\chi_1$ be a primitive Dirichlet character of conductor $q_1$, and set $\chi \coloneqq \chi_1 \chi_{0(q_2)}$. For $T \geq 1$, we have that
\begin{equation}
\label{eqn:centraldualbounds}
\frac{1}{\varphi(q)} \sum_{\psi \hspace{-.25cm} \pmod{q}} \int_{-T}^{T} \left|L\left(\frac{1}{2} + it, F \otimes \psi\right) L\left(\frac{1}{2} - it,\overline{\psi}\right) \ZZ_{\chi}(\psi;t)\right| \, dt \ll_{F,\e} (q_1 T)^{\frac{5}{4}} q_2^{\frac{1}{2}} (qT)^{\e}.
\end{equation}
\end{proposition}

The proof of \hyperref[prop:centraldualbounds]{Proposition \ref*{prop:centraldualbounds}} relies upon second moment bounds for $L(1/2 + it,F \otimes \psi)$ and $L(1/2 - it,\overline{\psi}) \ZZ_{\chi}(\psi;t)$. In doing so, we make use of the observation that by \hyperref[cor:Zbounds]{Corollary \ref*{cor:Zbounds}}, the left-hand side of \eqref{eqn:centraldualbounds} is equal to
\begin{multline}
\label{eqn:centraldualboundsLHS}
\frac{1}{\varphi(q)} \sum_{q_3' \mid q_3} \sum_{\psi_1 \hspace{-.25cm} \pmod{q_1}} \hspace{.2cm} \sideset{}{^{\star}} \sum_{\psi_3' \hspace{-.25cm} \pmod{q_3'}} \int_{-T}^{T} \left|L^{q_2/q_3'}\left(\frac{1}{2} + it,F \otimes \psi_1 \psi_3'\right)\right|	\\
\times \left|L^{q_2/q_3'}\left(\frac{1}{2} - it,\overline{\psi_1 \psi_3'}\right) \ZZ_{\chi}(\psi_1 \psi_3' \psi_{0(q_2/q_3')};t)\right| \, dt,
\end{multline}
where we have written $q_2 = q_3 q_4$ with $q_3$ squarefree, $q_4$ squarefull, and $(q_3,q_4) = 1$. We shall bound the expression \eqref{eqn:centraldualboundsLHS} via the Cauchy--Schwarz inequality, which in turn requires us to bound the second moment of $L(1/2 + it,F \otimes \psi)$. We achieve this via an application of Gallagher's hybrid large sieve.

\begin{proposition}
\label{prop:secondmomentFlargesieve}
Let $F$ be a Hecke--Maa\ss{} cusp form for $\SL_3(\Z)$. Let $q = q_1 q_2$ be a positive integer with $(q_1,q_2) = 1$, and write $q_2 = q_3 q_4$ with $q_3$ squarefree, $q_4$ squarefull, and $(q_3,q_4) = 1$. For $T \geq 1$ and a divisor $q_3'$ of $q_3$, we have that
\[\sum_{\psi_1 \hspace{-.25cm} \pmod{q_1}} \hspace{.2cm} \sideset{}{^{\star}} \sum_{\psi_3' \hspace{-.25cm} \pmod{q_3'}} \int_{-T}^{T} \left|L^{q_2/q_3'}\left(\frac{1}{2} + it,F \otimes \psi_1 \psi_3'\right)\right|^2 \, dt \ll_{F,\e} (q_1 q_3' T)^{\frac{3}{2}} (qT)^{\e}.\]
\end{proposition}

\begin{proof}
By a standard application of the approximate functional equation \cite[Theorem 5.3]{IK04}, $L(1/2 + it,F \otimes \psi_1 \psi_3')$ may be written as a sum of two Dirichlet polynomials with coefficients of the form $A_F(1,n) \psi(n) \psi_3'(n) n^{-1/2 - it}$ that are each of length $O_{F,\e}((q_1 q_3' (|t| + 1))^{3/2 + \e})$. The result then follows from the Rankin--Selberg bound $\sum_{n \leq N} |A_F(1,n)|^2 \ll_{F,\e} N^{1 + \e}$ in conjunction with Gallagher's hybrid large sieve \cite[Theorem 2]{Gal70}, which states that for any sequence of complex numbers $(a_n)$ and for $q \in \N$ and $T,N \geq 1$, we have that
\[\sum_{\psi \hspace{-.25cm} \pmod{q}} \int_{-T}^{T} \left|\sum_{n \leq N} a_n \psi(n) n^{-it}\right|^2 \, dt \ll \sum_{n \leq N} (qT + n) |a_n|^2.\qedhere\]
\end{proof}

We turn our attention to the second moment of $L(1/2 - it,\overline{\psi}) \ZZ_{\chi}(\psi;t)$. We shall shortly show the following.

\begin{proposition}
\label{prop:secondmomenthybrid}
Let $q = q_1 q_2$ be a positive integer with $(q_1,q_2) = 1$, and write $q_2 = q_3 q_4$ with $q_3$ squarefree, $q_4$ squarefull, and $(q_3,q_4) = 1$. Let $\chi_1$ be a primitive Dirichlet character of conductor $q_1$, and set $\chi \coloneqq \chi_1 \chi_{0(q_2)}$. For $T \geq 1$ and a divisor $q_3'$ of $q_3$, we have that
\begin{equation}
\label{eqn:secondmomenthybrid}
\sum_{\psi_1 \hspace{-.25cm} \pmod{q_1}} \hspace{.2cm} \sideset{}{^{\star}} \sum_{\psi_3' \hspace{-.25cm} \pmod{q_3'}} \int_{-T}^{T} \left|L^{q_2/q_3'}\left(\frac{1}{2} - it,\overline{\psi_1 \psi_3'}\right) \ZZ_{\chi}(\psi_1 \psi_3' \psi_{0(q_2/q_3')};t)\right|^2 \, dt \ll_{\e} \frac{q^3}{{q_3'}^2} T (qT)^{\e}.
\end{equation}
\end{proposition}

\hyperref[prop:secondmomentFlargesieve]{Propositions \ref*{prop:secondmomentFlargesieve}} and \ref{prop:secondmomenthybrid} immediately combine to yield \hyperref[prop:centraldualbounds]{Proposition \ref*{prop:centraldualbounds}}.

\begin{proof}[Proof of {\hyperref[prop:centraldualbounds]{Proposition \ref*{prop:centraldualbounds}}}]
We rewrite the left-hand side of \eqref{eqn:centraldualbounds} as \eqref{eqn:centraldualboundsLHS}. The result then follows by the Cauchy--Schwarz inequality coupled with \hyperref[prop:secondmomentFlargesieve]{Propositions \ref*{prop:secondmomentFlargesieve}} and \ref{prop:secondmomenthybrid}.
\end{proof}

A similar estimate to that in \hyperref[prop:secondmomenthybrid]{Proposition \ref*{prop:secondmomenthybrid}} arises in \cite{PY23} involving the \emph{fourth} moment of Dirichlet $L$-functions weighted by $|\ZZ_{\chi}(\psi;t)|$; in our setting, we instead have a \emph{second} moment of Dirichlet $L$-functions weighted by the \emph{square} of $|\ZZ_{\chi}(\psi;t)|$. Inspired by \cite{PY23}, our strategy towards proving \hyperref[prop:secondmomenthybrid]{Proposition \ref*{prop:secondmomenthybrid}} involves breaking up the sum over Dirichlet characters $\psi_1$ modulo $q_1$ based on the size of $\ZZ_{\chi}(\psi_1 \psi_3' \psi_{0(q_2/q_3')};t)$. From \hyperref[cor:Zbounds]{Corollary \ref*{cor:Zbounds}}, the size of $\ZZ_{\chi}(\psi_1 \psi_3' \psi_{0(q_2/q_3')};t)$ is essentially determined by the size of the character sum \eqref{eqn:gchipsidefeq}. As observed in \cite{PY23}, this character sum has square-root cancellation for \emph{most} Dirichlet characters, yet can be larger for certain exceptional characters; moreover, these exceptional Dirichlet characters form a \emph{coset} of the group of Dirichlet characters modulo $q$. For this reason, we require the following hybrid second moment bound for Dirichlet $L$-functions averaged over cosets of Dirichlet characters.

\begin{theorem}[{Cf.~\cite[Theorem 1.3]{GY25}}]
\label{thm:secondmomentcoset}
Let $\psi$ be a Dirichlet character modulo a positive integer $q$. Then for $T \geq 1$ and $q'$ a divisor of $q$, we have that
\begin{equation}
\label{eqn:secondmomentcosetbounds}
\sum_{\psi' \hspace{-.25cm} \pmod{q'}} \int_{-T}^{T} \left|L\left(\frac{1}{2} + it,\psi \psi'\right)\right|^2 \, dt \ll_{\e} \begin{dcases*}
\frac{q^{1/2}}{{q'}^{1/2}} (qT)^{\e} & if $q' \leq \frac{q^{1/3}}{T^{2/3}}$,	\\
q' T (qT)^{\e} & if $q' \geq \frac{q^{1/3}}{T^{2/3}}$.
\end{dcases*}
\end{equation}
\end{theorem}

\begin{remark}
One can improve the bounds \eqref{eqn:secondmomentcosetbounds} in the range $q' \leq q^{1/9} T^{-8/9}$ to $O_{\e}(q' q^{1/3} T^{4/3} q^{\e})$ by applying H\"{o}lder's inequality and invoking the hybrid sixth moment bound
\[\int_{-T}^{T} \left|L\left(\frac{1}{2} + it,\psi\right)\right|^6 \, dt \ll_{\e} qT^2 (qT)^{\e}\]
for primitive Dirichlet characters $\psi$ modulo $q$ due to Petrow and Young \cite[Theorems 1.2 and 1.3]{PY23}. Nonetheless, the weaker bound $O_{\e}(q^{1/2} {q'}^{-1/2} q^{\e})$ in this range that we prove below is more than sufficient for our purposes.
\end{remark}

\hyperref[thm:secondmomentcoset]{Theorem \ref*{thm:secondmomentcoset}} should be compared to \cite[Theorem 1.4]{PY23}, where an analogous fourth moment bound is proven. The proof of this fourth moment bound is quite involved, whereas the second moment bound in \hyperref[thm:secondmomentcoset]{Theorem \ref*{thm:secondmomentcoset}} is comparatively straightforward to show once we appeal to the following estimates due to Heath-Brown.

\begin{lemma}[Heath-Brown {\cite[Lemma 9]{H-B78}}]
Given a primitive Dirichlet character $\psi$ modulo a positive integer $q$ and integers $h,n \in \Z$, define
\begin{equation}
\label{eqn:HBsumdefeq}
S(q;\psi,h,n) \coloneqq \sum_{u \in \Z/q\Z} \psi(u + h) \overline{\psi}(u) e\left(\frac{nu}{q}\right).
\end{equation}
Then for a divisor $q'$ of $q$ and for $A,B \geq 1$, we have that
\begin{align}
\label{eqn:HBbounds1}
\sum_{1\leq h\leq A} |S(q;\psi,4hq',0)| & \ll_{\e} A q' q^{\e},	\\
\label{eqn:HBbounds2}
\sum_{1\leq h\leq A} \sum_{1\leq n\leq B} |S(q;\psi,4hq',n)| & \ll_{\e}
\begin{dcases*}
(Aq')^{1/4} q^{\frac{3}{4} + \e} & if $A B^{4/3} \leq q' q^{1/3}$,	\\
\frac{AB q^{\frac{1}{2} + \e}}{{q'}^{1/2}} & if $A B^{4/3} \geq q' q^{1/3}$.
\end{dcases*}
\end{align}
\end{lemma}

\begin{proof}[Proof of {\hyperref[thm:secondmomentcoset]{Theorem \ref*{thm:secondmomentcoset}}}]
We initially suppose that $\psi$ is a \emph{primitive} Dirichlet character modulo $q$; later, we shall remove this assumption. Let $\Omega_0$ be a fixed smooth nonnegative function that is equal to $1$ on $[-1,1]$ and supported on $[-3/2,3/2]$ with bounded derivatives and let $\Omega_1$ be a smooth real-valued function supported on $[1,2]$ with bounded derivatives. We shall show that for $T \geq 1$, $q' \mid q$, and for $1 \leq N \leq (qT)^{1/2 + \e}$,
\begin{multline}
\label{eqn:secondmomentcosetdesired}
\frac{1}{N} \sum_{\psi' \hspace{-.25cm} \pmod{4q'}} \int_{-\infty}^{\infty} \Omega_0\left(\frac{t}{T}\right) \left|\sum_{n = 1}^{\infty} \Omega_1\left(\frac{n}{N}\right) \psi(n) \psi'(n) n^{-it}\right|^2 \, dt	\\
\ll_{\e} q'T + \begin{dcases*}
\frac{q^{1/2}}{{q'}^{1/2}} q^{\e} & if $q' T \leq N \leq \min\left\{\frac{q^3}{{q'}^6 T^3}, q, (qT)^{1/2}\right\}$,	\\
0 & otherwise.
\end{dcases*}
\end{multline}
The bound \eqref{eqn:secondmomentcosetdesired} shall subsequently be used to prove the desired bound \eqref{eqn:secondmomentcosetbounds}. Note that the second term on the right-hand side of \eqref{eqn:secondmomentcosetdesired} can only occur if $q' \leq \min\{\frac{q^{3/7}}{T^{4/7}},\frac{q}{T},\frac{q^{1/2}}{T^{1/2}}\}$, and that it dominates the first term only if $q' \leq \frac{q^{1/3}}{T^{2/3}}$, which in turn can only occur provided that $T \leq q^{1/2}$.

Opening up the square and evaluating the sum over $\psi' \pmod{4q'}$ via character orthogonality, we find that the left-hand side of \eqref{eqn:secondmomentcosetdesired} is equal to
\[\frac{\varphi(4q') T}{N} \sum_{\substack{n,m = 1 \\ n\equiv m \hspace{-.25cm} \pmod{4q'}}}^{\infty} \psi(m) \overline{\psi}(n) \widehat{\Omega_0}\left(\frac{T}{2\pi}\log \frac{m}{n}\right) \Omega_1\left(\frac{m}{N}\right) \Omega_1\left(\frac{n}{N}\right),\]
where $\widehat{\Omega_0}$ denotes the Fourier transform of $\Omega_0$. We write this expression as $\mathcal{D} + \mathcal{OD}$, where $\mathcal{D}$ is the diagonal term consisting of the summands for which $n = m$, while $\mathcal{OD}$ is the remaining off-diagonal term consisting of the summands for which $n \neq m$.

The diagonal term is easily dealt with: we have that
\[\mathcal{D} \ll \frac{q'T}{N} \widehat{\Omega_0}(0) \sum_{n = 1}^{\infty} \Omega_1\left(\frac{n}{N}\right)^2 \ll q' T.\]
For the off-diagonal term, we write $h \coloneqq \frac{|m - n|}{4q'}$, so that
\[\mathcal{OD} = \frac{2\varphi(4q') T}{N} \Re\left(\sum_{h = 1}^{\infty} \sum_{n = 1}^{\infty} \Omega_{N,T,4hq'}(n) \psi(n + 4hq') \overline{\psi}(n)\right),\]
where
\[\Omega_{N,T,4hq'}(n) \coloneqq \widehat{\Omega_0}\left(\frac{T}{2\pi}\log\left(1 + \frac{4hq'}{n}\right)\right) \Omega_1\left(\frac{n + 4hq'}{N}\right) \Omega_1\left(\frac{n}{N}\right).\]
Since $\Omega_0$ is compactly supported, its Fourier transform $\widehat{\Omega_0}$ is rapidly decaying, which allows us to truncate the sum over $h \in \N$ to $h \leq \frac{N}{q'T} (qT)^{\e}$ at a cost of a negligibly small error term. In particular, the off-diagonal term is negligibly small unless $N \geq q' T (qT)^{-\e}$, which we henceforth assume; since $N \leq (qT)^{1/2 + \e}$, we thereby assume that $q' \leq (q/T)^{1/2} (qT)^{\e}$. We then break up the sum over $n \in \N$ into residue classes $u$ modulo $q$ and apply the Poisson summation formula to this sum, yielding
\[\mathcal{OD} = \frac{2\varphi(4q') T}{Nq} \Re\left(\sum_{1 \leq h \leq \frac{N}{q'T} (qT)^{\e}} \sum_{n = -\infty}^{\infty} \widehat{\Omega_{N,T,4hq'}}\left(\frac{n}{q}\right) S(q;\psi,4hq',n)\right) + o_{\e}(1)\]
for any $\e > 0$, where $S(q;\psi,4hq',n)$ is as in \eqref{eqn:HBsumdefeq}.

We observe that
\[\widehat{\Omega_{N,T,4hq'}}\left(\frac{n}{q}\right) = N \int_{-\infty}^{\infty} \widehat{\Omega_0}\left(\frac{T}{2\pi}\log\left(1+\frac{4hq'}{Nx}\right)\right) \Omega_1\left(x + \frac{4hq'}{N}\right) \Omega_1(x) e\left(-\frac{nNx}{q}\right) \, dx.\]
This is $O(N)$ if $|n| \leq \frac{q}{N} (qT)^{\e}$ and is negligibly small otherwise via repeated integration by parts, which allows us to truncate the sum over $n \in \Z$ to $|n| \leq \frac{q}{N} (qT)^{\e}$. We deduce that
\[\mathcal{OD} \ll_{\e} \frac{q' T}{q} \sum_{1 \leq h \leq \frac{N}{q'T} (qT)^{\e}} \sum_{0 \leq n \leq \frac{q}{N} (qT)^{\e}} |S(q;\psi,4hq',n)| + 1,\]
where for $n < 0$ we have used the fact that $S(q;\psi,k,n) = e\left(-\frac{nk}{q}\right) S(q;\psi,k,-n)$ and made the change of variables $n \mapsto -n$.

It remains to apply the bounds \eqref{eqn:HBbounds1} and \eqref{eqn:HBbounds2} with $A = \frac{N}{q'T} (qT)^{\e}$ and additionally $B = \frac{q}{N} (qT)^{\e}$ if $N \leq q (qT)^{\e}$. We deduce that
\[\mathcal{OD} \ll_{\e} \begin{dcases*}
1 & if $N \leq q' T (qT)^{-\e}$,	\\
\frac{q^{1/2}}{{q'}^{1/2}} (qT)^{\e} & if $q' T (qT)^{-\e} \leq N \leq (qT)^{\e} \min\left\{\frac{q^3}{{q'}^6 T^3}, q, (qT)^{1/2}\right\}$,	\\
\frac{N^{1/4} q' T^{3/4}}{q^{1/4}} (qT)^{\e} & if $\max\left\{q' T (qT)^{-\e},\frac{q^3}{{q'}^6 T^3} (qT)^{\e}\right\} \leq N \leq (qT)^{\e} \min\{q,(qT)^{1/2}\}$,	\\
\frac{N q'}{q} (qT)^{\e} & if $\max\{q'T (qT)^{-\e}, q (qT)^{\e}\} \leq N \leq (qT)^{1/2 + \e}$.
\end{dcases*}\]
This gives the claimed bounds \eqref{eqn:secondmomentcosetdesired}.

Now we show that \eqref{eqn:secondmomentcosetbounds} can be deduced from \eqref{eqn:secondmomentcosetdesired}. We let $q'$ be a divisor of $q$ and we let $\psi$ be a Dirichlet character modulo $q$, which we no longer enforce to be primitive. We may assume without loss of generality that the conductor $q_{\psi}$ of $\psi$ is a multiple of $q'$, since otherwise we may replace $\psi$ with $\psi \chi'$, where $\chi'$ is a Dirichlet character modulo $q'$ such that $\psi \chi'$ is primitive and has conductor that is a multiple of $q'$, and then make the change of variables $\psi' \mapsto \overline{\chi'} \psi'$ on the left-hand side of \eqref{eqn:secondmomentcosetbounds}. We write $q_{\psi} = q_1 q_2$ and $q' = q_1 q_2'$, where $q_1,q_2,q_2'$ are such that $(q_1,q_2) = (q_1,q_2') = 1$, $q_2' \mid q_2$, and $p \mid \frac{q_2}{q_2'}$ whenever $p \mid q_2'$. We may then write $\psi = \psi_1 \psi_2 \psi_{0(q)}$, where $\psi_1$ is a Dirichlet character modulo $q_1$ and $\psi_2$ is a primitive Dirichlet character modulo $q_2$. Finally, for a Dirichlet character $\psi'$ modulo $q'$, we write $\psi' = \psi_1' \psi_2'$, where $\psi_1',\psi_2'$ are Dirichlet characters modulo $q_1,q_2'$. We then have that
\[\sum_{\psi' \hspace{-.25cm} \pmod{q'}} \int_{-T}^{T} \left|L\left(\frac{1}{2} + it,\psi \psi'\right)\right|^2 \, dt \ll_{\e} q^{\e} \sum_{\psi_1' \hspace{-.25cm} \pmod{q_1}} \sum_{\psi_2' \hspace{-.25cm} \pmod{q_2'}} \int_{-T}^{T} \left|L\left(\frac{1}{2} + it,\psi_1 \psi_1' \psi_2 \psi_2'\right)\right|^2 \, dt.\]

We make the change of variables $\psi_1' \mapsto \overline{\psi_1} \psi_1'$ and write $\psi_1' = \psi_{0(q_1)} \psi_3'$, where $\psi_3'$ is a primitive character modulo $q_3'$ for some divisor $q_3'$ of $q_1$, so that this is in turn
\[\ll_{\e} q^{\e} \sum_{q_3' \mid q_1} \hspace{.15cm} \sideset{}{^{\star}} \sum_{\psi_3' \hspace{-.25cm} \pmod{q_3'}} \sum_{\psi_2' \hspace{-.25cm} \pmod{q_2'}} \int_{-T}^{T} \left|L\left(\frac{1}{2} + it,\psi_2 \psi_2' \psi_3'\right)\right|^2 \, dt.\]
Since $\psi_2$ is a primitive Dirichlet character and $p \mid \frac{q_2}{q_2'}$ whenever $p \mid q_2'$, the Dirichlet character $\psi_2 \psi_2' \psi_3'$ is a primitive Dirichlet character modulo $q_2 q_3'$. By applying the approximate functional equation for \cite[Theorem 5.3]{IK04}, a dyadic smooth partition of unity, and the Cauchy--Schwarz inequality applied to the dyadic sum, we therefore find that the left-hand side of \eqref{eqn:secondmomentcosetbounds} is
\begin{multline}
\label{eqn:cosetrearrange2}
\ll_{\e} (qT)^{\e} \sum_{q_3' \mid q_1} \sup_{1 \leq N \leq (q_2 q_3' T)^{1/2 +\e}} \\
\times \frac{1}{N} \sideset{}{^{\star}} \sum_{\psi_3' \hspace{-.25cm} \pmod{q_3'}} \sum_{\psi_2' \hspace{-.25cm} \pmod{q_2'}} \int_{-\infty}^{\infty} \Omega_0\left(\frac{t}{T}\right) \left|\sum_{n = 1}^{\infty} \Omega_1\left(\frac{n}{N}\right) \psi_2(n) \psi_2'(n) \psi_3'(n) n^{-it}\right|^2 \, dt,
\end{multline}
where $\Omega_0,\Omega_1$ are as in \eqref{eqn:secondmomentcosetdesired} (cf.\ \cite[Section 5]{PY23}).

By positivity, we may enlarge the sum over primitive Dirichlet characters $\psi_3'$ modulo $q_3'$ in \eqref{eqn:cosetrearrange2} to a sum over all Dirichlet characters modulo $q_3'$ (which incurs no penalties from the point of view of proving upper bounds), then make the change of variables $\psi_3' \mapsto \psi_3 \psi_3'$, where $\psi_3$ is some primitive Dirichlet character modulo $q_3'$. We write $\widetilde{q} \coloneqq q_2 q_3'$ and $\widetilde{\psi} \coloneqq \psi_2 \psi_3$, which is a primitive Dirichlet character modulo $\widetilde{q}$. We additionally write $\widetilde{q}' \coloneqq q_2' q_3'$ and replace the double sum over Dirichlet characters $\psi_2'$ modulo $q_2'$ and $\psi_3'$ modulo $q_3'$ with a single sum over Dirichlet characters $\widetilde{\psi}' \coloneqq \psi_2' \psi_3'$ modulo $\widetilde{q}'$. We thereby deduce that the second line of \eqref{eqn:cosetrearrange2} is bounded from above by
\[\frac{1}{N} \sum_{\widetilde{\psi}' \hspace{-.25cm} \pmod{\widetilde{q}'}} \int_{-\infty}^{\infty} \Omega_0\left(\frac{t}{T}\right) \left|\sum_{n = 1}^{\infty} \Omega_1\left(\frac{n}{N}\right) \widetilde{\psi}(n) \widetilde{\psi}'(n) n^{-it}\right|^2 \, dt.\]
Finally, we may enlarge the sum over Dirichlet characters modulo $\widetilde{q}'$ to run over Dirichlet characters modulo $4\widetilde{q}'$ by positivity. At this point, we may invoke the bounds \eqref{eqn:secondmomentcosetdesired} with $q,q',\psi,\psi'$ replaced by $\widetilde{q},\widetilde{q}',\widetilde{\psi},\widetilde{\psi}'$ respectively. This yields the desired result.
\end{proof}

With this result in hand, we are now able to prove \hyperref[prop:secondmomenthybrid]{Proposition \ref*{prop:secondmomenthybrid}}.

\begin{proof}[Proof of {\hyperref[prop:secondmomenthybrid]{Proposition \ref*{prop:secondmomenthybrid}}}]
We use \hyperref[cor:Zbounds]{Corollary \ref*{cor:Zbounds}} to bound $\ZZ_{\chi}(\psi_1 \psi_3' \psi_{0(q_2/q_3')};t)$. In this way, we see that the left-hand side of \eqref{eqn:secondmomenthybrid} is
\[\ll_{\e} \frac{q^3}{q_1 {q_3'}^3} (qT)^{\e} \sum_{\psi_1 \hspace{-.25cm} \pmod{q_1}} \prod_{p^{\beta} \parallel q_1} \left(\delta_{\psi_{p^{\beta}},\star} \frac{|g(\chi_{p^{\beta}},\psi_{p^{\beta}})|}{p^{\beta}} + 1\right)^2 \sideset{}{^{\star}} \sum_{\psi_3' \hspace{-.25cm} \pmod{q_3'}} \int_{-T}^{T} \left|L\left(\frac{1}{2} + it,\psi_1 \psi_3'\right)\right|^2 \, dt.\]
By positivity, we may extend the sum over primitive characters $\psi_3'$ modulo $q_3'$ to include imprimitive characters. To proceed further, we break up the sum over characters $\psi_1$ modulo $q_1$ dependent on the size of the product over $p^{\beta} \parallel q_1$. In doing so, we require control over the size of $g(\chi_{p^{\beta}},\psi_{p^{\beta}})$.

If $p = 2$, then $g(\chi_{p^{\beta}},\psi_{p^{\beta}})$ trivially vanishes, as observed in \cite[Remark 3.2]{PY23}. If $p$ is odd and $\beta = 1$, we have that $|g(\chi_p,\psi_p)| \leq 3p$ by \cite[Theorem 6.9]{PY20} (cf.\ \cite[Corollary 1.1]{Xi23}). If $p$ is odd and $\beta \geq 2$, the size of $g(\chi_{p^{\beta}},\psi_{p^{\beta}})$ is controlled by a certain invariant $\Delta(\chi_{p^{\beta}},\psi_{p^{\beta}}) \in \Z / p^{\beta - 1} \mathbb{Z}$ defined in \cite[Theorems 3.3 and 3.4]{PY23}. For $\alpha \in \{0,\ldots,\beta - 1\}$, we then define
\[m_{\chi_{p^{\beta}}}(\alpha) \coloneqq \inf\left\{m \in \frac{1}{2} \Z : \max_{\substack{\psi_{p^{\beta}} \hspace{-.25cm} \pmod{p^{\beta}} \text{ primitive} \\ v_p(\Delta(\chi_{p^{\beta}},\psi_{p^{\beta}})) = \alpha}} \frac{|g(\chi_{p^{\beta}},\psi_{p^{\beta}})|}{p^{\beta}} \leq 3 p^m\right\}.\]
We set $\widetilde{q_1} \coloneqq \prod_{p \mid q_1} p$. For $a \mid \frac{q_1}{\widetilde{q_1}}$, we let $\alpha$ be such that $p^{\alpha} \parallel a$ and let
\begin{equation}
\label{eqn:Mchiaq1defeq}
M_{\chi}(a,q_1) \coloneqq \prod_{p^{\beta} \parallel q_1} p^{m_{\chi_{p^{\beta}}}(\alpha)}.
\end{equation}
For each Dirichlet character $\psi_1$ modulo $q_1$, let $\Delta(\psi_1)$ be the unique positive integer in $\{1,\ldots,q_1/\widetilde{q_1}\}$ for which $\Delta(\psi_1) \equiv \Delta(\chi_{p^{\beta}},\psi_{p^{\beta}}) \pmod{p^{\beta - 1}}$ for each $p^{\beta} \parallel q_1$. Since $v_p(\Delta(\chi_{p^{\beta}},\psi_{p^{\beta}})) = \alpha$, we have that $(\Delta(\psi_1),\frac{q_1}{\widetilde{q_1}}) = a$.

From this, we see that the left-hand side of \eqref{eqn:secondmomenthybrid} is
\[\ll_{\e} \frac{q^3}{{q_3'}^3} (qT)^{\e} \sum_{a \mid \frac{q_1}{\widetilde{q_1}}} \frac{M_{\chi}(a,q_1)^2}{q_1} \sum_{\substack{\psi_1 \hspace{-.25cm} \pmod{q_1} \\ (\Delta(\psi_1),\frac{q_1}{\widetilde{q_1}}) = a}} \sum_{\psi_3' \hspace{-.25cm} \pmod{q_3'}} \int_{-T}^{T} \left|L\left(\frac{1}{2} + it,\psi_1 \psi_3'\right)\right|^2 \, dt.\]
By positivity, we may extend the sum over characters $\psi_1$ modulo $q_1$ for which $(\Delta(\psi_1),\frac{q_1}{\widetilde{q_1}}) = a$ to additionally include all characters $\psi_1$ modulo $q_1$ for which $\Delta(\psi_1) \equiv 0 \pmod{a}$. We then break up this sum over cosets. We let $G \coloneqq \{\psi_1 \pmod{q_1}\}$ denote the group of Dirichlet characters modulo $q_1$ and let $H_a$ be the subgroup $\{\psi_1 \pmod{\frac{q_1}{a}}\}$. If $\psi_1,\psi_1'$ are two characters modulo $q_1$ for which $\Delta(\psi_1) \equiv \Delta(\psi_1') \pmod{a}$, then they lie in the same $H_a$-coset by \cite[Lemma 2.1]{PY23}. It follows that the left-hand side of \eqref{eqn:secondmomenthybrid}
\[\ll_{\e} \frac{q^3}{{q_3'}^3} (qT)^{\e} \sum_{a \mid \frac{q_1}{\widetilde{q_1}}} \frac{M_{\chi}(a,q_1)^2}{q_1} \sum_{\substack{\psi_1 \in G / H_a \\ \Delta(\psi_1) \equiv 0 \hspace{-.25cm} \pmod{a}}} \sum_{\psi' \hspace{-.25cm} \pmod{\frac{q_1}{a} q_3'}} \int_{-T}^{T} \left|L\left(\frac{1}{2} + it,\psi_1 \psi'\right)\right|^2 \, dt.\]
To bound the sum over characters $\psi'$ modulo $\frac{q_1}{a} q_3'$ and integral over $t \in [-T,T]$, we apply \hyperref[thm:secondmomentcoset]{Theorem \ref*{thm:secondmomentcoset}} with $q$ replaced by $q_1 q_3'$, $\psi$ replaced by $\psi_1 \psi_{0(q_3')}$, and $q'$ replaced by $\frac{q_1}{a} q_3'$. Since there are at most $2^{\omega(q_1)}$ characters $\psi_1 \in G/H_a$ satisfying $\Delta(\psi_1) \equiv 0 \pmod{a}$, as discussed in \cite[Proof of Lemma 4.2]{PY23}, we see that the left-hand side of \eqref{eqn:secondmomenthybrid} is
\[\ll_{\e} \frac{q^3}{{q_3'}^2} (qT)^{\e} \left(T \sum_{a \mid \frac{q_1}{\widetilde{q_1}}} \frac{M_{\chi}(a,q_1)^2}{a} + \frac{1}{q_3'} \sum_{a \mid \frac{q_1}{\widetilde{q_1}}} \frac{M_{\chi}(a,q_1)^2 a^{1/2}}{q_1}\right).\]
Using the multiplicativity of $M_{\chi}(a,q_1)$ as a function of $a$ via \eqref{eqn:Mchiaq1defeq}, this is
\[\ll_{\e} \frac{q^3}{{q_3'}^2} (qT)^{\e} \left(T \prod_{p^{\beta} \parallel q_1} \sum_{\alpha = 0}^{\beta - 1} p^{2m_{\chi_{p^{\beta}}}(\alpha) - \alpha} + \frac{1}{q_3'} \prod_{p^{\beta} \parallel q_1} \sum_{\alpha = 0}^{\beta - 1} p^{2m_{\chi_{p^{\beta}}}(\alpha) + \frac{\alpha}{2} - \beta}\right).\]
Thus the result follows from the bounds
\[m_{\chi_{p^{\beta}}}(\alpha) \leq \min\left\{\frac{\alpha + \beta}{2}, \beta - \frac{\alpha}{4}\right\},\]
which are proven (in a more explicit form) in \cite[Proof of Lemma 4.2]{PY23} based on \cite[Lemma 3.1]{PY23}.
\end{proof}

\section{Proofs of \texorpdfstring{\hyperref[thm:subconvexbounds]{Theorems \ref*{thm:subconvexbounds}}}{Theorems \ref{thm:subconvexbounds}} and \ref{thm:firstmomentbounds}}
\label{sect:prooffirstmoment}

We now have all of our tools in place in order to complete the proof of \hyperref[thm:firstmomentbounds]{Theorem \ref*{thm:firstmomentbounds}}.

\begin{proof}[Proof of {\hyperref[thm:firstmomentbounds]{Theorem \ref*{thm:firstmomentbounds}}}]
We first deal with the case where $T - 2U > 2C$, where $C$ is the large fixed positive constant as in \hyperref[sect:testfunctionsII]{Section \ref*{sect:testfunctionsII}}. We take $(h,h^{\hol})$ as in \eqref{eqn:testfunctionchoice}. By \hyperref[lem:testfunction]{Lemma \ref*{lem:testfunction}}, $h(t)$ is nonnegative on $\R \cup i(-1/2,1/2)$ and $h^{\hol}(k)$ is nonnegative on $2\N$, while $h(t) \asymp 1$ if $T - U \leq |t| \leq T + U$ and $h^{\hol}(k) \asymp 1$ if $T - U \leq k \leq T + U$. Moreover, we have the lower bounds $\alpha(q,q',q_{\overline{\chi_1}^2}) \gg_{\e} q^{-\e}$ and 
\begin{align*}
\frac{L_q(1,\ad f)}{L_q\left(\frac{1}{2},F \otimes f \otimes \chi_1\right)} & \gg_{\e} q^{-\e} \quad \text{for $f \in \BB_0^{\ast}(q',\overline{\chi_1}^2)$ or $f \in \BB_{\hol}^{\ast}(q',\overline{\chi_1}^2)$},	\\
\left|\frac{L_q(1 + 2it,\psi_1\overline{\psi_2})}{L_q\left(\frac{1}{2} + it,F \otimes \psi_1 \chi_1\right)}\right|^2 & \gg_{\e} q^{-\e} \quad \text{for $\psi_1,\psi_2$ modulo $q$}.
\end{align*}
This relies on the fact that $L_p(1/2,F \otimes f \otimes \chi_1)$ is positive for all primes $p$, which in turn uses the fact the best-known bound towards the generalised Ramanujan conjecture is strictly less than $1/6$ \cite{Kim03}. From this and the nonnegativity of $L(1/2,F \otimes f \otimes \chi_1)$ \cite[Theorem 1.1]{Lap03}, we see that the left-hand side of \eqref{eqn:firstmomentbounds} is bounded by a constant multiple dependent on $F$ and $\e$ of the product of $(qT)^{\e}$ and the moment \eqref{eqn:centralLHS}. Thus to prove \hyperref[thm:firstmomentbounds]{Theorem \ref*{thm:firstmomentbounds}}, the $\GL_3 \times \GL_2 \leftrightsquigarrow \GL_4 \times \GL_1$ spectral reciprocity formula given in \hyperref[thm:central]{Theorem \ref*{thm:central}} shows that it suffices to prove that the primary main term \eqref{eqn:centralmain} and the secondary main term \eqref{eqn:centralsecondmain} are both $O_{F,\e} (q TU (qT)^{\e})$ and that the dual moment \eqref{eqn:centraldual} is $O_{F,\e}((q_1 T)^{5/4} q_2^{1/2} U^{-1/4} (qT)^{\e})$.

The bound $O_{F,\e} (q TU (qT)^{\e})$ for \eqref{eqn:centralmain} and \eqref{eqn:centralsecondmain} follows from \eqref{eqn:maintermbounds}. To bound \eqref{eqn:centraldual}, we divide the integral over $t \in \R$ into the ranges $|t| \leq 1$ and $2^{n - 1} \leq |t| \leq 2^n$ for each positive integer $n \in \N$. We then apply the triangle inequality and bound $\HH_{\mu_F}^{\pm}(t)$ pointwise via the bounds \eqref{eqn:HHpmbounds}. The ensuing expression is then bounded by means of \hyperref[prop:centraldualbounds]{Proposition \ref*{prop:centraldualbounds}}, which yields the desired estimate.

Finally, we deal with the case $T - 2U \leq 2C$, so that in particular $T$ and $U$ are bounded. Here we simply take $(h,h^{\hol})$ as in \cite[(3.24)]{BK19}, namely
\begin{align*}
h(t) & \coloneqq \frac{b!}{2^b} \prod_{j = 0}^{b} \frac{1}{t^2 + \left(\frac{a + b}{2} - j\right)^2},	\\
h^{\hol}(k) & \coloneqq \frac{b!}{2^b} \prod_{j = 0}^{b} \frac{1}{\frac{(i(k - 1))^2}{4} + \left(\frac{a + b}{2} - j\right)^2} + \delta_{k > a - b} c(a,b) k^{-2b - 1},
\end{align*}
so that by \cite[(3.25)]{BK19}, the associated transform $H$ as in \eqref{eqn:HdefeqKscr} is given by
\[H(x) = \frac{i^{b - 1} (4\pi)^{-b}}{2\pi} \JJ_{a + 1}^{\hol}(x) x^{-b} + \sum_{\substack{k > a - b \\ k \equiv 0 \hspace{-.25cm} \pmod{2}}} \frac{c(a,b)}{k^{2b + 1}} \frac{k - 1}{2\pi^2} \JJ_k^{\hol}(x).\]
Here $a,b \in \N$ are fixed positive integers satisfying $a \equiv b \pmod{2}$, $a - b > \max\{T + U,5\}$, and $b > 3$, while $c(a,b)$ is a positive constant such that $h^{\hol}(k) > 0$ for all $k \in 2\N$. This tuple of test functions is admissible of type $(a - b,2b - 2)$ and is such that both $h$ and $h^{\hol}$ are always nonnegative and additionally $h(t) \gg 1$ and $h^{\hol}(k) \gg 1$ if $|t|,k \leq a - b$. We then proceed by the same argument as in the case for which $T - 2U > 2C$ except that we appeal to the bounds \eqref{eqn:HHpmwzdecay} for $\HH_{\mu_F}^{\pm}(t)$ in place of the bounds \eqref{eqn:HHpmbounds}.
\end{proof}

\hyperref[thm:subconvexbounds]{Theorem \ref*{thm:subconvexbounds}} then follows directly from \hyperref[thm:firstmomentbounds]{Theorem \ref*{thm:firstmomentbounds}}.

\begin{proof}[Proof of {\hyperref[thm:subconvexbounds]{Theorem \ref*{thm:subconvexbounds}}}]
We use \hyperref[thm:firstmomentbounds]{Theorem \ref*{thm:firstmomentbounds}} with $q_2 = (q/q_1)^2$ and take
\[U = \begin{dcases*}
1 & if $q_1 \leq q^{4/5} T^{-1/5}$,	\\
q^{-4/5} q_1 T^{1/5} & if $q^{4/5} T^{-1/5} \leq q_1 \leq q^{4/5} T^{4/5}$,	\\
T & if $q_1 \geq q^{4/5} T^{4/5}$.
\end{dcases*}\]
If $f \in \BB_0^{\ast}(q^2,1)$ is such that $f \otimes \overline{\chi}$ has level dividing $q$, then upon writing $\chi = \chi_1 \chi_2$, where $\chi_1$ and $\chi_2$ are primitive characters modulo $q_1$ and $q_2^{1/2}$ respectively, we must have that $f \otimes \overline{\chi_1} \in \BB_0^{\ast}(q_1' q_2,\overline{\chi_1}^2)$ for some $q_1' \mid q_1$. From this, \hyperref[thmno:subconvexbounds2]{Theorem \ref*{thm:subconvexbounds} \ref*{thmno:subconvexbounds2}} follows by dropping all but one term via positivity, since $L(1/2,F \otimes f \otimes \chi)$ is nonnegative \cite[Theorem 1.1]{Lap03}, together with the upper bound $L(1,\ad f) \ll_{\e} (q(|t_f| + 1))^{\e}$ \cite[Corollary 1]{Li10}. \hyperref[thmno:subconvexbounds1]{Theorem \ref*{thm:subconvexbounds} \ref*{thmno:subconvexbounds1}} and \ref{thmno:subconvexbounds3} follow analogously\footnote{There is an additional subtlety in deducing a subconvex bound for $L(1/2 + it,F \otimes \chi)$ when $t$ is near $0$ and $\chi$ is a real character, since if $\psi_1 = \psi_2$, then $|L(1/2 + it, F \otimes \psi_1 \chi_1)/L(1 + 2it,\psi_1 \overline{\psi_2})|^2$ has a zero of order two at $t = 0$. One can circumvent this obstacle via an application of H\"{o}lder's inequality, as in \cite[pp.~1404--1405]{Blo12}.}.
\end{proof}

\section{The Eisenstein Case}
\label{sect:Eisenstein}

\subsection{Eisenstein Analogues of \texorpdfstring{\hyperref[thm:subconvexbounds]{Theorems \ref*{thm:subconvexbounds}}}{Theorems \ref{thm:subconvexbounds}} and \ref{thm:firstmomentbounds}}

Our method also extends, with some alterations, to the case where the Hecke--Maa\ss{} cusp form $F$ for $\SL_3(\Z)$ is replaced by an Eisenstein series for $\SL_3(\Z)$. When such an Eisenstein series is associated to the minimal parabolic with trivial spectral parameters, we have the following analogue of \hyperref[thm:firstmomentbounds]{Theorem \ref*{thm:firstmomentbounds}}.

\begin{theorem}
\label{thm:firstmomentboundsEis}
Let $q_1, q_2$ be coprime positive integers. Let $\chi_1$ be a primitive Dirichlet character of conductor $q_1$. Then for $T \geq 1$ and $1 \leq U \leq T$, we have that
\begin{equation}
\label{eqn:firstmomentboundsEis}
\begin{drcases*}
\sum_{\substack{q' \mid q_1 q_2 \\ q' \equiv 0 \hspace{-.25cm} \pmod{q_{\overline{\chi_1}^2}}}} \sum_{\substack{\psi_1,\psi_2 \hspace{-.25cm} \pmod{q_1 q_2} \\ \psi_1 \psi_2 = \overline{\chi_1}^2 \\ q_{\psi_1} q_{\psi_2} = q'}} \, \int\limits_{T - U \leq |t| \leq T + U} \left|\frac{L\left(\frac{1}{2} + it,\psi_1 \chi_1\right)^3}{L(1 + 2it,\psi_1\overline{\psi_2})}\right|^2 \, dt \\
\sum_{\substack{q' \mid q_1 q_2 \\ q' \equiv 0 \hspace{-.25cm} \pmod{q_{\overline{\chi_1}^2}}}} \sum_{\substack{f \in \BB_0^{\ast}(q',\overline{\chi_1}^2) \\ T - U \leq t_f \leq T + U}} \frac{L\left(\frac{1}{2},f \otimes \chi_1\right)^3}{L(1,\ad f)} \\
\sum_{\substack{q' \mid q_1 q_2 \\ q' \equiv 0 \hspace{-.25cm} \pmod{q_{\overline{\chi_1}^2}}}} \sum_{\substack{f \in \BB_{\hol}^{\ast}(q',\overline{\chi_1}^2) \\ T - U \leq k_f \leq T + U}} \frac{L\left(\frac{1}{2},f \otimes \chi_1\right)^3}{L(1,\ad f)} \\
\end{drcases*} \ll_{\e} q_1 q_2 TU (q_1 q_2 T)^{\e}.
\end{equation}
\end{theorem}

\hyperref[thm:firstmomentboundsEis]{Theorem \ref*{thm:firstmomentboundsEis}} recovers \cite[Theorems 1.2 and 1.3]{PY23} upon taking $q_2 = 1$\footnote{On the other hand, \hyperref[thm:firstmomentboundsEis]{Theorem \ref*{thm:firstmomentboundsEis}} does not supersede the works \cite{PY20,PY23} of Petrow and Young, since the proof of \hyperref[thm:firstmomentboundsEis]{Theorem \ref*{thm:firstmomentboundsEis}} is contingent upon \cite[Theorem 1.4]{PY23}; see \hyperref[prop:Eisensteincentraldualbounds]{Proposition \ref*{prop:Eisensteincentraldualbounds}} below.}. It additionally recovers \cite[Theorem 4.1]{AW23} upon taking $\chi_1$ to be quadratic and recovers \cite[Theorem 1]{PY19} upon taking $k_f$ to be fixed, $q_2$ squarefree, and $\chi_1$ quadratic.

As an immediate consequence of \hyperref[thm:firstmomentboundsEis]{Theorem \ref*{thm:firstmomentboundsEis}}, we obtain the following bounds for individual $L$-functions, which parallel the bounds in \hyperref[thm:subconvexbounds]{Theorem \ref*{thm:subconvexbounds}}.

\begin{theorem}
\label{thm:subconvexboundsEis}
Let $q_1,q_2$ be positive coprime integers. Let $\chi_1$ be a primitive Dirichlet character of conductor $q_1$. Let $q'$ be a divisor of $q_1$ for which $q' \equiv 0 \pmod{q_{\overline{\chi_1}^2}}$.
\begin{enumerate}[leftmargin=*,label=\textup{(\arabic*)}]
\item\label{thmno:subconvexboundsEis1} We have that
\[L\left(\frac{1}{2} + it, \chi_1\right) \ll_{\e} (q_1(|t| + 1))^{\frac{1}{6} + \e}.\]
\item\label{thmno:subconvexboundsEis2} Let $f$ be a Hecke--Maa\ss{} newform of weight $0$, level $q' q_2$, nebentypus $\overline{\chi_1}^2$, and Laplacian eigenvalue $\frac{1}{4} + t_f^2$. We have that
\[L\left(\frac{1}{2},f \otimes \chi_1\right) \ll_{\e} (q_1 q_2 (|t_f| + 1))^{\frac{1}{3} + \e}.\]
\item\label{thmno:subconvexboundsEis3} Let $f$ be a holomorphic Hecke newform of even weight $k_f$, level $q' q_2$, and nebentypus $\overline{\chi_1}^2$. We have that
\[L\left(\frac{1}{2},f \otimes \chi_1\right) \ll_{\e} (q_1 q_2 k_f)^{\frac{1}{3} + \e}.\]
\end{enumerate}
\end{theorem}

We note that the convexity bound is $O_{\e}((q_1(|t| + 1))^{1/4 + \e})$ in the first case, $O_{\e}((q_1 (|t_f| + 1))^{1/2 + \e} q_2^{1/4 + \e})$ in the second case, and $O_{\e}((q_1 k_g)^{1/2 + \e} q_2^{1/4 + \e})$ in the third case. The bounds in \hyperref[thm:subconvexboundsEis]{Theorem \ref*{thm:subconvexboundsEis}} imply hybrid subconvexity simultaneously in the $q_1$ and $t$, $t_f$, or $k_f$ aspects, where they are of Weyl-strength; however, they fall shy of the convexity bound in the $q_2$ aspect.

When $\chi_1$ is quadratic, hybrid bounds of this form have applications on progress towards the Ramanujan conjecture for half-integral weight automorphic forms; see, for example, \cite[Theorem 2]{PY19}. They also have applications towards proving small-scale equidistribution of geometric invariants associated to quadratic fields, such as Heegner points \cite[Section 2]{You17} and lattice points on the sphere \cite[Theorem 1.5]{HR22}, proving an effective rate of equidistribution of the reduction of CM elliptic curves \cite[Theorem 1.1]{LMY15}, and proving uniform bounds for the error term in the Hardy--Ramanujan--Rademacher formula for the partition function \cite[Theorem 1.1]{AW23}.

\subsection{\texorpdfstring{$\mathrm{GL}_3 \times \mathrm{GL}_2 \leftrightsquigarrow \mathrm{GL}_4 \times \mathrm{GL}_1$}{GL\textthreeinferior \textmultiply GL\texttwoinferior \textleftrightarrow GL\textfourinferior \textmultiply GL\textoneinferior} Spectral Reciprocity}

\hyperref[thm:firstmomentboundsEis]{Theorem \ref*{thm:firstmomentboundsEis}} is a consequence of a $\GL_3 \times \GL_2 \leftrightsquigarrow \GL_4 \times \GL_1$ spectral reciprocity identity akin to \hyperref[thm:central]{Theorem \ref*{thm:central}}. We do not give a proof of this spectral reciprocity identity but merely indicate the key modifications needed. The proof of this spectral reciprocity identity follows the same strategy except that $F$ is replaced by a minimal parabolic Eisenstein series associated to parameters $\mu = (\mu_1,\mu_2,\mu_3) \in \C^3$. Initially, we assume that none of these parameters are equal or are $0$; we also assume that each lies in a small neighbourhood of $0$. Eventually, we analytically continue to the central value, namely $\mu_1 = \mu_2 = \mu_3 = 0$.

Replacing $F$ by a minimal parabolic Eisenstein series has the effect of replacing the Hecke eigenvalues $A_F(1,n)$ and $A_F(m,1)$ by shifted triple divisor functions
\[A_F(1,n) = \sum_{n_1 n_2 n_3 = n} n_1^{-\mu_1} n_2^{-\mu_2} n_3^{-\mu_3}, \qquad A_F(m,1) = \sum_{m_1 m_2 m_3 = m} m_1^{\mu_1} m_2^{\mu_2} m_3^{\mu_3}.\]
In turn, this has the effect of replacing the $L$-functions $L^q(w,F \otimes f \otimes \chi_1)$ and $L^q(w + it,F \otimes \psi_1 \chi_1)$ appearing in \eqref{eqn:absoluteLHS} with $\prod_{j = 1}^{3} L^q(w + \mu_j,f \otimes \chi_1)$ and $\prod_{j = 1}^{3} L(w + it + \mu_j,\psi_1 \chi_1)$. Similarly, $L^q(2w,\widetilde{F})$ is replaced by $\prod_{j = 1}^{3} \zeta^q(2w - \mu_j)$ in \eqref{eqn:absolutemain} and $L(1/2 + z,\widetilde{F} \otimes \psi)$ is replaced by $\prod_{j = 1}^{3} L(1/2 + z - \mu_j,\psi)$ in \eqref{eqn:absolutedual}. This mildly alters the appearance of the moment \eqref{eqn:centralLHS} and the dual moment \eqref{eqn:centraldual}. The main terms \eqref{eqn:centralmain} and \eqref{eqn:centralsecondmain} are significantly altered, however: there are several additional main terms. We discuss below the shapes of the moment term, the dual moment term, and the main terms for this spectral reciprocity identity, as well as how these are treated with regards to the proof of \hyperref[thm:firstmomentboundsEis]{Theorem \ref*{thm:firstmomentboundsEis}}.

\subsubsection{The Moment}

After analytically continuing to $w = 1/2$ and $\mu_1 = \mu_2 = \mu_3 = 0$, the moment \eqref{eqn:centralLHS} is replaced by
\begin{multline*}
\sum_{\substack{q' \mid q \\ q' \equiv 0 \hspace{-.25cm} \pmod{q_{\overline{\chi_1}^2}}}} \alpha(q,q',q_{\overline{\chi_1}^2}) \sum_{f \in \BB_0^{\ast}(q',\overline{\chi_1}^2)} \frac{L^q\left(\frac{1}{2},f \otimes \chi_1\right)^3}{L^q(1,\ad f)} h(t_f)	\\
+ \sum_{\substack{q' \mid q \\ q' \equiv 0 \hspace{-.25cm} \pmod{q_{\overline{\chi_1}^2}}}} \alpha(q,q',q_{\overline{\chi_1}^2}) \sum_{\substack{\psi_1,\psi_2 \hspace{-.25cm} \pmod{q} \\ \psi_1 \psi_2 = \overline{\chi_1}^2 \\ q_{\psi_1} q_{\psi_2} = q'}} \frac{1}{2\pi} \int_{-\infty}^{\infty} \left|\frac{L^q\left(\frac{1}{2} + it,\psi_1 \chi_1\right)^3}{L^q(1 + 2it,\psi_1\overline{\psi_2})}\right|^2 h(t) \, dt	\\
+ \sum_{\substack{q' \mid q \\ q' \equiv 0 \hspace{-.25cm} \pmod{q_{\overline{\chi_1}^2}}}} \alpha(q,q',q_{\overline{\chi_1}^2}) \sum_{f \in \BB_{\hol}^{\ast}(q',\overline{\chi_1}^2)} \frac{L^q\left(\frac{1}{2},f \otimes \chi_1\right)^3}{L^q(1,\ad f)} h^{\hol}(k_f).
\end{multline*}
As in the proof of \hyperref[thm:firstmomentbounds]{Theorem \ref*{thm:firstmomentbounds}}, this provides a lower bound for the left-hand side of \eqref{eqn:firstmomentboundsEis} via an appropriate choice of tuple of test functions $(h,h^{\hol})$, which relies crucially on the nonnegativity of the central $L$-value $L(1/2,f \otimes \chi_1)$ \cite[Theorem]{Guo96}.

\subsubsection{The Dual Moment}

Similarly, after analytically continuing to $w = 1/2$ and $\mu_1 = \mu_2 = \mu_3 = 0$, the dual moment \eqref{eqn:centraldual} is replaced by
\[\frac{1}{\varphi(q)} \sum_{\psi \hspace{-.25cm} \pmod{q}} \frac{1}{2\pi} \int_{-\infty}^{\infty} L\left(\frac{1}{2} + it,\psi\right)^3 L\left(\frac{1}{2} - it,\overline{\psi}\right) \ZZ_{\chi}\left(\psi;\frac{1}{2},it\right) \sum_{\pm} \psi(\mp 1) \HH_{\mu_F}^{\pm}\left(\frac{1}{2},it\right) \, dt.\]
The proof of \hyperref[thm:firstmomentboundsEis]{Theorem \ref*{thm:firstmomentboundsEis}} requires us to bound this dual moment. As in the proof of \hyperref[thm:firstmomentbounds]{Theorem \ref*{thm:firstmomentbounds}}, we begin by breaking up the integral over $t \in \R$ into the ranges $|t| \leq 1$ and $2^{n - 1} \leq |t| \leq 2^n$ for each positive integer $n \in \N$. At this point, however, our treatment of this dual moment deviates from the approach given in \hyperref[sect:secondmomentbounds]{Section \ref*{sect:secondmomentbounds}}. In particular, we do \emph{not} use the Cauchy--Schwarz inequality coupled with second moment bounds in order to obtain the bound \eqref{eqn:centraldualbounds}. Instead, we have the following result.

\begin{proposition}
\label{prop:Eisensteincentraldualbounds}
Let $F$ be a Hecke--Maa\ss{} cusp form for $\SL_3(\Z)$. Let $q = q_1 q_2$ be a positive integer with $(q_1,q_2) = 1$. Let $\chi_1$ be a primitive Dirichlet character of conductor $q_1$, and set $\chi \coloneqq \chi_1 \chi_{0(q_2)}$. For $T \geq 1$, we have that
\[\frac{1}{\varphi(q)} \sum_{\psi \hspace{-.25cm} \pmod{q}} \int_{-T}^{T} \left|L\left(\frac{1}{2} + it, \psi\right)\right|^4 \left|\ZZ_{\chi}(\psi;t)\right| \, dt \ll_{F,\e} q_1 q_2^{\frac{1}{2}} T (qT)^{\e}.\]
\end{proposition}

The proof is via the same strategy as that of \hyperref[prop:secondmomenthybrid]{Proposition \ref*{prop:secondmomenthybrid}} except that instead of invoking bounds for the \emph{second} moment of Dirichlet $L$-functions along cosets, namely \hyperref[thm:secondmomentcoset]{Theorem \ref*{thm:secondmomentcoset}}, we invoke the much stronger bounds for the \emph{fourth} moment of Dirichlet $L$-functions along cosets due to Petrow and Young \cite[Theorem 1.4]{PY23}. The bounds given in \hyperref[prop:Eisensteincentraldualbounds]{Proposition \ref*{prop:Eisensteincentraldualbounds}} are \emph{stronger} than those \hyperref[prop:centraldualbounds]{Proposition \ref*{prop:centraldualbounds}}; the latter is lossy due to the fact that the bounds obtained in \hyperref[prop:secondmomentFlargesieve]{Proposition \ref*{prop:secondmomentFlargesieve}} for the second moment of $L(1/2 + it,F \otimes \psi)$ are suboptimal. It is for this reason that the upper bound in \eqref{eqn:firstmomentboundsEis} is stronger than that in \eqref{eqn:firstmomentbounds}.

\subsubsection{The Main Terms}

It remains to discuss the main terms. There are several additional main terms that arise.

\begin{enumerate}[leftmargin=*,label=\textup{(\arabic*)}]
\item Additional main terms arise from residues via the Vorono\u{\i} summation formula. A step of the proof of \hyperref[prop:absconv]{Proposition \ref*{prop:absconv}} involves the expression \eqref{eqn:Kloostermanintegrand} given by
\[q^{s - 1} \sum_{\substack{c,\ell = 1 \\ (\ell,q) = 1}}^{\infty} \frac{c^{s - 2}}{\ell^{2w}} \sum_{\substack{a \in \Z/cq\Z \\ (a,q) = 1}} \chi_1(a) S_{\overline{\chi_1}^2}(1,a;cq) \sum_{c_1 \mid cq} \sum_{b \in (\Z/c_1\Z)^{\times}} e\left(\frac{a\overline{b}}{c_1}\right) \Phi_F\left(c_1,-b,\ell;\frac{s}{2} + w\right)\]
that appears in the integrand of \eqref{eqn:Kloostermanterm}. In the proof of \hyperref[prop:absconv]{Proposition \ref*{prop:absconv}}, we shifted the contour of integration to the left, as the expression above is holomorphic in $s$ when $F$ is cuspidal. When $F$ is a minimal parabolic Eisenstein series, on the other hand, this shifting of the contour picks up residues at the poles of $\Phi_F(c_1,-b,\ell;s/2 + w)$. When the spectral parameters $(\mu_1,\mu_2,\mu_3)$ are distinct, there are three simple poles, which occur at $s = 2 - 2w - 2\mu_j$. The residues can be determined via work of Fazzari \cite{Faz24} and give three additional main terms.
\item When $q_1 = 1$, there are three additional main terms arising from the continuous spectrum via the analytic continuation to $w = 1/2$ of the $\GL_2$ moment \eqref{eqn:absoluteLHS}. While the first and third terms in \eqref{eqn:absoluteLHS} analytically continue to $w = 1/2$ with no complications, the second term yields additional degenerate terms if $\psi_1 = \psi_2 = \overline{\chi_1}$, for then the integrand in \eqref{eqn:absoluteLHS} has poles at $t = \pm i(w + \mu_j - 1)$. Note that this can only occur if $q_{\psi_1} = q_{\psi_2} = q_{\chi_1} = q_1$, so that $q' = q_1^2$; since $q' \mid q_1 q_2$ with $(q_1,q_2) = 1$, this can only occur if $q_1 = 1$, so that $\chi_1$ is the trivial character. So long as no two of the parameters $\mu_1,\mu_2,\mu_3$ are equal, the analytic continuation of these three degenerate terms to $w = 1/2$ is
\[-2 \alpha(q_2,1,1) \frac{\varphi(q_2)}{q_2} \sum_{j = 1}^{3} \frac{\prod_{\substack{k = 1 \\ k \neq j}}^{3} \zeta^{q_2}(1 - \mu_j + \mu_k) \prod_{k = 1}^{3} \zeta^{q_2}(\mu_j + \mu_k)}{\zeta^{q_2}(2 - 2\mu_j) \zeta^{q_2}(2\mu_j)} h\left(-i\left(\frac{1}{2} - \mu_j\right)\right).\]
\item The primary main term is essentially the same as that appearing in \eqref{eqn:absolutemain}, namely
\[q \prod_{j = 1}^{3} \zeta^q(2w - \mu_j) \left(\frac{1}{2\pi^2} \int_{-\infty}^{\infty} h(r) r \tanh\pi r \, dr + \sum_{\substack{k = 2 \\ k \equiv 0 \hspace{-.25cm} \pmod{2}}}^{\infty} \frac{k - 1}{2\pi^2} h^{\hol}(k)\right).\]
So long as each $\mu_j$ is nonzero, this extends holomorphically to $w = 1/2$.
\item The secondary main term is slightly different than that appearing in \eqref{eqn:centralsecondmain}. This secondary main term arises due to the pole at $z = 2w - 3/2$ of $L(2w - 1/2 - z,\overline{\psi})$ with $\psi = \psi_{0(q)}$ and is given by the sum of
\begin{multline*}
\frac{1}{q} \prod_{j = 1}^{3} \zeta(2w - 1 - \mu_j) \frac{\ZZ_{\chi}\left(\psi_{0(q)};w,2w - \frac{3}{2}\right)}{\prod_{j = 1}^{3} \zeta_q(2w - 1 - \mu_j)}	\\
\times \frac{1}{2\pi i} \int_{\sigma_3 - i\infty}^{\sigma_3 + i\infty} \widehat{H}(s) \sum_{\pm} \Gscr_{\mu_F}^{\pm}\left(1 - \frac{s}{2} - w\right) G^{\mp}\left(\frac{s}{2} + 3w - 2\right) \, ds,
\end{multline*}
where $2 - 6\Re(w) < \sigma_3 < 4 - 6\Re(w)$, and
\[\frac{2}{q} \prod_{j = 1}^{3} \zeta(2 - 2w + \mu_j) \frac{\ZZ_{\chi}\left(\psi_{0(q)};w,2w - \frac{3}{2}\right)}{\prod_{j = 1}^{3} \zeta_q(2w - 1 - \mu_j)} \widehat{H}(4 - 6w).\]
So long as each $\mu_j$ is nonzero, both of these terms extend holomorphically to $w = 1/2$. Note, however, that the first term need not vanish at $w = 1/2$, since $\prod_{j = 1}^{3} \zeta(-\mu_j)$ need not be zero, whereas the corresponding first term vanishes when $F$ is a selfdual Hecke--Maa\ss{} cusp form for $\SL_3(\Z)$ due to the fact that $L(0,F) = 0$.
\item There are additional secondary main terms that arise due to the poles at $z = 1/2 - \mu_j$ of $\prod_{j = 1}^{3} L(1/2 + z + \mu_j,\psi)$ with $\psi = \psi_{0(q)}$. The sum of the ensuing residues is
\begin{multline*}
\frac{1}{q} \sum_{j = 1}^{3} \prod_{\substack{k = 1 \\ k \neq j}}^{3} \zeta^q(1 + \mu_j - \mu_k) \zeta^q(2w - 1 - \mu_j) \ZZ_{\chi}\left(\psi_{0(q)};w,\frac{1}{2} + \mu_j\right)	\\
\times \sum_{\pm} \frac{1}{2\pi i} \int_{\sigma_2 - i\infty}^{\sigma_2 + i\infty} \widehat{H}(s) \Gscr_{\mu_F}^{\pm}\left(1 - \frac{s}{2} - w\right) G^{\mp}\left(\frac{s}{2} + w + \mu_j\right) \, ds,
\end{multline*}
where $-2\Re(w) < \sigma_2 < 2 - 2\Re(w)$. Again, these extend holomorphically to $w = 1/2$ provided the spectral parameters $(\mu_1,\mu_2,\mu_3)$ are distinct.
\end{enumerate}

The final step is to holomorphically extend the sum of these main terms to the value $(\mu_1,\mu_2,\mu_3) = (0,0,0)$ and subsequently bound this sum. While these additional main terms may \emph{individually} have singularities when one of $\mu_1,\mu_2,\mu_3$ is $0$ or when two of them are equal, the \emph{sum} of all of these additional main terms, which we denote by $\widetilde{h}_{\mu}$, extends holomorphically to $(\mu_1,\mu_2,\mu_3) = (0,0,0)$, since it is equal to a function that is holomorphic at that value, namely the difference of the $\GL_3 \times \GL_2$ moment and the $\GL_4 \times \GL_1$ moment.

To bound the sum of these main terms, we fix $\e > 0$ and set $\mu_j = jz$ with $|z| = \e$, so that by Cauchy's residue formula,
\[\widetilde{h}_{(0,0,0)} = \frac{1}{2\pi i} \oint_{|z| = \e} \frac{\widetilde{h}_{(z,2z,3z)}}{z} \, dz.\]
Thus it suffices to estimate each main term at $\mu = (z,2z,3z)$ with $|z| = \e$, and in every case we find that each main term is $O_{\e}(q_1 q_2 TU (q_1 q_2 T)^{\e})$.

\phantomsection
\addcontentsline{toc}{section}{Acknowledgements}
\hypersetup{bookmarksdepth=-1}

\subsection*{Acknowledgements}

The authors would like to thank Gergely Harcos, Ikuya Kaneko, and Matt Young for useful comments and discussions, as well as the anonymous referee for their thorough reading and constructive feedback. This project began while the third author was a member of the TAN group at \'{E}cole Polytechnique F\'{e}d\'{e}rale de Lausanne and the fourth author was a visiting scholar at Max-Planck Institut f\"{u}r Mathematik; they would like to thank these institutions for their hospitality.

\hypersetup{bookmarksdepth}

\end{document}